\documentclass[english]{article}
\usepackage[T1]{fontenc}
\usepackage[latin9]{luainputenc}
\usepackage{geometry}
\usepackage{color}
\usepackage{float}
\usepackage{mathrsfs}
\usepackage{amsmath}
\usepackage{amsthm}
\usepackage{amssymb}
\usepackage{graphicx}
\usepackage{wasysym}

\makeatletter

\numberwithin{equation}{section}
\numberwithin{figure}{section}
\theoremstyle{plain}
\newtheorem{thm}{\protect\theoremname}[section]
  \theoremstyle{remark}
  \newtheorem{rem}[thm]{\protect\remarkname}
  \theoremstyle{definition}
  
  \theoremstyle{plain}
  \newtheorem{prop}[thm]{\protect\propositionname}
  \theoremstyle{plain}
  \newtheorem{cor}[thm]{\protect\corollaryname}
  \theoremstyle{plain}
  \newtheorem{lem}[thm]{\protect\lemmaname}

\newtheorem{assump}[thm]{Assumption}
\usepackage[font={footnotesize}, labelfont={footnotesize, bf}]{caption}

\makeatother

\usepackage{babel}
  \providecommand{\corollaryname}{Corollary}
  \providecommand{\examplename}{Example}
  \providecommand{\lemmaname}{Lemma}
  \providecommand{\propositionname}{Proposition}
  \providecommand{\remarkname}{Remark}
\providecommand{\theoremname}{Theorem}

\begin{document}
\title{\textbf{\Large{}An SPDE Model for Systemic Risk}\\[-4pt]
\textbf{\Large{}with Endogenous Contagion}}

\author{Ben Hambly and Andreas S{\o}jmark\\[1pt]
	\medskip{}
	Mathematical Institute, University of Oxford}
\date{September 28, 2018}
\maketitle
\begin{abstract}
We propose a dynamic mean field model for `systemic risk' in large financial
systems, which we derive from a system of interacting diffusions on the positive half-line with an absorbing boundary at the origin. These diffusions represent the distances-to-default of financial institutions and absorption at zero corresponds to default. As a way of modelling correlated exposures and herd
behaviour, we consider a common source of noise and a form of mean-reversion
in the drift. Moreover, we introduce an endogenous contagion mechanism
whereby the default of one institution can cause a drop in the distances-to-default
of the other institutions. In this way, we aim to capture key `system-wide' effects on risk. The resulting mean field limit is characterized uniquely by a nonlinear SPDE
on the half-line with a Dirichlet boundary condition. The density of this SPDE gives the conditional law of a non-standard `conditional' McKean--Vlasov diffusion, for which we provide a novel upper Dirichlet heat kernel type estimate that is essential to the proofs. Depending on
the realizations of the common noise and the rate of mean
reversion, the SPDE can exhibit rapid accelerations in the loss of
mass at the boundary. In other words, the contagion mechanism can give rise to periods of significant systemic default clustering.
\end{abstract}

\section{Introduction\label{sec:Introduction}}

One of the most important lessons of the 2007\textendash 2009 financial crisis is the imperative nature of  system-wide
perspectives on risk. That is, financial models need to take into account the interconnectedness
of the financial system and they must incorporate sensible notions of financial
contagion whereby the distress of one institution can lead to
losses for the other members of the system. While such ideas have already
had implications for macroprudential policies (Benoit et al.~\cite{Benoit et al}, Duffie
\cite{Duffie reform}) and stress
testing practices (Dees, Henry \& Martin \cite{Jerome Henry STAMPE}),
there is still a pressing need for a better understanding of the dynamic
feedback effects and amplification mechanisms that seem to have been
the real drivers of the financial crisis. 

Indeed, it is widely accepted that \textemdash{} up to an order of
magnitude \textemdash{} the extent of the crisis cannot be explained
by simple references to (linear) exogenous shocks such as the devaluation of
mortgage-backed securities (Cochrane \cite{cochrane}, Brunnermeier
\cite{Bru2009}, Hellwig \cite{Hellwig}). Instead, smaller scale shocks are
understood to have unfolded into a spiral of events rooted within
the financial system itself and amplified by a myriad of interactions
between the individual institutions. Accordingly, there have been
judicious calls for a better understanding of the  \emph{endogenous} (nonlinear) nature
of systemic risk (Pedersen \cite{Ped2009}, Danielsson,
Shin \& Zigrand \cite{DSZ2009}) and it has been emphasized that systemic
risk is inherently \emph{dynamic} with a gradual build-up typically taking
place in the background before it materializes in a crisis (Brunnermeier,
Gorton \& Krishnamurthy \cite{BGK2012}).

\subsection{A system-wide perspective on risk}
 Consider a `representative' member of a large financial system and let $X_t\in[0,\infty)$ denote a measure of its financial health at time $t$, which we call its \emph{distance-to-default}. Neglecting system-wide effects, it is a classical approach in structural credit risk theory to simply model $X_t$ in isolation by a Brownian motion with drift --- where default occurs at the first hitting time of zero. If `isolation' is taken to mean that all the actors are assumed independent, then this means that the overall health of the financial system is effectively described by a linear (deterministic) heat equation.
 
The aim of this paper is to introduce \emph{systemic risk} into this picture by instead proposing a mean field model derived from an interacting particle system that incorporates simple notions of (i) common exposures, (ii) herd behaviour, and (iii) endogenous contagion. In turn, the health of the financial system will now be described by a nonlinear mean field type SPDE on the positive half-line (Theorems 2.4 \& 2.6) and, in a suitable sense, the dynamics of a `representative' financial institution is no longer a Brownian motion with drift but rather a conditional McKean--Vlasov type SDE with dependence on the conditional law of its paths given the noise of the common exposures (Theorem 2.7).

\subsection{Established literature on systemic risk}

A decade after the global financial crisis,
there is by now a significant body of work dealing with the complex
web of interactions in the financial system and the dual
rôle of interconnectedness as a source of diversification or a channel
for contagion.

In terms of mathematical modelling, it is possible to identify three main approaches to the challenges of systemic risk. First, there
is the large literature on \emph{network-based} models for clearing
and contagion, which extend the early frameworks of Eisenberg \& Noe
\cite{EN2001} and Allen \& Gale \cite{AG2000} (for a comprehensive review of this approach, see \cite{GY2016}). While these network models are principally static, a dynamic extension has recently been studied in Banerjee, Bernstein \& Feinstein \cite{Feinstein}. Next, there is the
much smaller literature on dynamic \emph{mean field} models in the
spirit of Carmona, Fouque \& Sun \cite{carmona fouque} and Garnier,
Papanicolaou \& Wei \cite{GPW2013}. These models benefit from a richer dynamic and stochastic structure, but they tend to focus on quite simple interbank interactions that neglect defaults and contagion (see also \cite{fouque ichiba,FS2013,capponi,GPW2017}).
Lastly, there is the related \emph{reduced-form} literature
on intensity-based models for large portfolio credit risk (see e.g.~\cite{giesecke and spil,GSSS2015,CMZ2012}). These models seek to incorporate implicit notions of default contagion and they have been discussed in a systemic risk context by Spiliopoulos
\cite{Spi2015} and Giesecke, Schwenkler \& Sirignano \cite{GSS2017}.

The model we propose here belongs naturally to the mean field literature,
however, we develop a more flexible framework that incorporates contagion endogenously via a structural mechanism for
defaults. This approach differs markedly from the reduced-form literature, where  contagion is in the form of self-exciting point processes, and it has the added benefit of being conceptually close to the network-based approaches.

\subsection{The endogenous contagion mechanism\label{subsec:Default-contagion}}

Our starting point is inspired by recent dynamic frameworks for the \emph{structural} modelling of large portfolio credit risk (see e.g.~\cite{HL2016,hambly reisinger etc}). Specifically, we identify each financial institution (henceforth:~\emph{bank}) with a notion of its \emph{distance-to-default} given by
\begin{equation*}
Y_{t}^{i}=\log(A_{t}^{i})-\log(D_{t}^{i}),\quad\text{for}\;\; i=1,\ldots,N,
\end{equation*}
where $A_{t}^{i}$ is the market value of bank $i$'s assets and $D_{t}^{i}$
denotes its default barrier. These distances-to-default will be modelled by suitable stochastic processes on $(0,\infty)$ with absorption at the origin corresponding to default. The precise dynamics will be specified in Section~1.4, but first we discuss how to incorporate the contagion mechanism.

For simplicity, we assume that the system can be described
by assigning a weight $a_{i}^{{\scriptscriptstyle N}}=a_{i}/\sum_{n=1}^{{\scriptscriptstyle N}}a_{n}$ to each of
the banks, where $c\leq a_i \leq C$ for fixed constants $C,c>0$.
These weights can depend on the initial distances-to-default, and
they reflect the relative importance of the banks in the sense that
$a_{i}^{{\scriptscriptstyle N}}$ will determine the strength of bank
$i$'s impact on the others. Notice that $\sum_{i=1}^{{\scriptscriptstyle N}}a_{i}^{{\scriptscriptstyle N}}=1$ with $a_{i}^{{\scriptscriptstyle N}}=O(1/N)$ as $N \rightarrow \infty$. In particular, no single bank can have a macroscopic
effect on the system in the large population limit.
\begin{rem}
	[Systemically important banks] Since our model will contain a \emph{common}
	source of noise, the latter condition is not as restrictive as it
	may appear. Indeed, we could model a group of particularly influential
	banks by a separate diffusion and then treat this as a common input
	in the dynamics of the smaller banks.
\end{rem}

Suppose bank $j$ is the first to default. Its contagious impact on any other bank $i$ will be determined by the weight $a_j^{\scriptscriptstyle N}$ and a parameter $\alpha_t^i\geq0$ measuring how costly defaults are to bank $i$ at time $t$. Specifically, we model the resulting contagion by `discounting' the asset values of the other banks according to the rule
\begin{equation}\label{eq:discount_assets}
A^i_{\cdot} \longmapsto \hat{A}^i_{\cdot}:=\exp\bigl\{  -a_{j}^{\scriptscriptstyle N}\! {\textstyle\int_0^{\cdot}}\alpha^i_s 
d\mathfrak{L}^{j,\scriptscriptstyle N}_s \bigr\}A^i_\cdot, \quad \text{for each} \;\; i\neq j,
\end{equation}
with $\mathfrak{L}^{j,\scriptscriptstyle N}_t:=\int_0^t \mathfrak{K}(t-r) \mathbf{1}_{r\geq\tau_{j}}dr$, where the \emph{impact kernel} $\mathfrak{K}\in L^1(\mathbb{R}_+)$ models the gradual realisation of the losses spurred by the default. We stress that these losses are not restricted to direct counterparty exposures, but may also arise from more indirect sources such as emerging liquidity shortages, fire sales, and drops in confidence. Observe that $\mathfrak{L}^{j,\scriptscriptstyle N}_t=0$ for $t\leq\tau_{j}$ and, by requiring that $\Vert \mathfrak{K} \Vert_{L^1}=1$, we have $\mathfrak{L}^{j,\scriptscriptstyle N}_t=1$ for all $t\geq\tau_{j}+\varepsilon$ whenever $\text{supp}\mathfrak{K}\subseteq[0,\varepsilon]$, for some $\varepsilon>0$.

\begin{rem}[Interpretations of $\alpha$]\label{rem:alpha_interpreation}
Consider the case where $\alpha$ is a fixed constant and $\text{supp}\mathfrak{K}\subseteq[0,\varepsilon]$. At time $t=\tau_j+\varepsilon$, the discounting in (\ref{eq:discount_assets}) is then of the form
\[
\hat{A}_t^i=\exp\{  -\alpha a_{j}^{\scriptscriptstyle N} 
 \}A^i_t \simeq (1-\alpha a_{j}^{\scriptscriptstyle N} )A^i_t, 
\]
for large $N$, since $a_{j}^{\scriptscriptstyle N}=O(1/N)$. In other words, \emph{by time $\tau_j+\varepsilon$, the default of bank $j$ has caused  each bank $i\neq j$ to lose a proportion $\alpha a_{j}^{\scriptscriptstyle N}\!$ of their asset values, relative to what they would have been worth without the contagion}. Note also that $\alpha$ can be related to the connectivity of the system: Suppose, for example, that each default only affects a randomly sampled proportion $\hat{p}$ of the banks, each losing $\alpha a_{j}^{{\scriptscriptstyle N}}$ times their asset values upon bank $j$'s default. As $N$ gets large, this has a similar effect on the system as all the banks incurring the smaller loss of $\hat{\alpha}a_{j}^{{\scriptscriptstyle N}}$ times their asset values, where $\hat{\alpha}:=\hat{p}\alpha$.
\end{rem}

As more banks default, we continue to apply the rule from (\ref{eq:discount_assets}). Therefore, the actual (updated) asset values, $\hat{A}$, are given by
\begin{align*}
\hat{A}^i_t:=\prod_{j\neq i} \exp\Bigl\{ - a_{j}^{{\scriptscriptstyle N}}\!\!\int_0^t \!\alpha^i_s
d\mathfrak{L}^{j,{\scriptscriptstyle N}}_s \Bigr\}A_t^i
= 
\exp\Bigl\{  -\sum_{j\neq i}a_{j}^{{\scriptscriptstyle N}}\! \!\int_0^t \!\alpha^i_s
d\mathfrak{L}^{j,{\scriptscriptstyle N}}_s \Bigr\}A_t^i,
\end{align*}
for $i=1,\ldots,N$, where $\mathfrak{L}^{j,\scriptscriptstyle N}_t\!:=\int_0^t \mathfrak{K}(t-r) \mathbf{1}_{r\geq\hat{\tau}_{j}}dr$ with $\hat{\tau}_{j}:=\inf\{t>0:\hat{Y}_t^j\leq0\}$ and $\hat{Y}_t^j:=\log(\hat{A}_t^j)-\log(D_t^j)$.
Summing over the terms $a_j^{\scriptscriptstyle N}\mathfrak{L}_s^{j,\scriptscriptstyle N}$, this simplifies to
\begin{equation}\label{eq:A_hat}
\hat{A}^i_t = \exp\Bigl\{  -\int_0^t \!\alpha_s^i
d\mathfrak{L}^{\scriptscriptstyle N}_s  \Bigr\}A_t^i, \quad \text{for} \quad t<\hat{\tau}_i,\quad i=1,\ldots,N,
\end{equation}
 where
 \begin{equation}\label{eq:curly_loss}
 \mathfrak{L}^{\scriptscriptstyle N}_t := \int_0^t \mathfrak{K}(t-s)L_s^{\scriptscriptstyle N}ds \quad \text{and} \quad L_t^{\scriptscriptstyle N}:= \sum_{j=1}^{ N}a_{j}^{{\scriptscriptstyle N}}\mathbf{1}_{t\geq\hat{\tau}_{j}}.
 \end{equation}
Taking logarithms in (\ref{eq:A_hat}), it follows that the actual (updated) distances-to-default, $\hat{Y}$, have dynamics of the form
\[
d\hat{Y}_{t}^{i}= \:dY^i_t-\alpha_t^i d\mathfrak{L}_{t}^{{\scriptscriptstyle N}}\quad\text{for}\quad t<\hat{\tau}_{i},\quad i=1,\ldots,N.
\]
Here the first part, $Y^i_t$, is simply the original distance-to-default without contagion, while the latter part is a new contagion term
driven by the \emph{contagion process}, $\mathfrak{L}_{t}^{{\scriptscriptstyle N}}$, from (\ref{eq:curly_loss}).

\begin{rem}[The impact kernel $\mathfrak{K}$] Further to the above, we assume $\mathfrak{K}\in \mathcal{W}_{0}^{1,1}(\mathbb{R}_{+})$,
with $\Vert \mathfrak{K}\Vert _{1}=1$, where $\mathcal{W}_{0}^{1,p}(\mathbb{R}_{+})$
denotes the Sobolev space with one weak derivative in $L^{p}$ and
zero trace. The benefits of this construction
are: $\mathfrak{L}^{{\scriptscriptstyle N}}_t$
remains adapted, it inherits the monotonicity of $L^{{\scriptscriptstyle N}}_t$, and it has a weak derivative $\mathfrak{K}^{\prime}\ast L^{{\scriptscriptstyle N}}\in L^{\infty}$. The r\^ole of the kernel is to impose a continuous notion of latency whereby the impact of contagion is realised gradually as counterparty exposures are sorted out and indirect effects start to kick in.
\end{rem}

\begin{rem}
[Capital structure]As may be expected, the default contagion alone cannot
deplete the entire asset base. However, due in large part to the low volatility of banking assets
in normal times, financial institutions tend to have leverage ratios
as high as 85\textendash 95\% (see e.g.~\cite{berg gider,Stebulaev}). Thus, there
is ample room for the contagion to be detrimental.
\end{rem}

\subsection{A simple model for systemic risk\label{subsec:A-Base-Model}}
In addition to contagion, we  want our model to include common exposures and a notion of herding. The two latter effects have already been considered in Carmona, Fouque \& Sun \cite{carmona fouque}, by means of a (constant-coefficient) particle system with a common Brownian motion and mean reversion in the drift. Inspired by this, we can now present a precise formulation of our `base case' model for systemic risk: Letting $X^j$ denote the actual distances-to-default with default times $\tau_j:=\inf\{t>0:X_t^j\leq0\}$, for $j=1,\ldots,N$, we propose to model a large financial system  by an interacting particle system
of the form \vspace{-6pt}
\begin{align*}
dX_{t}^{j}=&\;\mu(t,X_{t}^{j})dt+\pi(t,\nu_{t}^{{\scriptscriptstyle N}})\sum_{i=1}^{N}a_{i}^{{\scriptscriptstyle N}}\!\cdot\!\bigl\{\bigl(X_{t}^{i}\mathbf{1}_{t<\tau_{i}}+\gamma(t,\nu_{t}^{{\scriptscriptstyle N}})\mathbf{1}_{t\geq\tau_{i}}\bigr)-X_{t}^{j}\bigr\} dt\\
&\;+\sigma(t,X_{t}^{j})\bigl(\sqrt{1-\rho(t,\nu_{t}^{{\scriptscriptstyle N}})^{2}}dW_{t}^{j}+\rho(t,\nu_{t}^{{\scriptscriptstyle N}})dW_{t}^{0}\bigr)-\alpha(t,X_t^j,\nu_{t}^{{\scriptscriptstyle N}})d\mathfrak{L}_{t}^{{\scriptscriptstyle N}},
\end{align*}
where $W^0$ and $W^1,\ldots,W^{\scriptscriptstyle N}$ are independent Brownian motions with
\[
\mathfrak{L}_{t}^{{\scriptscriptstyle N}}=\int_{0}^{t}\!\mathfrak{K}(t-s)L_{s}^{{\scriptscriptstyle N}}ds,\quad L_{t}^{{\scriptscriptstyle N}}=\sum_{i=1}^{N}a_{i}^{{\scriptscriptstyle N}}\mathbf{1}_{t\geq\tau_{i}},\quad \text{and}\quad \nu_{t}^{{\scriptscriptstyle N}}=\sum_{i=1}^{N}a_{i}^{{\scriptscriptstyle N}}\mathbf{1}_{t<\tau_{i}}\delta_{X_{t}^{i}}.
\]
Notice that, after each default, we do not renormalize the mean reverting interaction in the drift
as it is intended to reflect herding in investment decisions:
if banks are defaulting, this suggests the investments are not
performing and thus defaults should not suddenly adjust the drift
upwards by a renormalization. Also, we emphasize that our model is
intended for the study of short-run market imperfections rather than the long-run behaviour of the financial system. The parameters can be summarized as follows:
\begin{itemize}
\item $\sigma$ and $\mu$ model, respectively, the volatility of the banks and their core return net of the rate of change in the default
barrier.
\item $\rho$
is the correlation parameter which reflects the extent of common exposures.
\item $\alpha$
is the contagion parameter which decides how costly defaults are to the system.
\item $\pi$
determines the rate of mean reversion which adjusts the core return due
to herding in investment decisions or other interbank interactions.
\item $\gamma$
can capture what is left after defaults and could be influenced
by actions of the central bank or government seeking to stabilize the drift of the
system.
\end{itemize}
In order to capture system-wide influences on these variables, it is natural to allow them to depend on the empirical measure  $\nu^{{\scriptscriptstyle N}}$. In particular, the correlation can then act as an indirect source of contagion, in line with the observation that
correlations tend to increase in times of financial distress (see e.g.~Cont
\& Wagalath \cite{Cont fire sale forensics,CW13}). Similarly, the
rate of herding and the costliness of contagion can vary with the
health of the system, which can capture potentially self-reinforcing
amplification mechanisms. Furthermore, we will incorporate discontinuities
in the dependence on the losses, $L_t^{{\scriptscriptstyle N}}$, which can allow
for more abrupt adjustments of the herding, the correlation, or the costliness of defaults
(see also Section \ref{subsec:Credit-derivatives}).

\subsection{Outline of the paper}

In Section 2 we state our main results concerning the existence of a unique mean field limit for a general version of the systemic risk model from Section \ref{subsec:A-Base-Model}. Furthermore, we discuss some qualitative insights for systemic risk and consider closely related problems.

In Section 3 we study the regularity of the particle system,
which is centred around the boundary and tail behaviour of the densities
of the particles. The backbone of this is a novel family of upper Dirichlet
heat kernel type estimates, whose proofs we postpone to Section 6
in order to make the presentation as clear as possible.

In Section 4 we proceed to establish tightness of the system and
we show that the resulting limit points are solutions to a nonlinear
SPDE on the positive half-line.

In Section 5 we rely on energy estimates to prove uniqueness of the
SPDE and thus we deduce the full convergence in law to this limit.
We present the uniqueness proof in Section 5.1, postponing the 
technical estimates to Sections 5.2 and 5.3. Therefore, the reader 
can get a complete picture of existence and uniqueness by only reading up to Section 5.1.

\section{Main results}
 \label{sec:Main-Mathematical-Results}
To make our framework as flexible as possible, we will consider a more
general version of the model introduced in Section \ref{subsec:A-Base-Model}.
Specifically, we will focus on particle systems of the form
\begin{equation}
\left\{ \begin{aligned}dX_{t}^{i}= & \, b(t,X_{t}^{i},\nu_{t}^{{\scriptscriptstyle N}})dt + \sigma(t,X_{t}^{i})\sqrt{1-\rho(t,\nu_{t}^{{\scriptscriptstyle N}})}dW_{t}^{i}  \\
 & + \sigma(t,X_{t}^{i})\rho(t,\nu_{t}^{{\scriptscriptstyle N}})dW_{t}^{0} -\alpha(t,X_{t}^{i}, \nu_t^{\scriptscriptstyle N})d\mathfrak{L}_{t}^{{\scriptscriptstyle N}},\\
 \mathfrak{L}_{t}^{{\scriptscriptstyle N}}= & \,(\mathfrak{K}\ast L^{{\scriptscriptstyle N}})_{t}, \;\;  L_{t}^{{\scriptscriptstyle N}}=1-\nu_t^{{\scriptscriptstyle N}}(0,\infty),\\
 \nu_{t}^{{\scriptscriptstyle N}}=&\,{\textstyle\sum_{i=1}^{{\scriptscriptstyle N}}}a_{i}^{{\scriptscriptstyle N}}\mathbf{1}_{t<\tau_{i}}\delta_{X_{t}^{i}}, \;\; \tau_{i}=\inf\{ t\geq0:X_{t}^{i}\leq0\},
\end{aligned}\right.
\label{eq: Particle System}
\end{equation}
where  $\mathfrak{K}\in\mathcal{W}_{0}^{1,1}(\mathbb{R}_{+})$ and $W^0,\ldots,W^{\scriptscriptstyle N}$ are independent Brownian motions. Concerning the weights, we assume that there exist $C,c>0$ such that
\begin{equation}
a_{i}^{{\scriptscriptstyle N}}=\frac{a_{i}(X_{0}^{i})}{{\textstyle \sum_{j=1}^{{\scriptscriptstyle N}}}a_{j}(X_{0}^{j})}\quad\text{with}\quad c\leq a_{i}(\cdot)\leq C\quad\text{for each} \quad i=1,\ldots,N.
\label{eq: a_i coefficients}
\end{equation}
For the well-posedness of (\ref{eq: Particle System}) we refer to the beginning remarks of Section 3. As regards the model from Section \ref{subsec:A-Base-Model}, we stress that the drift $b$ depends explicitly on the \emph{mean process} $M_t^{\scriptscriptstyle N}:=\langle \nu_t^{\scriptscriptstyle N} , \text{Id} \rangle$ and similarly the \emph{loss process} $L^{{\scriptscriptstyle N}}_t$ plays a vital r\^ole. More generally, these two mean field statistics can serve as useful indicators of systemic risk. However, their convergence as $N\rightarrow\infty$ does not follow directly from our notion of convergence for the empirical measures $\nu^{{\scriptscriptstyle N}}$, so they will require special attention.

\subsection{Assumptions}

In view of Section \ref{subsec:A-Base-Model}, we need to consider
 local notions of Lipschitzness in $X_{t}^{i}$ and $\nu_{t}^{{\scriptscriptstyle N}}$, and crucially we must allow the drift to
have linear growth in both $X_{t}^{i}$ and $M_{t}^{{\scriptscriptstyle N}}$. To this end, we need suitable notions of distance on the space of sub-probability measures $\mathbf{M}_{\leq1}(\mathbb{R})$. This leads us to introduce the Kantorovich type distances
\begin{align*}
d_0(\mu,\tilde{\mu}):=&\;\sup \{ |\langle \mu-\tilde{\mu}, \psi \rangle| : \Vert \psi \Vert_{\text{Lip}} \leq 1, |\psi(0)|\leq 1\}, \;\;\text{and}\\ 
d_1(\mu,\tilde{\mu})&:=\,\sup \{ |\langle \mu-\tilde{\mu}, \psi \rangle| : \Vert \psi \Vert_{\text{Lip}} \leq 1, \Vert \psi \Vert_{\infty}\leq 1 \}.
\end{align*}
\begin{assump}[Structural assumptions]\label{Assumption 1 - finite system} Let the coefficients of (\ref{eq: Particle System}) be of the form $b(t,x,\mu,\ell^\mu)$, $\alpha(t,x,\mu,\ell^\mu)$, $\sigma(t,x)$, and $\rho(t,\mu,\ell^\mu)$ with $\ell^\mu:=1-\mu(0,\infty)$. We assume:
\begin{enumerate}
\item [\emph{(i).}]\emph{(Linear growth and space/time regularity).} Let $g=b,\alpha$. The map
$x\mapsto g(t,x,\mu,\ell)$ is $\mathcal{C}^{2}(\mathbb{R})$ and
 $(t,x)\mapsto\sigma(t,x)$ is $\mathcal{C}^{1,2}([0,T]\times\mathbb{R})$.
Moreover, there exists $C>0$ s.t.
\begin{align*}
&\left|g(t,x,\mu,\ell)\right|\leq C(1+\left|x\right|+ \langle \mu, \left|\cdot\right| \rangle),\;\; |\partial_{x}^{(n)}g(t,x,\mu,\ell)| \leq C,\;\; n=1,2,\\
&\left|\sigma(t,x)\right|\leq C,\;\; \left|\partial_{t}\sigma(t,x)\right|\leq C,\;\; \emph{\text{and} }\;\; |\partial_{x}^{(n)}\sigma(t,x)| \leq C,\;\; n=1,2.
\end{align*}
\item [\emph{(ii).}]\emph{(Local $d_0/d_1$-Lipschitzness in $\mu$).} Let $g=b,\alpha$. There exists $C>0$ s.t.
\begin{align*}
\left|g(t,x,\mu,\ell)-g(t,x,\tilde{\mu},\ell)\right| \leq& \;C(1+\left|x\right|+ \langle \mu, \left|\cdot\right| \rangle) d_0(\mu,\tilde{\mu})
\\ \left|\rho(t,\mu,\ell)-\rho(t,\tilde{\mu},\ell)\right| \leq& \;C(1+ \langle \mu, \left|\cdot\right| \rangle) d_1(\mu,\tilde{\mu})
\end{align*}
\item [\emph{(iii).}]\emph{(Piecewise local Lipschitzness in $\ell$).} There
exist $0=\theta_{0}<\cdots<\theta_{k}=1$ s.t.
\[
\bigl|g(t,x,\mu,\ell)-g(t,x,\mu,\tilde{\ell})\bigr|\leq C(1+\left|x\right|+ \langle \mu, \left|\cdot\right| \rangle)|\ell-\tilde{\ell}|
\]
whenever $\ell,\tilde{\ell}\in[\theta_{i-1},\theta_{i})$ for some
$1\leq i\leq k$, where $g$ is any of $b$, $\alpha$, or $\rho$.
\item [\emph{(iv).}]\emph{(Non-degeneracy).} There exists $\epsilon>0$ s.t. $0<\epsilon\leq\sigma(t,x)$ and $0\leq\rho(t,\mu)\leq1-\epsilon$.
\item [\emph{(v).}]\emph{(Sub-Gaussian initial law).} The sequence $X_{0}^{1},\ldots,X_{0}^{{\scriptscriptstyle N}}$
is  i.i.d and independent of the driving Brownian motions. Their common law, $\mu_0$, has a density in $L^2(0,\infty)$ and
\[
\exists\epsilon>0\quad\text{s.t.}\quad\mu_{0}(\lambda,\infty)=O(\exp\{-\epsilon\lambda^{2}\})\quad\text{as}\quad\lambda\rightarrow\infty.
\]
Letting $\nu_{0}^{{\scriptscriptstyle N}}:=\sum_{i=1}^{{\scriptscriptstyle N}}a_{i}^{{\scriptscriptstyle N}}\delta_{X_{0}^{i}}$, we assume $\nu_{0}^{{\scriptscriptstyle N}}$ converges weakly to some probability measure $\nu_{0}$ with a density $V_{0}\in L^{2}(0,\infty)$ s.t.
$\left\Vert xV_{0}\right\Vert _{L^{2}}=\int_{0}^{\infty}|xV_{0}(x)|^{2}dx<\infty$.
\end{enumerate}
\end{assump}
\begin{rem}
\label{rem:weaker_assumptions}In (i) it suffices for $\partial_{xx}b$,
$\partial_{xx}\sigma$ and $\partial_{t}\sigma$ to exist as weak
derivatives in $L^{\infty}$, and if $\sigma(t,x)=\sigma_{1}(t)\sigma_{2}(x)$
then no regularity in time is needed. Moreover, the assumption of a common law in (v) is not essential, but this is not the focus here.
\end{rem}

For the uniqueness of the limit SPDE we must restrict attention to
a class of processes with reasonable regularity and, naturally, we
want the limit points of the particle system to be included
in this class. As we will see in Section \ref{sec:Tightness-and-Convergence}
(see also Theorem \ref{thm:MCKEAN-VLASOV}), the conditions we impose
below are indeed weaker than those that hold for the limit points of
$\nu^{{\scriptscriptstyle N}}$.

\begin{assump}[Conditions for uniqueness]\label{Assumption 2.3}In Theorem \ref{thm:UNIQUENESS} we consider
the class of càdlàg processes $\tilde{\nu}$ that take values in $\mathbf{M}_{\leq1}(\mathbb{R})$ and satisfy the following conditions:
\begin{enumerate}
\item [\emph{(i).}]\emph{(Support on $\mathbb{R}_{+}$).} For every $t\in\left[0,T\right]$,
$\tilde{\nu}_{t}$ is supported on $\mathbb{R}_{+}=[0,\infty)$.
\item [\emph{(ii).}]\emph{(Increasing loss).} $\tilde{L}_{t}:=\!1-\tilde{\nu}_{t}(0,\infty)$
is strictly increasing when it is less than $1$.
\item [\emph{(iii).}]\emph{(Exponential tails).} For every $\epsilon>0$,
we have
\[
\mathbb{E}{\textstyle\int_{0}^{T}}\tilde{\nu}_{t}(a,\infty)dt=o(\exp\{-\epsilon a\})\quad\text{as}\quad a\rightarrow\infty.
\]
\item [\emph{(iv).}]\textcolor{black}{\emph{(Boundary decay).}} 
There exists $\beta>0$ such that
\[
\mathbb{E}{\textstyle\int_{0}^{T}}\tilde{\nu}_{t}(0,\varepsilon)dt=O(\varepsilon^{1+\beta})\quad\text{as}\quad\varepsilon\rightarrow0.
\]
\item [\emph{(v).}]\emph{(Spatial concentration).} There exist $C>0$ and
$\delta>0$ such that
\[
\mathbb{E}{\textstyle\int_{0}^{T}}(\tilde{\nu}_{t}(a,b)){}^{2}dt\leq C\left|b-a\right|^{\delta}\quad\forall a,b\in\mathbb{R}.
\]
\end{enumerate}
\end{assump}

\subsection{The limit SPDE\label{subsec:Existence-and-Uniqueness}}

Let $\mathscr{S}$ denote the space of Schwartz
functions on $\mathbb{R}$ and let $\mathscr{S}^{\prime}$ denote
its dual, the space of tempered distributions. Then we can think of
the empirical measures $\nu^{{\scriptscriptstyle N}}$ as living in the space $D_{\mathscr{S}^{\prime}}=D_{\mathscr{S}^{\prime}}[0,T]$
consisting of $\mathscr{S}^{\prime}$-valued càdlàg processes on $[0,T]$.
With a view to exploiting the monotonicity of the loss process, $L^{{\scriptscriptstyle N}}$, it becomes 
natural to use weak convergence with respect to Skorohod's M1 topology
on $D_{\mathscr{S}^{\prime}}$. 
The extension of the M1 topology to
this setting was introduced in \cite{Led2016} and further details can be found there.

It remains to mention that we consider the limit SPDE in its weak
formulation with respect to the space of test functions $\mathscr{C}_{{\scriptscriptstyle 0}}$
as defined by
\[
\mathscr{C}_{{\scriptscriptstyle 0}}:=\{\phi\in\mathscr{S}:\phi(0)=0\}.
\]
This encodes the idea that the SPDE is posed as a Dirichlet problem
on $\mathbb{R}_{+}$.
\begin{thm}
[Limit SPDE]\label{thm:EXISTENCE} Suppose Assumption \ref{Assumption 1 - finite system}
is satisfied. Then $(\nu^{{\scriptscriptstyle N}},W^{0})_{N\geq1}$
is tight on $(D_{\mathscr{S}^{\prime}},\text{\emph{M1}})$ and each
limit point $(\nu,W^{0})$ is a continuous $\mathbf{M}_{\leq1}(\mathbb{R}_+)$-valued
process satisfying Assumption \ref{Assumption 2.3}. Moreover, $\nu$
obeys, with probability 1, the limit SPDE
\begin{align}
d\negthinspace\left\langle \nu_{t},\phi\right\rangle = & \,\left\langle \nu_{t},b(t,\cdot,\nu_{t})\partial_{x}\phi\right\rangle dt+{\textstyle \frac{1}{2}}\left\langle \nu_{t},\sigma(t,\cdot)^{2}\partial_{xx}\phi\right\rangle dt\label{eq:LIMIT SPDE}\\
 & \;+\left\langle \nu_{t},\sigma(t,\cdot)\rho(t,\nu_{t})\partial_{x}\phi\right\rangle dW_{t}^{0}-\left\langle \nu_{t},\alpha(t,\cdot,\nu_t)\partial_{x}\phi\right\rangle d\mathfrak{L}_{t},\nonumber 
\end{align}
for $\phi\in\mathscr{C}_{{\scriptscriptstyle 0}}$ and $t\in[0,T]$, where $\mathfrak{L}_{t}:={\textstyle \int_{0}^{t}}\mathfrak{K}(t-s)L_{s}ds$ with $L_{t}:=1-\nu_{t}(0,\infty)$. Finally, if $\nu$ is attained along a subsequence $(\nu^{{\scriptscriptstyle N}_{\hspace{-1pt}k}},W^{0})_{k\geq1}$,
then $(L^{{\scriptscriptstyle N}_{\hspace{-1pt}k}},W^{0})$ and $(M^{{\scriptscriptstyle N}_{\hspace{-1pt}k}},W^{0})$
converge weakly to $(L,W^{0})$ and $(M,W^{0})$ on $(D_{\mathbb{R}},\text{\emph{M1}})\times(C_{\mathbb{R}},\left\Vert \cdot\right\Vert _{\infty})$, where $M_t:=\langle \nu_t, \emph{Id} \rangle$.
\end{thm}

It is instructive to think of the measure-valued SPDE (\ref{eq:LIMIT SPDE})
in terms of the equation for its density process, $V_{t}$, which
is guaranteed to exist in $L^{2}$ by Corollary \ref{cor:DENSITY}.
Based on the boundary decay from Theorem \ref{thm:MCKEAN-VLASOV}
below, a simple Borel\textendash Cantelli argument implies $V_{t}(0)=0$
in the sense that $\lim_{\varepsilon\downarrow0}\varepsilon^{-1}\nu_{t}(0,\varepsilon)=0$,
see Lemma \ref{lem:BOREL-CANTELLI-ARGUMENTS}. Therefore, integrating
by parts in (\ref{eq:LIMIT SPDE}) leads to the following observation:
\begin{rem}
[SPDE for the density]\label{rem:DENSITY}If $\nu$ is a limit point
from Theorem \ref{thm:EXISTENCE}, then it has a density process
$V_t$ in $L^{2}(\mathbb{R}_{+})$ which, in the weak sense, satisfies the Dirichlet
problem
\begin{align*}
dV_{t}(x)&={\textstyle \frac{1}{2}}\partial_{xx}\bigl(\sigma(t,x)^{2}V_{t}(x)\bigr)dt-\partial_{x}\bigl(b(t,x,\nu_{t})V_{t}(x)\bigr)dt\\
&-\rho(t,\nu_{t})\partial_{x}\bigl(\sigma(t,x)V_{t}(x)\bigr)dW_{t}^{0}+\alpha(t,x,\nu_t)\partial_{x}V_{t}(x)d\mathfrak{L}_{t},\quad V_{t}(0)=0,
\end{align*}
with $\mathfrak{L}_t=(\mathfrak{K}\ast L)_t$ and $L_{t}=1-\int_{0}^{\infty}V_{t}(x)dx$.
Assuming sufficient regularity, we can use this SPDE for $V$ and integrate by parts to obtain
the formal expression
\[
\frac{d}{dt}\mathfrak{L}_{t}=\frac{1}{2}\int_{0}^{t}\!\mathfrak{K}(t-s)\sigma(s,0)\partial_{x}V_{s}(0)ds.
\]
That is, the contagion term driven by $\mathfrak{L}$ acts like an
extra transport term proportional to the flux across the boundary --- smoothed in time according to $\mathfrak{K}$. In the case of constant
coefficients and $\alpha\equiv0$, sharp regularity results can be
found in \cite{Led2014}.
\end{rem}

Since every limit point of the finite particle system obeys the SPDE (\ref{eq:LIMIT SPDE}), we can deduce the full weak convergence to this mean-field limit once we have uniqueness of the SPDE in the
class of solutions satisfying Assumption \ref{Assumption 2.3}.
\begin{thm}
[Uniqueness \& LLN]\label{thm:UNIQUENESS}Let $(\nu,W^{0})$
be as in Theorem \ref{thm:EXISTENCE} and suppose $\tilde{\nu}$ is
another solution to the SPDE (\ref{eq:LIMIT SPDE})\emph{ }satisfying
Assumption \ref{Assumption 2.3}. Then, with probability 1, \textcolor{black}{\vspace{-2pt}
}
\[
\nu_{t}(A)=\tilde{\nu}_{t}(A)\quad\forall t\in[0,T]\quad\forall A\in\mathcal{B}(\mathbb{R}).
\]
In particular, a solution to the SPDE (\ref{eq:LIMIT SPDE}) has a
unique law on $(D_{\mathscr{S}^{\prime}},\text{\emph{M1}})\times(C_{\mathbb{R}},\left\Vert \cdot\right\Vert _{\infty})$
and $(\nu^{{\scriptscriptstyle N}},W^{0})$ converges weakly to this
law. Furthermore, the loss process $L^{{\scriptscriptstyle N}}$ and the mean process $M^{{\scriptscriptstyle N}}$
converge weakly to $L$ and $M$ as defined in\emph{ }Theorem \ref{thm:EXISTENCE}.
\end{thm}

In view of the pathwise uniqueness, the Yamada\textendash Watanabe
theorem ensures that the unique solution to the SPDE can be taken
to be strong. Our final result
shows that this solution can be recast as the conditional law, given
the common Brownian motion $W^{0}$, of a `conditional' McKean\textendash Vlasov
type diffusion with absorbing boundary. Furthermore, we establish an
Aronson type upper bound on the density of this absorbed SDE and
we show that it has power law decay near the boundary. This is analogous
to the classical Dirichlet heat kernel estimates that are available
for more standard diffusions
(see \cite{saloff-coste,Cho 2012}).
\begin{thm}
	[Conditional McKean--Vlasov formulation]\label{thm:MCKEAN-VLASOV}Let $(\nu,W^{0})$
	be the unique strong solution to the SPDE (\ref{eq:LIMIT SPDE}).
	Then, for any Brownian motion $W\perp(X_{0},W^{0})$, we have
	\[
	\nu_{t}=\mathbb{P}(X_{t}\in\cdot\,,\:t<\tau\mid W^{0})\quad\text{for}\quad\tau:=\inf\left\{ t>0:X_{t}\leq0\right\} ,
	\]
	where $X_{t}$ is the unique solution to the conditional McKean\textendash Vlasov
	diffusion
	\[
	\begin{cases}
	dX_{t}=b(t,X_{t},\nu_{t})dt+\sigma(t,X_{t})\sqrt{1-\rho(t,\nu_{t})^{2}}dW_{t}\\
	\qquad\quad+\,\sigma(t,X_{t})\rho(t,\nu_{t})dW_{t}^{0}
	-\alpha(t,X_t,\nu_t)d\mathfrak{L}_{t}\\[2pt]
	 \mathfrak{L}_{t}=(\mathfrak{K}\ast L)_{t},\;\;L_{t}=\mathbb{P}(\tau\leq t\mid W^{0}), \;\; X_{0}\sim\nu_{0}.
	\end{cases}
	\]
	Moreover, the absorbed process has a density $\mathfrak{p}(t,x)=\mathbb{E}V_{t}(x)$
	so that\vspace{-2pt}
	\[
	\mathbb{E}\nu_{t}(a,b)=\mathbb{P}(X_{t}\in(a,b),\,t<\tau)=\int_{a}^{b}\mathfrak{p}(t,x)dx,
	\]
	where, for any $\epsilon>0$, there exists $\kappa\in(0,1)$ and
	 $C,c>0$ such that 
	\begin{equation}
	\mathfrak{p}(t,x)\leq C\!\int_{0}^{\infty}\Bigl({\textstyle \frac{1}{\sqrt{t}}}({\textstyle \frac{x}{\sqrt{t}}}\land1)({\textstyle \frac{y}{\sqrt{t}}}\land1)+({\textstyle \frac{x^{\kappa}y^{\kappa}}{t^{\kappa}}}\land1)e^{\epsilon y^{2}}\Bigr)\bigl(1\land e^{-\frac{(x-y)^{2}}{ct}+c_{x\hspace{-0.4bp},\hspace{-0.4bp}y}}\bigr)d\nu_{0}(y)\label{eq: Aronson Estimate}
	\end{equation}
	and
	\begin{equation}
	\mathfrak{p}(t,x)\leq C\!\int_{0}^{\infty}\bigl({\textstyle \frac{1}{\sqrt{t}}}+e^{\epsilon y^{2}}\bigr)e^{-\frac{(x-y)^{2}}{ct}}d\nu_{0}(y),\label{eq: Aronson estimate whole space}
	\end{equation}
	with $c_{x,y}\apprle\left|x-y\right|(x\wedge y)$. Furthermore, if $\sigma(t,x)=\sigma_{1}(t)\sigma_{2}(x)$,
	then $c_{x,y}\equiv0$.
\end{thm}

\begin{rem}
The bound (\ref{eq: Aronson estimate whole space}) also holds on
the whole space. Note that the correction term $c_{x,y}$
in (\ref{eq: Aronson Estimate}) is only relevant when $x,y\rightarrow\infty$
jointly (inside a cone determined by $c_{x,y}$) and in this case
the tail is controlled by (\ref{eq: Aronson estimate whole space}).
Note also that the factor $e^{\epsilon y^{2}}$ requires the initial
law to be sub-Gaussian. This is a natural consequence of the linear
growth. If e.g.~the drift is bounded, then $e^{\epsilon y^{2}}$ can be dropped and $\kappa$
can be taken arbitrarily close to $1$. Similarly, the factor $e^{\epsilon y^{2}}$ drops if the linear growth is in terms of $\left|M_{t}-M_{0}\right|$ and $\left|X_{t}-X_{0}\right|$.
\end{rem}

\begin{proof}
[Proof of Theorem \ref{thm:MCKEAN-VLASOV}]Given Theorems \ref{thm:EXISTENCE}
and \ref{thm:UNIQUENESS}, the first claim is straightforward. Indeed, treating
$\nu_{t}$
as given, the SDE has a strong solution, $\tilde{X}$, by the standard theory
and we can thus define $\tilde{\nu}_{t}:=\mathbb{P}(\tilde{X}_{t}\in\cdot\,,\:t<\tau\mid W^{0})$.
Arguing as in Section 9 of \cite{HL2016}, with the obvious changes,
we can apply Itô's formula to $\phi(\tilde{X}_{t})$, for $\phi\in\mathscr{C}_{{\scriptscriptstyle 0}}$,
to show that $\tilde{\nu}$ solves a linear version of the SPDE (\ref{eq:LIMIT SPDE})
with the fixed $\nu$. But then $\tilde{\nu}=\nu$ by uniqueness
(for the linear SPDE) which proves the claim. The density estimates
(\ref{eq: Aronson Estimate}) and (\ref{eq: Aronson estimate whole space}) are more involved, but they
will follow from Proposition \ref{prop:First regularity prop} (based on the work in
Section 6).
\end{proof}
For the proofs of Theorems \ref{thm:EXISTENCE} and \ref{thm:UNIQUENESS},
our techniques build upon and extend the methods of \cite{HL2016},
which dealt with a similar SPDE problem, albeit without the contagion
term (i.e.~$\alpha\equiv0$) and with a bounded drift as well as with coefficients that only depend on $\nu$ via the losses $L$.
The main insight from \cite{HL2016} is that it can be fruitful to approach
uniqueness via suitable energy estimates in $H^{-1}$ (the dual of the first Sobolev space) when combined with careful control over
$\mathbb{E}\nu_{t}(0,\varepsilon)$ as $\varepsilon\downarrow0$
and $\mathbb{E}\nu_{t}(\lambda,\infty)$ as $\lambda\uparrow\infty$.
However, in order for this to work in our setting, several
extensions of the arguments in \cite{HL2016} are needed and, crucially, we must rely on novel upper bounds for the densities
of the absorbed particles (Proposition~\ref{prop: Aronson estimate}).
Ultimately, we thus arrive at the essential $H^{-1}$ energy estimate (Proposition~\ref{smooth_energy_est} and Lemma \ref{lem: Lemma8.8 from hambly-ledger}),
by establishing power law decay of $x\mapsto\mathbb{E}V_{t}(x)$ near
the Dirichlet boundary and Gaussian tails towards infinity (Corollary~\ref{Cor:Regularity of the empirical measures}).

The proofs of the density bounds are the subject of Section 6 and it is
these efforts that lead to the absorbing density estimates
in Theorem~\ref{thm:MCKEAN-VLASOV}. As far as we are aware, these
estimates are not available from results elsewhere in the literature,
and we believe they are of independent interest. In particular, they have already proved useful in a separate paper \cite{HLS2018}
related to the problem discussed in Section \ref{subsec:Non-smoothed loss}.

\subsection{Financial contagion and default clustering}

Recall that the nonlinearities of the limit SPDE are of a non-local nature tied to the flux across the origin. This differentiates our setting from the existing theory for Zakai type SPDEs and it has striking consequences for the qualitative behaviour of the solution.

In particular, the health of the financial system --- governed by the limit SPDE --- depends critically on the interplay between the common noise $W^0$ and the nonlinear effects of the contagion process $\mathfrak{L}$ (see Fig.~\ref{Fig1}). At least conceptually, this captures the main forces at work in the 2007--2009 financial crisis, where  contagion ensued from the correlated corrections across the U.S. housing markets as \emph{``financial institutions had levered
	up on similar large portfolios of securities and loans that faced
	little idiosyncratic risk, but large amounts of systematic risk''
}(Acharya et al.~\cite{Acharya Pedersen}).
\begin{figure}[H]\vspace{-9pt}
	\begin{center}
		\hspace{-2cm} \includegraphics[width=0.59\textwidth]{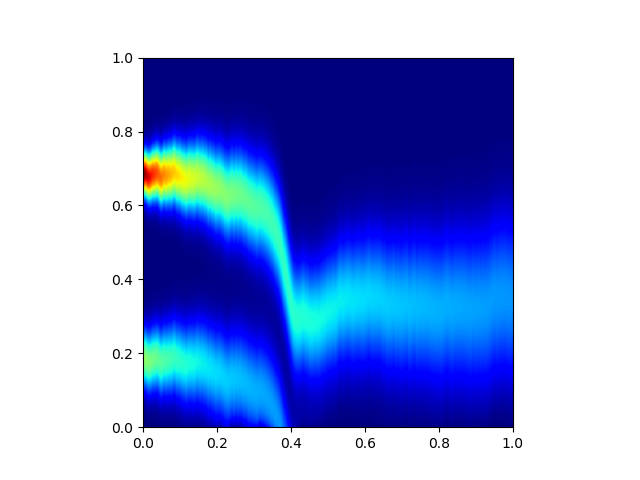} \hspace{-1.7cm} \includegraphics[width=0.59\textwidth]{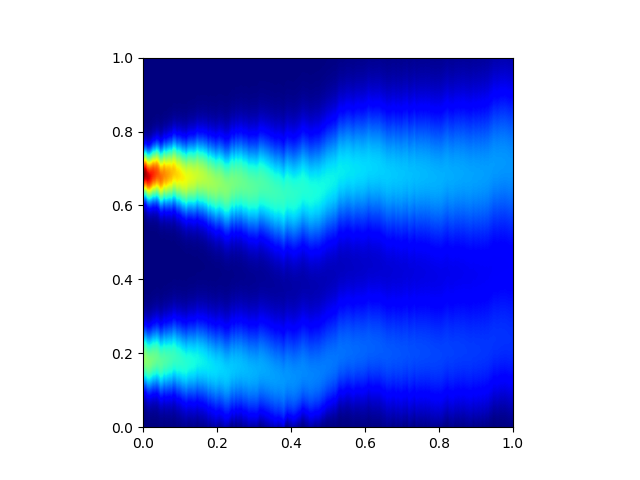} \hspace{-2.1cm}\vspace{-4pt}
		\caption{The figure shows two heat plots of the density $(t,x)\mapsto V_t(x)$ from Remark \ref{rem:DENSITY}, based on a numerical simulation of the SPDE for a fixed realization of $W^0$. On the left $\alpha=1.5$, while on the right $\alpha=0$. The other parameters are: $\rho=0.1$, $\sigma=1$, $b=0$, and $\mathfrak{K}$ is an isosceles triangle on $[0,0.015]$ with height $2/0.015$. The realization of $W^0$ starts on a slightly negative trend, but then moves back up again.}
		\label{Fig1}\vspace{-10pt}
	\end{center}
\end{figure} 
The left-hand plot in Figure \ref{Fig1} illustrates how contagion can cause a period of significant default clustering (from $t=0.2$ to $t=0.4$), with the right-hand plot confirming that the system would have done just fine in the absence of contagion. Indeed, the entire lower group of `unhealthy' banks default in the left plot, while the right plot shows a comfortable recovery from the slight initial deterioration (caused by the common exposures).

Notice also that the demise of the lower `unhealthy' part of the system causes a substantial drop in the distance-to-default of the `healthier' upper part. However, these problems do not result in further defaults, so the contagious effects die out. 
As illustrated by Figure \ref{Fig2} below, a sharper decline of the common exposures can prompt much more severe contagion, which in turn can result in a default cascade that also wipes out the healthier upper part (over a very short period of time). Finally, we emphasise the critical r\^ole of the common exposures as an instigator of such periods of high contagion. This is made clear by the rightmost plot of Figure \ref{Fig2}, which shows the system trending solidly upwards in a contrasting scenario where the common exposures are doing well.
\begin{figure}[H]\vspace{-9pt}
	\begin{center}
		\hspace{-2cm} \includegraphics[width=0.568\textwidth]{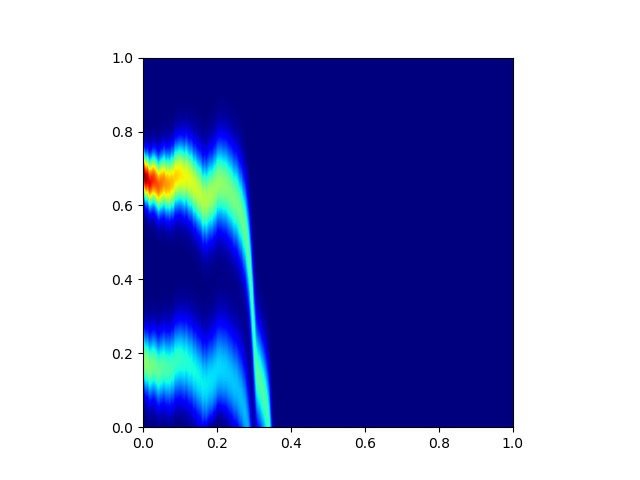} \hspace{-1.66cm} \includegraphics[width=0.65\textwidth]{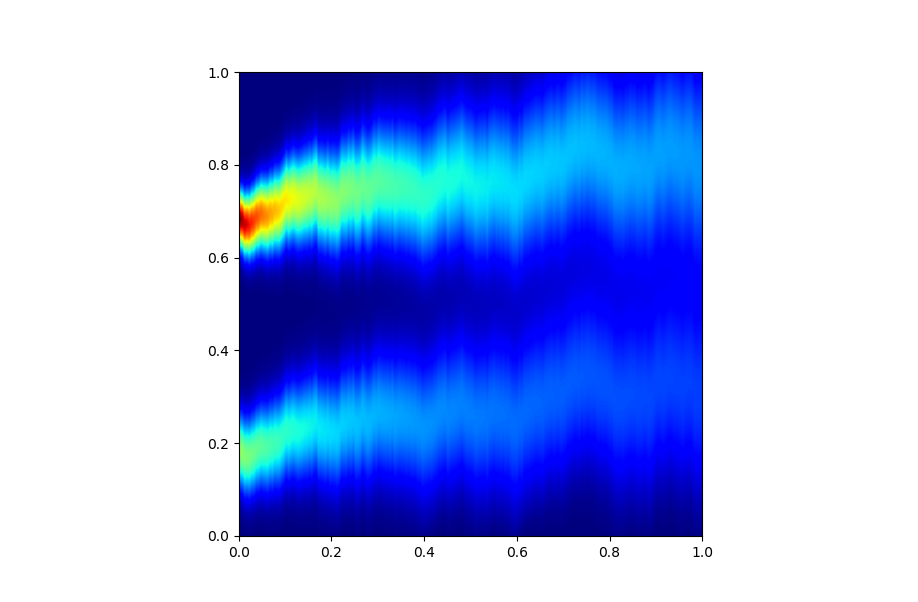} \hspace{-2.45cm}\vspace{-4pt}
		\caption{In these two heat plots of $(t,x)\mapsto V_t(x)$, we have $\alpha=2$ on both sides and otherwise the parameters are the same as in Fig.~\ref{Fig1}. However, the simulations are done for two different realizations of the common noise $W^0$: on the left it declines steadily, while on the right it increases correspondingly.}
		\label{Fig2}\vspace{-10pt}
	\end{center}
\end{figure}
The importance of the common noise (as portrayed by Fig.~\ref{Fig2}) agrees with observations
from the network-based literature, which suggest that idiosyncratic shocks are unlikely to significantly impact a large network, whereas
the addition of common shocks can generate substantial losses from
contagion (see e.g.~the discussion in Cont, Moussa \& Santos \cite{Cont Moussa}). Also, recalling the model from Section \ref{subsec:A-Base-Model},
we note that a high rate of herding can have the dual effect of producing a
healthier system in `normal' times, while serving as an amplifier of potential crises by causing even more default clustering if the common exposures decline significantly. Similarly, the correlation function presents another endogenous channel
that can amplify the effect of declining common exposures. We intend to return to a closer study of the interplay between the different parameters in future work.

\subsubsection{Default cascades and the limiting case of instantaneous contagion\label{subsec:Non-smoothed loss}}
In view of the steep decline in Figure \ref{Fig2}, it is interesting to consider what happens as the contagious impact of defaults becomes instantaneous. That is, as a sequence of impact kernels $(\mathfrak{K}_\varepsilon)$ approximates the dirac mass at $0$.

In the case of constant coefficients and no common noise, the resulting limit has recently been studied independently by Nadtochiy \& Shkolnikov \cite{Misha} and the authors of this paper together with Ledger \cite{HLS2018}, building on earlier work by Delarue et al.~\cite{delarue global solv,delarue approx}. To be specific, the limiting McKean--Vlasov problem  is of the form
\begin{equation}\label{eq:Non-smoothed McKean-Vlasov}
X_{t}=X_0+bt+\sigma W_{t}-\alpha L_{t}, \quad
L_{t}=\mathbb{P}(\tau\leq t),
\end{equation}
where $\tau=\inf\{t>0:X_{t}\leq0\}$. As it is, (\ref{eq:Non-smoothed McKean-Vlasov}) is ill-posed, however, it is conjectured (see Conj.~1.9 in \cite{HLS2018})  that it is well-posed in the class of `physical' solutions as introduced in \cite{delarue approx}. Global existence of a `physical' solution is known from \cite{delarue approx}, but uniqueness remains unsettled: If $\alpha$ is sufficiently large (given $X_0$), $t\mapsto L_t$ cannot be continuous \cite[Thm.~1.1]{HLS2018} and uniqueness is then only
known up to the first time the $\mathcal{W}^{1,2}$ norm of $ L$
explodes \cite[Thm.~1.8]{HLS2018}, see also \cite{Misha}. On the other hand, if $\alpha$
is sufficiently small, then it follows from \cite[Thm.~2.4]{delarue global solv} 
that there is a unique global solution such that $t\mapsto L_{t}$
is in $\mathcal{C}^1[0,T]$.

Mathematically, this phase transition in $\alpha$ is very interesting: it means that the steep decline in Figure \ref{Fig2} may degenerate into a jump discontinuity whereby a macroscopic part of the system is lost at the blink of an eye. Financially, such a jump could offer an idealized definition of a true `systemic default cascade', which is the approach adopted in  \cite{Misha} for a variant of (\ref{eq:Non-smoothed McKean-Vlasov}). However, beyond the benefits of a precise definition this may be too categorical, as instantaneous default cascades are not observed in practice and, from a systemic risk point of view, it is the default clustering and steep declines in distances-to-default that matter, not whether they materialized over short periods of time or as jumps. Therefore, we believe the framework proposed in this paper can serve as the reference model, with the instantaneous problem (\ref{eq:Non-smoothed McKean-Vlasov}) arising as an important  limiting case.

In addition, there are two theoretical advantages of the model in this paper. Firstly, it needs no extra notion of `physical' solutions and it makes sense as an SPDE globally. Secondly, the SPDE characterizes the unique limit of the finite system. This latter point is somewhat problematic for (\ref{eq:Non-smoothed McKean-Vlasov}), at least until the aforementioned conjecture is resolved, since global uniqueness is not known and the explosion time (up to which we have uniqueness) may in principle lie strictly before the first jump. In turn, even up to the first jump time, the finite system is not guaranteed to converge to a unique limit, and thus the jump definition of a `systemic default cascade' cannot be said to rigorously represent the finite financial system. This issue is even more pronounced with non-constant (nonlinear) coefficients and common noise, as nothing is then known about uniqueness.

\subsubsection{Large portfolio credit risk \label{subsec:Credit-derivatives}}

Although the focus here is on financial institutions, our framework
could also be applied to the study of default clustering in large portfolios
of more general defaultable entities. Viewed as a structural large portfolio model for pricing of credit derivatives, 
our framework extends \cite{HL2016,hambly reisinger etc}. In this regard, the idea of having a loss-dependent correlation
with finitely many discontinuities (Assumption \ref{Assumption 1 - finite system},
part iii) was considered in \cite{HL2016} as a possible way
to accommodate for the implied correlation skew across the
tranches of a CDO.

Recently, there has been a great
deal of interest in disentangling the rôles of contagion and common
risk factors as drivers of corporate default clustering in large portfolios (Azizpour, Giesecke \& Schwenkler \cite{Azizpour Giesescke Schwenkler},
Lando \& Nielsen \cite{Lando and Nielsen}, Duffie et al.~\cite{Duffie Frailty}).
However, the existing literature remains inconclusive and has focused almost exclusively on
self-exciting point processes, so the model we present here could serve as a first attempt towards a structural platform from
which to approach these matters.

\subsubsection{Connections to mathematical neuroscience\label{subsec:Connections-to-computational}}

Interestingly, our setup is closely related to nonlinear
\emph{leaky integrate-and-fire} models for electrically coupled neurons
with noisy input. These models can be phrased as  particle systems, where each SDE corresponds to the electrical potential of a neuron, and when this potential reaches a threshold voltage, the
neuron is then said to \emph{spike}, thus causing it to fire an electrical
signal to the other neurons exciting them to higher voltage
levels.

As suggested by Inglis \& Talay \cite{Inglis Talay}, the transmission
of this signal can be modelled by a \emph{cable equation}, which translates
to a gradual impact of the signal in complete analogy with our
contagion mechanism. However, instead of being absorbed at the boundary,
the neurons that spike are instantly reset to a predetermined
value (called the resting potential) and they then continue to evolve
according to this rule ad infinitum. In the mean field limit, this yields a McKean\textendash Vlasov
problem analogous to that of Theorem 2.7 except for the resetting
of the particle. So far, this has only been studied without common noise (and with simpler interactions) in which case the global well-posedness was proved in \cite{Inglis Talay}, building on ideas from \cite{delarue global solv,delarue approx}.
For further background, we refer to \cite{Inglis Talay,Brunel Hakim,Moreno-Bote and Parga}.

\section{The finite particle system\label{sec:The-particle-system}}

We begin by observing that the finite system (\ref{eq: Particle System}) is well-posed for each $N\geq1$.
To see this, we write $\mathbf{X}^{{\scriptscriptstyle N}}=(X^{1},\ldots,X^{{\scriptscriptstyle N}})$ and
express the system as a
vector-valued SDE
\[
d\mathbf{X}_{t}^{{\scriptscriptstyle N}}=\boldsymbol{b}(t,\mathbf{X}_{t}^{{\scriptscriptstyle N}}\!, \nu_t^{{\scriptscriptstyle N}})dt+\boldsymbol{\sigma}(t,\mathbf{X}_{t}^{{\scriptscriptstyle N}})\bigl(\rho_{t}d\mathbf{W}_{\!t}^{0}+(1\!-\!\rho_{t}^{2})^{\frac{1}{2}}d\mathbf{W}_{\!t}\bigr)-\boldsymbol{\alpha}(t,\mathbf{X}_{t}^{{\scriptscriptstyle N}}\!, \nu_t^{{\scriptscriptstyle N}})d\mathfrak{L}_{t}^{{\scriptscriptstyle N}}\!.
\]
\begin{rem}
[Shorthand notation]Here $\rho_t=\rho(t,\nu_{t}^{{\scriptscriptstyle N}})$. Similarly, we will sometimes write $\alpha^i_t=\alpha(t,X_t^i,\nu_{t}^{{\scriptscriptstyle N}})$ and $b_{t}^{i}=b(t,X_{t}^{i},\nu_{t}^{{\scriptscriptstyle N}})$ as well as $\sigma_{t}^{i}=\sigma(t,X_{t}^{i})$.
\end{rem}

Recall from (iii) of Assumption \ref{Assumption 1 - finite system} that the Lipschitzness in $L_t^{{\scriptscriptstyle N}}$ only holds piecewise. However, on each of the intervals between defaults, $L_{t}^{{\scriptscriptstyle N}}$ is simply equal to a fixed $\mathcal{F}_{0}$-measurable random variable with $L_{t}^{{\scriptscriptstyle N}}=0$ on $[0,\varsigma_{1})$
and then $L_{t}^{{\scriptscriptstyle N}}=\sum_{k=1}^{n}a_{i_{k}}^{{\scriptscriptstyle N}}$
on $[\varsigma_{n},\varsigma_{n+1})$, where $\varsigma_{n}$ is the
time of the $n$'th default (and $i_{1},\ldots,i_{n}$
have defaulted). Thus, on each of these intervals we can treat the coefficients as functions of just $t$ and the surviving members of $\mathbf{X}_t^{{\scriptscriptstyle N}}$, with (i)-(ii)  of Assumption \ref{Assumption 1 - finite system} giving local Lipschitzness in the euclidean norm. That is, we can solve the system inductively on each of the intervals
by running it up to the next default and then restarting it
with the new (fixed) value for $L_{t}^{{\scriptscriptstyle N}}$ and with
the defaulted particle removed from $\nu_{t}^{{\scriptscriptstyle N}}$. 
Hence the well-posedness follows by the standard theory for SDEs with locally Lipschitz coefficients of  at most linear growth.

\subsection{The finite-dimensional evolution equation}

Since we have sufficient symmetry in the coefficients of the particles, we can
obtain a single evolution equation for the dynamics of the empirical
measures. Furthermore, the assumptions on the weights $a_{i}^{{\scriptscriptstyle N}}$
ensure that enough averaging is taking place in order for the idiosyncratic
noise in this equation to vanish in the large population limit. These observations
are made more precise in the next proposition and the work that follows.
\begin{prop}
[Finite evolution equation]\label{prop: halfspace finite eqn}Given
$N\geq1$, it holds for all $\phi\in\mathscr{C}_{{\scriptscriptstyle 0}}$
that
\begin{align}
d\negthinspace\left\langle \nu_{t}^{{\scriptscriptstyle N}},\phi\right\rangle = \langle &\nu_{t}^{{\scriptscriptstyle N}},b(t,\cdot,\nu_{t}^{{\scriptscriptstyle N}})\partial_{x}\phi\rangle dt+{\textstyle \frac{1}{2}}\langle \nu_{t}^{{\scriptscriptstyle N}},\sigma(t,\cdot)^{2}\partial_{xx}\phi\rangle dt\label{eq:finite evol}\\
 & +\,\langle \nu_{t}^{{\scriptscriptstyle N}},\sigma(t,\cdot)\rho_{t}\partial_{x}\phi\rangle dW_{t}^{0}-\langle \nu_{t}^{{\scriptscriptstyle N}},\alpha_t\partial_{x}\phi\rangle d\mathfrak{L}_{t}^{{\scriptscriptstyle N}}+dI_{t}^{{\scriptscriptstyle N}}(\phi),\nonumber 
\end{align}
where the idiosyncratic driver\emph{ }$I_{t}^{{\scriptscriptstyle N}}$
satisfies
\[
\mathbb{E}\Bigl[\,\sup_{t\leq T}\left|I_{t}^{{\scriptscriptstyle N}}(\phi)\right|^{2}\Bigr]=O(1/N)\quad\text{as}\quad N\rightarrow\infty.
\]
\end{prop}

\begin{proof}
Notice that, since $\phi(0)=0$ for $\phi\in\mathscr{C}_{{\scriptscriptstyle 0}}$,
we have
\[
\left\langle \nu_{t}^{{\scriptscriptstyle N}},\phi\right\rangle =\sum_{i=1}^{N}a_{i}^{{\scriptscriptstyle N}}\mathbf{1}_{s<\tau_{i}}\phi(X_{t}^{i})=\sum_{i=1}^{N}a_{i}^{{\scriptscriptstyle N}}\phi(X_{t\land\tau_{i}}^{i}),\quad \text{for} \quad \phi\in\mathscr{C}_{{\scriptscriptstyle 0}}.
\]
Applying Itô's formula to $\phi(X_{t\land\tau_{i}}^{i})$, the first
result then follows with
\[
I_{t}^{{\scriptscriptstyle N}}(\phi):=\sum_{i=1}^{N}\int_{0}^{t}a_{i}^{{\scriptscriptstyle N}}\sigma(s,X_{s}^{i})(1-\rho(s,\nu_{s}^{{\scriptscriptstyle N}})^{2})^{\frac{1}{2}}\partial_{x}\phi(X_{s\land\tau_{i}}^{i})dW_{s}^{i}.
\]
Using the independence of the Brownian motions and the boundedness
of $\sigma$, we get
\begin{align*}
\mathbb{E}{\langle I_{\cdot}^{{\scriptscriptstyle N}}(\phi)\rangle{}}_{t} & =\sum_{i=1}^{N}\mathbb{E}\Bigl[\int_{0}^{t}\!(a_{i}^{{\scriptscriptstyle N}})^{2}(\sigma_{s}^{i})^{2}(1-\rho_{s}^{2})\partial_{x}\phi(X_{s\land\tau_{i}}^{i})^{2}ds\Bigr]\leq C\left\Vert \partial_{x}\phi\right\Vert _{\infty}^{2}\sum_{i=1}^{N}\mathbb{E}\bigl[(a_{i}^{{\scriptscriptstyle N}})^{2}\bigr].
\end{align*}
Thus, the claim follows by Doob's martingale inequality, since $a_{i}^{{\scriptscriptstyle N}}\leq C/N$.
\end{proof}

\subsection{Regularity properties of the particles \label{sec:Probabilistic-Estimates}}

In Section 4 we pass to the limit in the finite evolution
equation. However, first we need to ensure sufficient regularity at the level of the particles.
The cornerstone of this is a family of upper Dirichlet heat kernel type estimates for the densities of the particles.
\begin{prop}
[Density estimates]\label{prop: Aronson estimate}Let $X_{t}^{i}$
be given by (\ref{eq: Particle System}) under Assumption \ref{Assumption 1 - finite system}.
Then the absorbed process $(X_{t}^{i},t<\tau_{i})$
has a transition density $\mathfrak{p}_{t}^{i,{\scriptscriptstyle N}}$, which satisfies the following bounds: For
every $\epsilon>0$ there exist $\kappa\in(0,1)$ and constants $C,c>0$, uniformly in $N\geq1$ and $t\in (0,T]$, such that\vspace{-4pt}
\begin{equation}
\mathfrak{p}_{t}^{i,{\scriptscriptstyle N}}\hspace{-1pt}(x,y)\leq 
C\Bigl({\textstyle \frac{1}{\sqrt{t}}}({\textstyle \frac{x}{\sqrt{t}}}\land1)({\textstyle \frac{y}{\sqrt{t}}}\land1)+({\textstyle \frac{x^{\kappa}y^{\kappa}}{t^{\kappa}}}\land1)e^{\epsilon y^{2}}\Bigr)\bigl(1\land e^{-\frac{(x-y)^{2}}{ct}+c_{x\hspace{-1pt},\hspace{-1pt}y}}\bigr)
\label{eq:finite half-line density estimate}
\end{equation}
and
\begin{equation}
\mathfrak{p}_{t}^{i,{\scriptscriptstyle N}}\hspace{-1pt}(x,y)\leq C\bigl({\textstyle \frac{1}{\sqrt{t}}}+e^{\epsilon y^{2}}\bigr)e^{-\frac{(x-y)^{2}}{ct}},\label{eq:finite whole-space density estimate}
\end{equation}
where $c_{x,y}\apprle\left|x-y\right|\left(x\land y\right)$. Furthermore,
if $\sigma(t,x)=\sigma_{1}(t)\sigma_{2}(x)$, then $c_{x,y}\equiv0$. 
\end{prop}

\begin{proof}
The proof is postponed to Section 6. For the final proof, see Section
\ref{subsec:Proof-of-Prop3.2}.
\end{proof}
The above estimates provide critical control over the first moment
of the mass of $\nu^{{\scriptscriptstyle N}}$ near the boundary as
well as Gaussian decay towards infinity (see Corollary 3.4 below).
As we verify in Proposition \ref{prop:limit_regularity}, these
features carry over to the limit points of the particle system and
this will be essential to the proof of the energy estimates in Section 5.
\begin{cor}
[Regularity of the empirical measures]\label{Cor:Regularity of the empirical measures}The
empirical measures $\nu^{{\scriptscriptstyle N}}$ satisfy $\mathbb{E}\nu_{t}^{{\scriptscriptstyle N}}(a,b)\leq Ct^{-\frac{1}{2}}\left|b-a\right|$,
and it holds uniformly in $N\geq1$ and $t\in (0,T]$ that
\[
\begin{cases}
\exists\epsilon>0: & \mathbb{E}\nu_{t}^{{\scriptscriptstyle N}}(a,\infty)=O(\exp\left\{ -\epsilon a^{2}\right\} )\quad\text{as}\quad a\rightarrow\infty\\[2pt]
\exists\delta\in(0,1],\,\beta>0: & \mathbb{E}\nu_{t}^{{\scriptscriptstyle N}}(0,\varepsilon)=t^{-\frac{\delta}{2}}O(\varepsilon^{1+\beta})\quad\text{as}\quad\varepsilon\rightarrow0
\end{cases}
\]
\end{cor}

\begin{proof}
Recalling the definition of $\nu_{t}^{{\scriptscriptstyle N}}$ and
using that $a_{i}^{{\scriptscriptstyle N}}\leq C/N$, it follows
from Proposition \ref{prop: Aronson estimate} that, for any $(a,b)\subseteq\mathbb{R}_{+}$,
we have
\begin{align*}
\mathbb{E}\nu_{t}^{{\scriptscriptstyle N}}(a,b) & \leq\frac{C}{N}\sum_{i=1}^{N}\mathbb{P}_{\mathbf{}}\bigl(X_{t\land\tau_{i}}^{i}\in(a,b)\bigr)\leq C^{\prime}\int_{0}^{\infty}\hspace{-3pt}\int_{a}^{b}f_{t}(x,y)dxd\mu_{0}(y),
\end{align*}
where $f_{t}(x,y)$ can take the value of either of the right-hand
sides of (\ref{eq:finite half-line density estimate}) or (\ref{eq:finite whole-space density estimate}). 
Given (\ref{eq:finite whole-space density estimate}) and the sub-Gaussianity
of $\mu_{0}$, the first two claims are immediate. Similarly, the
final claim follows from (\ref{eq:finite half-line density estimate})
by exploiting the power law decay at the boundary for $t>0$.
\end{proof}

\subsection{Regularity properties of the loss process}

Below we present two important results about the limiting behaviour of the loss
process. The first result assures that, in the large population limit,
it is strictly increasing when there is mass
left in the system. This is crucial for the convergence to the limit
SPDE (Section \ref{subsec:The-limit-SPDE:}) and for uniqueness
(Section \ref{subsec:Uniqueness:-Proof-of}) as it ensures the loss process cannot get stuck at one of the coefficients'
finitely discontinuity points (see Assumption \ref{Assumption 1 - finite system}(iii)).
\begin{prop}
\label{prop: 1_loss increments}For any $t\in[0,T)$ and $h>0$, it
holds for all $r<1$ that
\[
\lim_{\delta\rightarrow0}\limsup_{N\rightarrow\infty}\mathbb{P}(L_{t+h}^{{\scriptscriptstyle N}}-L_{t}^{{\scriptscriptstyle N}}<\delta,L_{t}^{{\scriptscriptstyle N}}<r)=0.
\]
\end{prop}

\begin{proof}
See Section A.2 in the Appendix.
\end{proof}
While the system is progressively losing mass by the previous proposition, the next result ensures
that there cannot be too large losses in arbitrarily small amounts
of time. This is used in the tightness arguments below (Section \ref{sec:Tightness-and-Convergence}).
\begin{prop}
\label{prop: 2_loss increments}For every $t\in[0,T]$ and $\eta>0$,
we have
\[
\lim_{\delta\rightarrow0}\lim_{N\rightarrow\infty}\mathbb{P}\left(L_{t+\delta}^{{\scriptscriptstyle N}}-L_{t}^{{\scriptscriptstyle N}}\geq\eta\right)=0.
\]
\end{prop}

\begin{proof}
See Section A.2 in the Appendix.
\end{proof}
The above relies on the
impact kernel $\mathfrak{K}$ being in $\mathcal{W}^{1,1}$. As discussed in Section \ref{subsec:Non-smoothed loss}, instantaneous contagion can imply jumps in the loss process with positive probability.

\section{Tightness and convergence\label{sec:Tightness-and-Convergence}}

The aim of this section is to recover the SPDE (\ref{eq:LIMIT SPDE})
by passing to the limit in the finite evolution equation
(\ref{eq:finite evol}). To achieve this, we need to establish tightness of $(\nu^{\scriptscriptstyle N})$ and then we need some continuity results for the integrals in (\ref{eq:finite evol}). The first step towards tightness is to control the increments of the particles.
\begin{lem}
\label{lem:Fourth moments}For all $s,t\in[0,T]$, it holds uniformly in $N\geq1$ and $i=1,\ldots,N$ that
\[
\mathbb{E}\bigl[|X_{t\wedge\tau_{i}}^{i}-X_{s\wedge\tau_{i}}^{i}|^{4}\bigr]=O(|t-s|^2)\quad\text{as}\quad|t-s|\rightarrow0,
\]
where $\tau_i=\inf\{t>0:X_t^i\leq0\}$.
\end{lem}

\begin{proof}
Using the equation for $X_{t}^{i}$ with $dB_{t}^{i}=\sqrt{1-\rho_{t}^{2}}dW_{t}^{i}+\rho_{t}dW_{t}^{0}$,
we have
\[
\mathbb{E}\Bigl[\bigl|X_{t\wedge\tau_{i}}^{i}-X_{s\wedge\tau_{i}}^{i}\bigr|^{4}\Bigr]\apprle\mathbb{E}\Bigl[\Bigl(\int_{s}^{t}\sigma_{r}\mathbf{1}_{r<\tau_{i}}dB_{r}^{i}\Bigr)^{\!4}\Bigr]+\mathbb{E}\Bigl[\Bigl(\int_{s}^{t}\bigl|b^i_{r}-\alpha^i_r(\mathfrak{L}^{{\scriptscriptstyle N}})_{t}^{\prime}\bigr|dr\Bigr)^{\!4}\Bigr].
\]
Since $\sigma_{u}$ is bounded, the Burkholder\textendash Davis\textendash Gundy
inequality yields
\[
\mathbb{E}\Bigl[\Bigl(\int_{s}^{t}\sigma_{r}\mathbf{1}_{r<\tau_{i}}dB_{r}^{i}\Bigr)^{\!4}\Bigr]\apprle\mathbb{E}\bigl[\left\langle B_{\cdot}-B_{s}\right\rangle _{t}^{2}\bigr]=(t-s)^{2}.
\]
Letting $\Lambda_{r}^{i,{\scriptscriptstyle N}}:=|X_{t}^{i}|+\sum_{j=1}^{{\scriptscriptstyle N}}a_{j}^{{\scriptscriptstyle N}}|X_{t}^{j}|$,
the conditions in Assumption \ref{Assumption 1 - finite system}
imply the bound $|b^i_{r}| +|\alpha^i_{r}| \apprle1+\Lambda_{r}^{i,{\scriptscriptstyle N}}$.
Noting also that $(\mathfrak{L}^{{\scriptscriptstyle N}})^{\prime}\in L^{\infty}$,
Jensen's inequality gives
\begin{align*}
\mathbb{E}\Bigl[\Bigl(\int_{s}^{t}\bigl|b^i_{r}-\alpha^i_r(\mathfrak{L}^{{\scriptscriptstyle N}})_{r}^{\prime}\bigr|dr\Bigr)\Bigr] & \apprle\mathbb{E}\Bigl[\Bigl(\int_{s}^{t}(1+\Lambda_{r}^{i,{\scriptscriptstyle N}})dr\Bigr)^{\!4}\Bigr]\leq(t-s)^{4}\sup_{r\leq T}\mathbb{E}\bigl[(1+\Lambda_{r}^{i,{\scriptscriptstyle N}})^{4}\bigr],
\end{align*}
where last term is bounded uniformly, by (\ref{eq:finite whole-space density estimate})
of Proposition \ref{prop: Aronson estimate}.
\end{proof}
\begin{prop}[\textcolor{black}{Tightness}]
\textcolor{black}{\label{prop:First regularity prop}The sequence
$(\nu^{{\scriptscriptstyle N}},W^{0})$ is tight on $(D_{\mathscr{S}^{\prime}},\text{\emph{M1}})\times(C_{\mathbb{R}},\left\Vert \cdot\right\Vert _{\infty})$}
and any limit point $\nu^{*}$ is $\mathbf{M}_{\leq1}(\mathbb{R}_{+})$-valued.
Moreover, if we set $L_{t}^{*}:=1-\nu_{t}^{*}(0,\infty)$, then
$L^{*}$ is strictly increasing when $L^{*}<1$,
and \textcolor{black}{$(L^{{\scriptscriptstyle N}},W^{0})$ converges
weakly to $(L^{*},W^{0})$ on $(D_{\mathbb{R}},\text{\emph{M1}})\times(C_{\mathbb{R}},\left\Vert \cdot\right\Vert _{\infty})$
whenever $(\nu^{{\scriptscriptstyle N}},W^{0})$ converges weakly
to $(\nu^{*},W^{0})$.}
\end{prop}

\begin{proof}
For the tightness of $(\nu^{{\scriptscriptstyle N}},W^{0})$, it suffices
to show the M1 tightness of $\left\langle \nu^{{\scriptscriptstyle N}},\phi\right\rangle $
on $D_{\mathbb{R}}$ for every $\phi\in\mathscr{S}$, by Theorem 3.2
in \cite{Led2016}. To this end, we verify the sufficient conditions
(i) and (ii) from Theorem 12.12.3 in \cite{Whi2002}. The first condition
is trivial since $\left\langle \nu^{{\scriptscriptstyle N}},\phi\right\rangle $
is uniformly bounded by $\left|\left\langle \nu^{{\scriptscriptstyle N}},\phi\right\rangle \right|\leq\left\Vert \phi\right\Vert _{\infty}$
for all $N\geq1$.

For the second condition, we can consider the decomposition
\begin{equation}
\left\langle \nu_{t}^{{\scriptscriptstyle N}},\phi\right\rangle =\left\langle \hat{\nu}_{t}^{{\scriptscriptstyle N}},\phi\right\rangle -\phi(0)L_{t}^{{\scriptscriptstyle N}}\quad\text{with}\quad\hat{\nu}_{t}^{{\scriptscriptstyle N}}:={\textstyle \sum_{i=1}^{{\scriptscriptstyle N}}}a_{i}^{{\scriptscriptstyle N}}\delta_{X_{t\land\tau_{i}}^{i}}.\label{eq: DECOMPOSITION-1}
\end{equation}
The advantage of this is that the monotone part, $\phi(0)L_{t}^{{\scriptscriptstyle N}}$,
is immaterial to the M1 modulus of continuity. Indeed, by Propositions
4.1 and 4.2 of \cite{Led2016}, it is sufficient to verify that
\begin{equation}
\mathbb{E}\Bigl[\bigl|\left\langle \hat{\nu}_{t}^{{\scriptscriptstyle N}},\phi\right\rangle -\left\langle \hat{\nu}_{s}^{{\scriptscriptstyle N}},\phi\right\rangle \bigr|^{4}\Bigr]=O(|t-s|^{2})\quad\forall s,t\in[0,T]\label{eq:tight_1}
\end{equation}
and that, for every $\varepsilon>0$,
\begin{equation}
\lim_{\delta\rightarrow0}\lim_{N\rightarrow\infty}\mathbb{P}\Bigl(\sup_{t\in(0,\delta)}\bigl|\left\langle \nu_{t}^{{\scriptscriptstyle N}}-\nu_{0}^{{\scriptscriptstyle N}},\phi\right\rangle \bigr|+\sup_{t\in(T-\delta,T)}\bigl|\left\langle \nu_{T}^{{\scriptscriptstyle N}}-\nu_{t}^{{\scriptscriptstyle N}},\phi\right\rangle \bigr|>\varepsilon\Bigr)=0.\label{eq:tight_2}
\end{equation}
Recalling $a_{i}^{{\scriptscriptstyle N}}\leq C/N$, Jensen's inequality
and the Lipschitzness of $\phi\in\mathscr{S}$ imply
\[
\mathbb{E}\Bigl[\bigl|\left\langle \hat{\nu}_{t}^{{\scriptscriptstyle N}},\phi\right\rangle -\left\langle \hat{\nu}_{s}^{{\scriptscriptstyle N}},\phi\right\rangle \bigr|^{4}\Bigr]\leq\frac{C}{N}\left\Vert \phi\right\Vert _{\text{Lip}}\sum_{i=1}^{N}\mathbb{E}_{\mathbf{}}\Bigl[\bigl|X_{t\land\tau_{i}}^{i}-X_{s\land\tau_{i}}^{i}\bigr|^{4}\Bigr].
\]
Thus, we can conclude from Lemma \ref{lem:Fourth moments} that (\ref{eq:tight_1})
is satisfied. With regard to (\ref{eq:tight_2}), the decomposition
(\ref{eq: DECOMPOSITION-1}) yields
\[
\mathbb{P}\Bigl(\sup_{t\in(0,\delta)}\bigl|\left\langle \nu_{t}^{{\scriptscriptstyle N}}-\nu_{0}^{{\scriptscriptstyle N}},\phi\right\rangle \bigr|>\varepsilon\Bigr)\leq\mathbb{P}\Bigl(\sup_{t\in(0,\delta)}\bigl|\left\langle \hat{\nu}_{t}^{{\scriptscriptstyle N}}-\hat{\nu}_{0}^{{\scriptscriptstyle N}},\phi\right\rangle \bigr|\geq\frac{\varepsilon}{2}\Bigr)+\mathbb{P}\Bigl(|\phi(0)|L_{\delta}^{{\scriptscriptstyle N}}\geq\frac{\varepsilon}{2}\Bigr)
\]
and likewise for the supremum over $t\in(T-\delta,T)$. By Markov's
inequality and the same arguments as above, the first term vanishes uniformly in $N\geq1$ as $\delta\rightarrow0$. Combining this with
Proposition \ref{prop: 2_loss increments}, we deduce (\ref{eq:tight_2}).
Therefore, $(\nu^{{\scriptscriptstyle N}},W^{0})$ is tight, as desired.

For the final claims, Proposition 5.3 of \cite{HL2016} ensures that
each limit point $\nu_{t}^{*}$ can be recovered as an element of $\mathbf{M}_{\leq1}(\mathbb{R}_{+})$.
Moreover, using Propositions \ref{prop: 1_loss increments} and \ref{prop: 2_loss increments},
the claims about $L^{*}$ and $L^{{\scriptscriptstyle N}}$ follow
by arguing as in Propositions 5.5 and 5.6 of \cite{HL2016}.
\end{proof}
\begin{prop}[Limit point regularity]
\label{prop:limit_regularity}For every $\epsilon>0$ there exist
$\kappa\in(0,1)$ and $c>0$ such that any continuous
limit point $\nu^{*}$ of $(\nu^{{\scriptscriptstyle N}})$ satisfies\vspace{-2pt}
\[
\begin{cases}
\mathbb{E}\nu_{t}^{*}(a,b)\apprle {\textstyle \int_{0}^{\infty}\hspace{-2pt}\int_{a}^{b}}\bigl({\textstyle \frac{1}{\sqrt{t}}}({\textstyle \frac{x}{\sqrt{t}}}\land1)({\textstyle \frac{y}{\sqrt{t}}}\land1)+({\textstyle \frac{x^{\kappa}y^{\kappa}}{t^{\kappa}}}\land1)e^{\epsilon y^{2}}\bigr)\bigl(1\land e^{-\frac{(x-y)^{2}}{ct}+c_{x\hspace{-1pt},\hspace{-1pt}y}}\bigr)dxd\nu_{0}(y)\\[5pt]
\mathbb{E}\nu_{t}^{*}(a,b) \apprle {\textstyle \int_{0}^{\infty}\hspace{-2pt}\int_{a}^{b}}\bigl({\textstyle \frac{1}{\sqrt{t}}}+e^{\epsilon y^{2}}\bigr)e^{-\frac{(x-y)^{2}}{ct}}dxd\nu_{0}(y),
\end{cases}
\]
for any $(a,b)\subseteq\mathbb{R}_{+}$. Here $c_{x,y}\apprle\left|x-y\right|\left(x\land y\right)$,
and if $\sigma(t,x)=\sigma_{1}(t)\sigma_{2}(x)$, then $c_{x,y}\equiv0$.
In particular, $\nu^{*}$ satisfies Assumption \ref{Assumption 2.3}
with the pointwise properties
\[
\begin{cases}
\exists\epsilon>0: & \mathbb{E}\nu_{t}^{*}(\lambda,\infty)=O(\exp\left\{ -\epsilon \lambda^{2}\right\} )\quad\text{as}\quad \lambda\rightarrow\infty\\[3pt]
\exists\delta\in(0,1],\,\beta>0: & \mathbb{E}\nu_{t}^{*}(0,\varepsilon)=t^{-\frac{\delta}{2}}O(\varepsilon^{1+\beta})\quad\text{as}\quad\varepsilon\rightarrow0
\end{cases}
\]
\end{prop}

\begin{proof}
Fix an interval $(a,b)\subseteq\mathbb{R}_{+}$ and fix a small $\varepsilon>0$
s.t. $a-\varepsilon>0$. Let $\varphi_{\varepsilon}$ be a smooth
cut-off function in $\mathcal{C}_{c}^{\infty}(\mathbb{R};[0,1])$
equal to $1$ on $(a,b)$ and zero outside $(a-\varepsilon,b+\varepsilon)$.
Then, as in the proof of Corollary \ref{Cor:Regularity of the empirical measures},
we have
\begin{equation}
\mathbb{E}\left\langle \nu_{t}^{{\scriptscriptstyle N}},\varphi_{\varepsilon}\right\rangle \leq\mathbb{E}\nu_{t}^{{\scriptscriptstyle N}}(a-\varepsilon,b+\varepsilon)\leq C\int_{0}^{\infty}\int_{a-\varepsilon}^{b+\varepsilon}f_{t}(x,y)dxd\nu_{0}(y),
\label{eq:bound_limit_regualrity}
\end{equation}
where $f_{t}(x,y)$ can take the value of either of the right-hand
sides of (\ref{eq:finite half-line density estimate}) and (\ref{eq:finite whole-space density estimate}). Here we have also used that $\mu_0$ is dominated by a constant times $\nu_{0}=\lim_{{\scriptscriptstyle N}}\sum_{i=1}^{{\scriptscriptstyle N}}a_{i}^{{\scriptscriptstyle {\scriptscriptstyle N}}}\delta_{X_{0}^{i}}$, in light of (\ref{eq: a_i coefficients}).
Since $\varphi_{\varepsilon}\in\mathscr{S}$, after passing to a convergent
subsequence, $\left\langle \nu^{{\scriptscriptstyle N_{\hspace{-1pt}k}}},\varphi_{\varepsilon}\right\rangle $
converges weakly to $\left\langle \nu^{*},\varphi_{\varepsilon}\right\rangle $
on $D_{\mathbb{R}}$ by \cite[Prop.~2.7]{Led2016}. Thus, by the assumption that $\nu^{*}$ is continuous, \cite[Thm.~12.4.1]{Whi2002} gives that $\left\langle \nu_{t}^{{\scriptscriptstyle N}_{\hspace{-1pt}k}},\varphi_{\varepsilon}\right\rangle $
converges weakly to $\left\langle \nu_{t}^{*},\varphi_{\varepsilon}\right\rangle $
on $\mathbb{R}$ for all $t\in[0,T]$. Using bounded convergence, we get convergence
of the means and hence, for all $\varepsilon>0$,
\[
\mathbb{E}\nu_{t}^{*}(a,b)\leq\mathbb{E}\left\langle \nu_{t}^{*},\varphi_{\varepsilon}\right\rangle =\lim_{k\rightarrow\infty}\mathbb{E}\left\langle \nu_{t}^{{\scriptscriptstyle N_{\hspace{-1pt}k}}},\varphi_{\varepsilon}\right\rangle 
\]
Recalling (\ref{eq:bound_limit_regualrity})
and applying dominated convergence, the proof is complete. 
\end{proof}
\begin{rem}
\label{rem:limit_reg}If $\nu^{*}$ is not continuous, it still holds
that $\int_{0}^{T}\!\left\langle \nu_{t}^{{\scriptscriptstyle N}_{\hspace{-1pt}k}},\varphi_{\varepsilon}\right\rangle \!dt\Rightarrow\int_{0}^{T}\!\left\langle \nu_{t}^{*},\varphi_{\varepsilon}\right\rangle \!dt$,
by \cite[Thm.~11.5.1]{Whi2002}, where `$\Rightarrow$' denotes
weak convergence. Hence \textcolor{black}{Assumption \ref{Assumption 2.3}
holds for every limit point of $(\nu^{{\scriptscriptstyle N}})$ without
any a priori knowledge of continuity.}
\end{rem}

\subsection{Convergence of the mean process}

 Let $M_{t}^{{\scriptscriptstyle N}}:=\left\langle \nu_{t}^{{\scriptscriptstyle N}},\mbox{\ensuremath{\psi}}\right\rangle $ for any given Lipschitz function $\psi\in\text{Lip}(\mathbb{R})$. While we are mainly interested in the case $\psi=\text{Id}$, other applications may call for a general Lipschitz function, so we allow for this in the ensuing arguments.

Given $\nu^{{\scriptscriptstyle N}}\Rightarrow\nu^{*}$ on $(D_{\mathscr{S}'},\text{M1})$,
where `$\Rightarrow$' denotes weak convergence, we would like to
know that $M^{{\scriptscriptstyle N}}\Rightarrow M^{*}:=
 \langle\nu^{*},\psi\rangle $ on $(D_{\mathbb{R}},\text{M1})$.
However, $\psi$ need not be in $\mathscr{S}$, so we
cannot simply appeal to the continuity of the projection map from $(D_{\mathbb{\mathscr{S}}^{\prime}},\text{M1})$
to $(D_{\mathbb{R}},\text{M1})$. Nevertheless, we can work around
this obstacle by exploiting the
uniformly exponential tails of the empirical measures in expectation, cf.~Proposition \ref{prop: Aronson estimate}.

First of all, we can observe that $(M^{{\scriptscriptstyle N}})_{N\geq1}$
is tight on $(D_{\mathbb{R}},\text{M1})$. To see this, we verify
the necessary conditions (i) and (ii) from Theorem 12.12.3 in \cite{Whi2002}.
The first condition amounts to showing that, for all $\varepsilon>0$,
there exists $c>0$ such that
\[
\mathbb{P}\Bigl(\,\sup_{t\leq T}\left|M_{t}^{{\scriptscriptstyle N}}\right|>c\Bigr)\leq\varepsilon\quad\forall N\geq1.
\]
This property is immediate, since we even have uniformly sub-Gaussian tails of the
running max of $M_{t}^{{\scriptscriptstyle N}}$, by Corrollary \ref{cor: Sup is subgaussian}.
For the second condition, we can rely on the decomposition (\ref{eq: DECOMPOSITION-1}),
which amounts to
\[
M_{t}^{{\scriptscriptstyle N}}=\hat{M}_{t}^{{\scriptscriptstyle N}}-\psi(0)L_{t}^{{\scriptscriptstyle N}},\quad\hat{M}_{t}^{{\scriptscriptstyle N}}:=\left\langle \hat{\nu}_{t}^{{\scriptscriptstyle N}},\psi\right\rangle.
\]
Since $\psi$ is Lipschitz, the tightness then follows by repeating the same
arguments as in the proof of Proposition \ref{prop:First regularity prop}. Given this, we can now address the weak convergence.
\begin{prop}[Functional LLN for the mean]
	\label{Convergence_of_mean}
Suppose $(\nu^{{\scriptscriptstyle N}},W^{0})\Rightarrow (\nu^{*},W^{0})$. Then $(M^{{\scriptscriptstyle N}},W^{0}) \Rightarrow(M^{*},W^{0})$ on $(D_\mathbb{R},\emph{M1})\times(C_\mathbb{R},\left\Vert \cdot \right\Vert_\infty)$.
\end{prop}

\begin{proof}
Let $(M^{{\scriptscriptstyle N}_{\hspace{-1pt}k}},W^{0})_{k\geq1}$
be an arbitrary subsequence. By the tightness of $(M^{ {\scriptscriptstyle N}_{\hspace{-1pt}k}})$, we can pass to a further convergent subsequence, also
indexed by $k\geq1$. Let $(M^{\dagger},W^{0})$ denote the weak limit
of this subsequence and note that we still have $\nu^{{\scriptscriptstyle N}_{\hspace{-1pt}k}}\Rightarrow\nu^{*}$. We need to show that the laws of $(M^{*},W^{0})$ and $(M^{\dagger},W^{0})$ agree.
To this end, let $\phi_{\lambda}:=\varphi_{\lambda}\psi$, where $\varphi_{\lambda}\in\mathcal{C}_{c}^{\infty}(\mathbb{R};[0,1])$
is a standard cutoff function equal to $1$ on $[-\lambda,\lambda]$.
Then $\phi_{\lambda}\in\mathscr{S}(\mathbb{R})$ and $\phi_{\lambda}\rightarrow\psi$
pointwise as $\lambda\rightarrow\infty$. Since $\phi_{\lambda}\in\mathscr{S}$,
the projection map $\pi^{\phi_{\lambda}}:(D_{\mathbb{\mathscr{S}}^{\prime}},\text{M}1)\rightarrow(D_{\mathbb{R}},\text{M}1)$
is continuous (see~\cite[Prop.~2.7]{Led2016}), so we get
\begin{equation}
\left\langle \nu^{ {\scriptscriptstyle N}_{\hspace{-1pt}k}},\phi_{\lambda}\right\rangle \Rightarrow\left\langle \nu^{*},\phi_{\lambda}\right\rangle \quad\mbox{on}\quad(D_{\mathbb{R}},\text{M}1),\label{eq:a}
\end{equation}
by the continuous mapping theorem. Moreover, for all $\lambda>0$,
we have
\[
\mathbb{E}\left|\left\langle \nu_{t}^{{\scriptscriptstyle N}_{\hspace{-1pt}k}},\psi-\phi_{\lambda}\right\rangle \right|\leq\mathbb{E}\left\langle \nu_{t}^{ {\scriptscriptstyle N}_{\hspace{-1pt}k}},\left|\psi\right|\mathbf{1}_{[\lambda,\infty)}\right\rangle \apprle\mathbb{E}\left\langle \nu_{t}^{ {\scriptscriptstyle N}_{\hspace{-1pt}k}},(1+\left|x\right|)\mathbf{1}_{[\lambda,\infty)}\right\rangle ,
\]
and hence the exponential tails (in expectation)
imply
\begin{equation}
\mathbb{E}\left|\left\langle \nu_{t}^{{\scriptscriptstyle N}_{\hspace{-1pt}k}},\psi-\phi_{\lambda}\right\rangle \right|=o(1)\quad\mbox{as}\;\lambda\rightarrow\infty,\quad\text{uniformly}\;\text{in}\;k\geq1,\;t\in[0,T],\label{eq: o(1) tail}
\end{equation}
see Lemma \ref{lem:Tails for centre of mass}. Next, we define
\[
\mathbb{T}:=\mbox{cont}(M^{\dagger})\cap{\textstyle \bigcap_{\lambda\in\mathbb{N}}}\,\mbox{cont}(\left\langle \nu^{*},\phi_{\lambda}\right\rangle ),\quad\mbox{cont}(\xi)=\left\{ t\in[0,T]:\mathbb{P}(\xi_{t-}=\xi_{t})=1\right\} ,
\]
and note that this set is co-countable since $\left\langle \nu^{*},\phi_{\lambda}\right\rangle $
and $M^{\dagger}$ are in $D_{\mathbb{R}}$ (see e.g.~\cite[Cor.~12.2.1]{Whi2002}). Given this, we define $\widetilde{\mathbb{T}}:=\mathbb{T}\cap\left\{ t\in[0,T]:\mathbb{P}(M_{t}^{*}<\infty)=1\right\}$, 
where the latter set has full Lebesgue measure, so $\widetilde{\mathbb{T}}$
is dense in $[0,T]$. We will show that all the finite dimensional
distributions of $(M^{*},W^{0})$ and $(M^{\dagger},W^{0})$ agree
for indices $\left\{ t_{1},\ldots,t_{l}\right\} \in\widetilde{\mathbb{T}}$.
By a monotone class argument, it suffices to show that
\begin{equation}
\mathbb{E}\biggl[\prod_{i=1}^{l}f_{i}(M_{t_{i}}^{*})g_{i}(W_{t_{i}}^{0})\biggr]=\lim_{k\rightarrow\infty}\mathbb{E}\biggl[\prod_{i=1}^{l}f_{i}(M_{t_{i}}^{{\scriptscriptstyle N}_{\hspace{-1pt}k}})g_{i}(W_{t_{i}}^{0})\biggr]\label{eq:to_be_proved}
\end{equation}
for bounded functions $f_{i},g_{i}\in\mbox{Lip}(\mathbb{R};\mathbb{R}_{+})$, $i=1,\ldots,l$, where we have used convergence of the marginals $(M_{t_i}^{{\scriptscriptstyle N}_{\hspace{-1pt}k}},W_{t_i}^0)_{i=1,\ldots,l}\Rightarrow(M_{t_i}^{\dagger},W_{t_i}^0)_{i=1,\ldots,l}$,
see \cite[Thm.~11.6.6]{Whi2002}.

Recalling $M_{t_{i}}^{ {\scriptscriptstyle N}_{\hspace{-1pt}k}}=\langle \nu_{t_{i}}^{{\scriptscriptstyle N}_{\hspace{-1pt}k}},\psi\rangle $,
we can observe that
\begin{equation}
\bigl|f_{i}(M_{t_{i}}^{ {\scriptscriptstyle N}_{\hspace{-1pt}k}})-f_{i}\bigl(\left\langle \nu_{t_{i}}^{{\scriptscriptstyle N}_{\hspace{-1pt}k}},\phi_{\lambda}\right\rangle \bigr)\bigr|\leq\left\Vert f_{i}\right\Vert _{\text{Lip}}\left|\left\langle \nu_{t_{i}}^{{\scriptscriptstyle N}_{\hspace{-1pt}k}},\psi-\phi_{\lambda}\right\rangle \right|,\quad i=1,\ldots,l.\label{eq:bounds on f(M)}
\end{equation}
Using the upper bounds on each $f_{i}(M_{t_{i}}^{{\scriptscriptstyle {\scriptscriptstyle N}_{\hspace{-1pt}k}}})$,
as implied by (\ref{eq:bounds on f(M)}), it follows from (\ref{eq: o(1) tail})
and the boundedness of $f_{i},g_{i}\geq0$, that 
\begin{equation}
\mathbb{E}\left[\prod_{i=1}^{l}f_{i}(M_{t_{i}}^{{\scriptscriptstyle N}_{\hspace{-1pt}k}})g_{i}(W_{t_{i}}^{0})\right]\leq\mathbb{E}\left[\prod_{i=1}^{l}f_{i}\left(\left\langle \nu_{t_{i}}^{{\scriptscriptstyle N}_{\hspace{-1pt}k}},\phi_{\lambda}\right\rangle \right)g_{i}(W_{t_{i}}^{0})\right]+\left\{ o(1)\;\mbox{terms}\right\} ,\label{eq: first in equaity}
\end{equation}
as $\lambda\rightarrow\infty$, where the $o(1)$ terms are uniform
in $k\geq1$. By construction of $\mathbb{\widetilde{T}}$, \cite[Thm.~12.4.1]{Whi2002} implies that the canonical projections $\pi_{t_{i}}:(D_{\mathbb{R}},\text{M}1)\rightarrow\mathbb{R}$
are continuous at each $\left\langle \nu^{{\scriptscriptstyle N}_{\hspace{-1pt}k}},\phi_{\lambda}\right\rangle $
for $\lambda\in\mathbb{N}$. Thus, with $\langle \nu_{t_{i}}^{{\scriptscriptstyle N}_{\hspace{-1pt}k}},\phi_{\lambda}\rangle =\pi_{t_{i}}\!\left\langle \nu^{ {\scriptscriptstyle N}_{\hspace{-1pt}k}},\phi_{\lambda}\right\rangle $,
the weak convergence (\ref{eq:a}) implies that, by taking a $\limsup$
over $k\geq1$ on both sides of (\ref{eq: first in equaity}),
\[
\lim_{k\rightarrow\infty}\mathbb{E}\left[\prod_{i=1}^{l}f_{i}(M_{t_{i}}^{{\scriptscriptstyle N}_{\hspace{-1pt}k}})g_{i}(W_{t_{i}}^{0})\right]\leq\mathbb{E}\left[\prod_{i=1}^{l}f_{i}\left(\left\langle \nu_{t_{i}}^{*},\phi_{\lambda}\right\rangle \right)g_{i}(W_{t}^{0})\right]+o(1)\quad\text{as}\;\lambda\rightarrow\infty.
\]
Finally, we note that $\left\langle \nu_{t}^{*},\phi_{\lambda}\right\rangle \rightarrow\left\langle \nu_{t}^{*},\psi\right\rangle =M_{t}^{*}$
as $\lambda\rightarrow\infty$, by dominated convergence, since $\left|\phi_{\lambda}\right|\leq\left|\psi\right|$
and $\psi$ is integrable with respect to $\nu_{t}^{*}$ for $t\in\widetilde{\mathbb{T}}$.
Consequently, by the boundedness and continuity of each $(f_{i},g_{i})$,
we can send $\lambda\rightarrow\infty$ to find that
\[
\lim_{k\rightarrow\infty}\mathbb{E}\left[\prod_{i=1}^{l}f_{i}(M_{t_{i}}^{{\scriptscriptstyle N}_{\hspace{-1pt}k}})g_{i}(W_{t_{i}}^{0})\right]\leq\mathbb{E}\left[\prod_{i=1}^{l}f_{i}\left(M_{t_{i}}^{*}\right)g_{i}(W_{t_{i}}^{0})\right].
\]
Now, if we had instead relied on the lower bounds for each $f_{i}(M_{t_{i}}^{{\scriptscriptstyle {\scriptscriptstyle N}_{\hspace{-1pt}k}}})$
implied by (\ref{eq:bounds on f(M)}), then the analogous arguments
would yield the reverse inequality, thus proving (\ref{eq:to_be_proved}).
\end{proof}

\subsection{The limit SPDE \textemdash{} Proof of Theorem \ref{thm:EXISTENCE}\label{subsec:The-limit-SPDE:}}
By \cite[Thm.~3,2]{Led2016}, the tightness from Proposition \ref{prop:First regularity prop} implies relative compactness in the sense that every subsequence of $(\nu^{{\scriptscriptstyle N}},W^0)$ has a weakly convergent subsequence. Using this, we show that the corresponding integrals in the finite evolution equation (\ref{eq:finite evol}) converge weakly (Proposition \ref{prop:Weak convergence of integrals}) and then we employ a martingale argument (Proposition \ref{prop: Maringale approach}) to see that this gives rise to the desired limit SPDE.

\begin{rem}[Skorokhod's representation theorem]\label{Skorokhod_Rep} While  $(D_{\mathscr{S}'},\text{M1})$ fails to be Polish, it is a completely regular Suslin space. In particular, condition (10) of \cite[Thm.~2]{Jakubowski} holds and hence any weakly convergent subsequence of  $(\nu^{\scriptscriptstyle N})$ has a further subsequence with the usual a.s.~Skorokhod representation property. Firstly, $(D_{\mathscr{S}'},\text{M1})$  is  completely regular and Hausdorff by \cite[Prop.~2.7]{Led2016} and \cite[Thm.~2.1.1]{Kallianpur_Xiong}. Secondly, \cite[Prop.~2.7]{Led2016} and its proof gives $D_{\mathscr{S}'}=\bigcup_{n=1}^{\infty}D_{S_{\text{-}n}}$ where each $(D_{S_{\text{-}n}},\text{M1})$ is Polish and the inclusions $D_{S_{\text{-}n}}\!\hookrightarrow D_{\mathscr{S}'}$ are M1-continuous. Thus, each $D_{S_{\text{-}n}}$ is Suslin in the Hausdorff space $(D_{\mathscr{S}'},\text{M1})$ and so  their union is Suslin under the finest topology such that the inclusions are continuous \cite[Thm.~I.II.3]{Schwartz}. As the M1-topology is coarser, we conclude that $(D_{\mathscr{S}'},\text{M1})$ is Suslin.
\end{rem}

\begin{prop}[Convergence of integrals]
\label{prop:Weak convergence of integrals}Fix $t\leq T$
and $\phi\in\mathscr{S}$, and define
\[
\Psi,\Phi:\left\{ \zeta\in D_{\mathscr{S}^{\prime}}:\zeta_{s}\in\mathbf{M}_{\leq1}(\mathbb{R})\right\} \times  D_{\mathbb{R}}\rightarrow\mathbb{R}\quad\text{by}
\]
\[
\Psi(\zeta,\ell):=\int_{0}^{t}\left\langle \zeta_{s},g(s,\cdot,\zeta_s,\ell_{s})\phi\right\rangle ds,\;\;\Phi(\zeta,\ell):=\int_{0}^{t}\left\langle \zeta_{s},\alpha(s,\cdot,\zeta_s,\ell_{s})\phi\right\rangle d(\mathfrak{K}\ast\ell)_{s},
\]
where $g$ is any of $b$, $\sigma^{2}$, or $\sigma\rho$. Given $(\nu^{{\scriptscriptstyle N}},L^{{\scriptscriptstyle N}})\Rightarrow(\nu^{*},L^{*})$ on $(D_{\mathscr{S}'},\emph{M1})\times(D_\mathbb{R},\emph{M1})$, $\Psi^{{\scriptscriptstyle N}}:=\Psi(\nu^{{\scriptscriptstyle N}},L^{{\scriptscriptstyle N}})$ converges weakly to $\Psi^{*}:=\Psi(\nu^{*},L^{*})$, on $\mathbb{R}$, and likewise for $\Phi$.
\end{prop}

\begin{proof}
Fix a bounded $f\in\text{Lip}(\mathbb{R})$. By Remark \ref{Skorokhod_Rep} any subsequence of $(\nu^{{\scriptscriptstyle N}},L^{{\scriptscriptstyle N}})$ has a further subsequence, also indexed by $N$, for which we can assume almost sure convergence. By applying the triangle inequality, we have
\begin{align*}
\bigl|\mathbb{E}f(\Psi^{{\scriptscriptstyle N}})-&\mathbb{E}f(\Psi^{*})\bigr|
\apprle \mathbb{E} \,\Bigl|\int_{0}^{t}\bigl\langle \nu^{*}_{s}-\nu^{{\scriptscriptstyle N}}_{s},g(s,\cdot,\nu^*_s,L^*_{s})\phi\bigr\rangle ds\Bigr|\\
&  +\mathbb{E}\,\Bigl|\int_{0}^{t}\bigl\langle \nu^{{\scriptscriptstyle N}}_{s},g(s,\cdot,\nu^{*}_s,L^{{\scriptscriptstyle N}}_{s})\phi-g(s,\cdot,\nu^{{\scriptscriptstyle N}}_{s},L^{{\scriptscriptstyle N}}_{s})\phi\bigr\rangle ds\Bigr|\\
&  +\mathbb{E}\,\Bigl|\int_{0}^{t}\bigl\langle \nu^{{\scriptscriptstyle N}}_{s},g(s,\cdot,\nu^*_s,L^*_{s})\phi-g(s,\cdot,\nu^{*}_s,L^{{\scriptscriptstyle N}}_{s})\phi\bigr\rangle ds\Bigr|\\
& =: \; \mathbb{E}I_{1}^{{\scriptscriptstyle N}}+\mathbb{E}I_{2}^{{\scriptscriptstyle N}}+\mathbb{E}I_{3}^{{\scriptscriptstyle N}}.
\end{align*}
Starting with $I_{1}$, fix $\delta>0$ and recall $|g|\leq|x|+C(\nu^*)$ with $C(\nu^*)\apprle1+\sup_{s\leq t}|M^*_{s}|$. Hence, using that $\phi\in\mathscr{S}$, we can take $\lambda=\lambda(\delta)$
sufficiently large so that
\begin{equation}
\sup_{x\in\mathbb{R}\setminus[-\lambda,\lambda]}\left|g(s,x,\nu^*_s,L^*_{s})\phi(x)\right|<(1+C(\nu^*))\delta/2 \quad\forall s\in[0,T].
\label{eq: b*phi < delta}
\end{equation}
Now take a family of mollifiers $\psi_{\varepsilon}\in {\mathcal{C}{}}^{\infty}_c(\mathbb{R})$ and consider the (random) mollifications
\[
g_{s}^{\varepsilon}(x):=\bigl(g(s,\cdot,\nu^*_s,L^*_{s})\ast\psi_{\varepsilon}\bigr)(x), \quad \varepsilon>0.
\]
As $\nu^*$ and $\nu^{\scriptscriptstyle N}$ are sub-probability measures,
we have
\[
I_{1}\leq\Bigl|\int_{0}^{t}\bigl\langle \nu^{*}_{s}\!-\nu^{{\scriptscriptstyle N}}_{s},g^{\varepsilon}_s\phi\bigr\rangle ds\Bigr|+2\left\Vert \phi\right\Vert _{\infty}\!\int_{0}^{t}\!\sup_{|x|\leq\lambda}\left|g_{s}^{\varepsilon}-g_{s}\right|ds+(1+C(\nu^*))\delta t.
\]
Since $g_{s}^{\varepsilon}\rightarrow g_{s}$ uniformly
in $x\in[\lambda,\lambda]$, dominated convergence implies that, for $\varepsilon$ sufficiently small, the second term is less than $\delta$ in expectation. Likewise,  $g_{s}^{\varepsilon}\phi\in\mathscr{S}$ ensures that, for $N$ large, the first term is less than $\delta$ in expectation. Hence
  $\mathbb{E}I_1^{\scriptscriptstyle N} \leq C (\delta+\delta + \delta t)$ for large enough $N$, where $C$ only depends on $\mathbb{E}C(\nu^*)$, so $\mathbb{E}I_1^{\scriptscriptstyle N}\rightarrow0$ as $N\rightarrow\infty$.

	For $I_{2}^{\scriptscriptstyle N}$, the local Lipschitzness of $g$, together with $\phi \in \mathscr{S}$ and $\nu^{\scriptscriptstyle N}_s\in \mathbf{M}_{\leq1}$, gives
	\[
	I_{2}^{\scriptscriptstyle N}\leq C(\nu^*) \!\!\int_{0}^{t} \sup \{ | \langle \nu_s^* - \nu_s^{\scriptscriptstyle N} , \psi \rangle | : \psi\in \mathcal{C}_{d_0} \}ds,
	\]
	where $\mathcal{C}_{d_0}:= \{\psi\in\mathcal{C}(\mathbb{R}):\left\Vert \psi\right\Vert_\text{Lip}\leq1, |\psi(x)|\leq 1+|x|\}$, and $C(\nu^*)\apprle1+\sup_{s\leq t}|M^*_s|$. Fix $\delta>0$ and take $\lambda=\lambda(\delta)$ large (to be determined later). By the Arzel\`a--Ascoli theorem, there is a finite family $\psi_1,\ldots,\psi_k\in \mathcal{C}_{d_0}$ supported in $[-\lambda,\lambda]$ so that, for each $\psi \in \mathcal{C}_{d_0}$,
	\[
	\sup\{|\psi(x)-\psi_i(x)| :x\in[-\lambda,\lambda]\}<\delta/2
	\]
	for some $i\in\{1,\ldots,k\}$. Fixing any $\psi \in \mathcal{C}_{d_0}(\mathbb{R})$ and the corresponding $\psi_i$, we have
	\[
	|\langle \nu^*\!-\nu^{\scriptscriptstyle N},\psi\rangle|
	\leq \Bigl|\int_{0}^{\lambda} \! (\psi-\psi_i) d(\nu^*\!-\nu^{\scriptscriptstyle N})\Bigr|
	+ \Bigl|\int_{\lambda}^{\infty} \!\!(\psi-\psi_i) d(\nu^*\!-\nu^{\scriptscriptstyle N})\Bigr|
	+ |\langle \nu^*\!-\nu^{\scriptscriptstyle N},\psi_i\rangle|.
	\]
	By construction, the first term is bounded by $\delta$ uniformly in $N\geq1$. Moreover, by Cauchy--Schwartz and Jensen's inequality, we have
	\[
	\mathbb{E}\Bigl[C(\nu^*)\!\!\int_0^t\Bigl|\int_{\lambda}^{\infty} \!\!(\psi-\psi_i) d(\nu^*_s\!-\nu^{\scriptscriptstyle N}_s)\Bigr|ds\Bigr] \leq C\mathbb{E}\!\int_0^t\! \langle \nu^*_s+\nu^{\scriptscriptstyle N}_s,\lvert\cdot\rvert^2\textbf{1}_{[\lambda,\infty)}\rangle ds,
	\]
and hence
	\[
	\mathbb{E}I_{2}^{\scriptscriptstyle N} \leq C\delta t + 
	 C\mathbb{E}\!\int_0^t\! \langle \nu^*_s + \nu^{\scriptscriptstyle N}_s,\lvert\cdot\rvert^2\textbf{1}_{[\lambda,\infty)}\rangle ds + C\sup_{i=1,\ldots,k} \mathbb{E}\int_0^t|\langle \nu_s^*\!-\nu_s^{\scriptscriptstyle N},\psi_i\rangle|ds.
	\]
	By Corollary \ref{Cor:Regularity of the empirical measures} and Proposition \ref{prop:limit_regularity}, Lemma \ref{lem:Tails for centre of mass} gives that the middle term vanishes uniformly in $N\geq1$ as $\lambda\rightarrow\infty$, so we can take $\lambda$ large enough so that it is uniformly bounded by $\delta$. For the final term, recall that $\psi_i$ has compact support in $[-\lambda,\lambda]$, so we can use the same mollification argument as for $I_1^{\scriptscriptstyle N}$. Since there are only finitely many $\psi_i$'s to consider, we can thus take $N$ sufficiently large such that $\mathbb{E}I_2^{\scriptscriptstyle N}\leq C(\delta t +\delta+ 2\delta )$,
	where $C$ is a fixed numerical constant.
	This proves that $\mathbb{E}I_2^{\scriptscriptstyle N}$ vanishes as $N \rightarrow \infty$.

Finally, we consider the last integral $I_3^{\scriptscriptstyle N}$. By (iii) of Assumption \ref{Assumption 1 - finite system}, we have
\begin{equation}\label{eq:I_3^N}
\bigl|\bigl\langle \nu^{{\scriptscriptstyle N}}_{s},g(s,\cdot,\nu^*_s,\ell^*_{s})\phi-g(s,\cdot,\nu^{*}_s,\ell^{{\scriptscriptstyle N}}_{s})\phi\bigr\rangle \bigr| \apprle C(\nu^{*}) |\ell_s^*-\ell_s^{\scriptscriptstyle N}|
\end{equation}
whenever $\ell_s^*,\ell_s^{\scriptscriptstyle N}\in [\theta_{i-1},\theta_{i})$ for some $i=1,\ldots,k$. Let $\{\ell^*,(\ell^{\scriptscriptstyle N})_{\scriptscriptstyle N \geq1}\}$ represent a fixed realization of $\{L^*, (L^{\scriptscriptstyle N})_{\scriptscriptstyle N \geq1}\}$. Then $\ell^{\scriptscriptstyle N}\rightarrow\ell^*$ in $(D_\mathbb{R},\text{M1})$ and hence $\ell^{\scriptscriptstyle N}_{s}\rightarrow\ell^*_{s}$ for any fixed $s\in\left\{ r\in[0,t]:\ell^*_{r-}=\ell^*_{r}\right\} $. Fix $\varepsilon>0$. Since $\ell^*$ is
strictly increasing (by Assumption \ref{Assumption 2.3}), we can
take $\delta=\delta(\varepsilon)$ small so that $\text{Leb}\bigl(\left\{ r\in[0,T]:|\theta_{i}-\ell^*_{r}|<\delta\;\text{for}\;\text{some}\;i\right\} \bigr)\leq\varepsilon$.
On the other hand, if $|\theta_{i}-\ell^*_{s}|\geq\delta$ for all
$i$, then we eventually have $\ell^*_{s},\ell^{\scriptscriptstyle N}_{s}\in[\theta_{i-1},\theta_{i})$
for some $i$, so (\ref{eq:I_3^N}) applies.
Thus, for the given realization of the randomness,
\[
\limsup_{N\rightarrow\infty} I_3^{{\scriptscriptstyle N}} \leq
C'\!(\nu^*) \int_0^t \mathbf{1}_{\left\{ r\in[0,T]\,:\,|\theta_{i}-\ell^*_{r}|<\delta\;\text{for}\;\text{some}\;i\right\} }(s)ds \leq C'\!(\nu^*) \varepsilon.
\]
As $\varepsilon>0$ was arbitrary, we deduce that $\lim_{N} I_3^{{\scriptscriptstyle N}}=0$ almost surely. Noting the uniformity of the bound in (\ref{eq:I_3^N}), dominated convergence gives $\mathbb{E}I_3^{{\scriptscriptstyle N}}\rightarrow0$ as $N\rightarrow \infty$.

It remains to prove $\mathbb{E}f(\Phi^{\scriptscriptstyle N})\rightarrow\mathbb{E}f(\Phi^*)$ as $N\rightarrow\infty$. To this end, the main points are simply that
$|(\mathfrak{K}^{\prime}\ast L^*)_{s}|\leq\left\Vert \mathfrak{K}^{\prime}\right\Vert _{1}$ and $|\int_0^t (\mathfrak{K}^{\prime}\ast (L^*-L^{\scriptscriptstyle N}))_{s}ds|\leq\left\Vert \mathfrak{K}^{\prime}\right\Vert _{1}\int_0^t|L^*_s-L^{\scriptscriptstyle N}_s|ds$.
Using these observations, the arguments are the same as for $\Psi$.
\end{proof}

\begin{prop}[Martingale argument]
\label{prop: Maringale approach}
Fix an arbitrary $\phi\in\mathscr{C}_{{\scriptscriptstyle 0}}$ and define, for all $(\zeta,\ell,w)\in D_{\mathscr{S}^{\prime}}\times D_{\mathbb{R}}\times C_{\mathbb{R}}$,
the $D_{\mathbb{R}}$-processes\vspace{-2pt}
\begin{align*}
\mathcal{M}_{t}(\zeta,\ell) &:=\left\langle \zeta_{t},\phi\right\rangle -\left\langle \nu_{0},\phi\right\rangle -{\textstyle \int_{0}^{t}}\left\langle \zeta_{s},b(s,\cdot,\zeta_s,\ell_{s})\partial_{x}\phi\right\rangle ds\\
 +&\,{\textstyle \frac{1}{2}}{\textstyle \int_{0}^{t}}\left\langle \zeta_{s},\sigma(s,\cdot)^{2}\partial_{xx}\phi\right\rangle ds-{\textstyle \int_{0}^{t}}\left\langle \zeta_{s},\alpha(s,\cdot,\zeta_s,\ell_{s})\partial_{x}\phi\right\rangle d(\mathfrak{K}\ast\ell)_{s},\\
\mathcal{N}_{t}(\zeta,\ell) &:=\mathcal{M}_{t}(\zeta,\ell)^{2}-{\textstyle \int_{0}^{t}}\left\langle \zeta_{s},\sigma(s,\cdot)\rho(s,\zeta_{s},\ell_{s})\partial_{x}\phi\right\rangle ^{2}ds,\\
\mathcal{K}_{t}(\zeta,\ell,w)& :=\mathcal{M}_{t}(\zeta,\ell)\cdot w_{t}-{\textstyle \int_{0}^{t}}\left\langle \zeta_{s},\sigma(s,\cdot)\rho(s,\zeta_s,\ell_{s})\partial_{x}\phi\right\rangle ds.
\end{align*}
If $(\nu^{{\scriptscriptstyle N}},L^{{\scriptscriptstyle N}},W^0) \Rightarrow (\nu^{*},L^{*},W^0)$, then $\mathcal{M}(\nu^{*},L^{*})$,
$\mathcal{N}(\nu^{*},L^{*})$, and $\mathcal{K}(\nu^{*},L^{*},W^0)$
are all continuous martingales.
\end{prop}

\begin{proof}
Let $\mathcal{M}_{t}^{{\scriptscriptstyle N}}:=\mathcal{M}_{t}(\nu^{{\scriptscriptstyle N}},L^{{\scriptscriptstyle N}})$
and $\mathcal{M}_{t}^{*}:=\mathcal{M}_{t}(\nu^{*},L^{*})$.
Now fix $s,t\in[0,T]$ with $s<t$ and define, for any $s_{1},\ldots,s_{n}\in[0,s]$,
\[
F_{\!\mathcal{{\scriptscriptstyle M}}}^{{\scriptscriptstyle N}}:=\bigl(\mathcal{M}_{t}^{{\scriptscriptstyle N}}-\mathcal{M}_{s}^{{\scriptscriptstyle N}}\bigr)\prod_{i=1}^{n}f_{i}(\mathcal{M}_{s_{i}}^{{\scriptscriptstyle N}})\quad\text{and}\quad F_{\!\mathcal{{\scriptscriptstyle M}}}^{*}:=\bigl(\mathcal{M}_{t}^{*}-\mathcal{M}_{s}^{*}\bigr)\prod_{i=1}^{n}f_{i}(\mathcal{M}_{s_{i}}^{*}),
\]
where $f_{1},\ldots,f_{n}\in C_{b}(\mathbb{R})$ are arbitrary. Proceeding
analogously for $\mathcal{N}$ and $\mathcal{K}$, it
follows from Proposition \ref{prop:Weak convergence of integrals}
and the continuous mapping theorem that
\[
F_{\!\mathcal{{\scriptscriptstyle M}}}^{{\scriptscriptstyle N}}\Rightarrow F_{\!\mathcal{{\scriptscriptstyle M}}}^{*},\quad F_{\!\mathcal{{\scriptscriptstyle N}}}^{{\scriptscriptstyle N}}\Rightarrow F_{\!\mathcal{{\scriptscriptstyle N}}}^{*},\quad\text{and}\quad F_{\!\mathcal{{\scriptscriptstyle K}}}^{{\scriptscriptstyle N}}\Rightarrow F_{\!\mathcal{{\scriptscriptstyle K}}}^{*}.
\]
Using this, and appealing to the finite dimensional evolution equation
in Proposition \ref{prop: halfspace finite eqn}, the goal is now
to show that
\[
\mathbb{E}F_{\!\mathcal{{\scriptscriptstyle M}}}^{*}=\lim_{N\rightarrow\infty}\mathbb{E}F_{\!\mathcal{{\scriptscriptstyle M}}}^{{\scriptscriptstyle N}}=0,\quad\mathbb{E}F_{\!\mathcal{{\scriptscriptstyle N}}}^{*}=\lim_{N\rightarrow\infty}\mathbb{E}F_{\!\mathcal{{\scriptscriptstyle N}}}^{{\scriptscriptstyle N}}=0,\quad\mathbb{E}F_{\!\mathcal{{\scriptscriptstyle K}}}^{*}=\lim_{N\rightarrow\infty}\mathbb{E}F_{\!\mathcal{{\scriptscriptstyle K}}}^{{\scriptscriptstyle N}}=0,
\]
thus proving that $\mathcal{M}^{*}$, $\mathcal{N}^{*}$,
and $\mathcal{K}^{*}$ are true martingales, by a standard monotone
class argument. Relying on uniform integrability 
to conclude the convergence of the means, this follows easily by minor modifications of the arguments in \cite[Prop.~5.11]{HL2016}.
\end{proof}
By \cite[Thm.~3.2]{Led2016} and the tightness from Propositions \ref{prop:First regularity prop} and \ref{Convergence_of_mean}, we can extract a weakly convergent subsequence $(\nu^{{\scriptscriptstyle N}},L^{{\scriptscriptstyle N}},W^0)\Rightarrow(\nu^{*},L^*,W^0)$. In turn, Proposition \ref{prop: Maringale approach} and
the Doob\textendash Meyer decomposition theorem allows us to conclude that, for each $\phi\in\mathscr{S}$,
\begin{align*}
\left\langle \mathcal{M}(\nu^{*},L^{*})\right\rangle _{t}={\textstyle \int_{0}^{t}}&\left\langle \nu_{s}^{*},\sigma(s,\cdot)\rho(s,\nu^{*}_s,L^{*}_s)\partial_{x}\phi\right\rangle ^{2}\!ds \;\; \text{and}
\\\left\langle \mathcal{M}(\nu^{*},L^{*}),W^{0}\right\rangle _{t}&={\textstyle \int_{0}^{t}}\left\langle \nu_{s}^{*},\sigma(s,\cdot)\rho(s,\nu^{*}_s,L^{*}_s)\partial_{x}\phi\right\rangle ds,
\end{align*}
so it holds for all $t\in[0,T]$ that
\[
\left\langle \mathcal{M}(\nu^{*},L^{*})-{\textstyle \int_{0}^{\cdot}}\left\langle \nu_{s}^{*},\sigma(s,\cdot)\rho(s,\nu^{*}_s,L^{*}_s)\partial_{x}\phi\right\rangle dW_{s}^{0}\right\rangle _{t}=0.
\]
Hence $(\nu^{*},W^0)$ satisfies the SPDE (\ref{eq:LIMIT SPDE}) and
thus the proof of Theorem \ref{thm:EXISTENCE} is complete.

\section{Uniqueness arguments\label{sec:Uniqueness}}

In this section we present a proof of Theorem \ref{thm:UNIQUENESS}. In view of Section \ref{subsec:The-limit-SPDE:}, we fix a limit point $\nu$
of the sequence of empirical measures ($\nu^{{\scriptscriptstyle N}}$)
and let $\tilde{\nu}$ denote another candidate solution to the SPDE
(\ref{eq:LIMIT SPDE}). Then the strategy is to establish an energy
estimate in $H^{-1}$ for the difference $\Delta_{t}:=\nu_{t}-\tilde{\nu}_{t}$,
where $H^{-1}$ is the usual dual space of $H_{0}^{1}=\mathcal{W}_{0}^{1,2}(\mathbb{R}_{+})$.

	More specifically, we will rely on a `smoothed' $H^{-1}$  estimate given in Proposition~\ref{smooth_energy_est}. Based on this estimate, we derive the uniqueness of the SPDE in Section \ref{subsec:Uniqueness:-Proof-of}, and then Sections \ref{L2reg_and_bdr_effects} and \ref{subsec:The H^-1 Estimate} are devoted to the proof of Proposition \ref{smooth_energy_est}.

\subsection{Energy estimates and smoothing}\label{subsec:Energy-Estimates} 

Rather than estimating the $H^{-1}$ norm of $\Delta_t$ directly, our approach relies on smoothing the solutions $\nu$ and $\tilde{\nu}$ by means of convolution (with
a family of kernels approximating the identity). In this way, we can manipulate
the resulting equations classically. Since our problem is phrased on
the positive half-line with absorption at the boundary, it is natural
to consider the family of Dirichlet heat kernels $G_{\varepsilon}$
given by
\begin{equation}
G_{\varepsilon}(x,y):=p_{\varepsilon}(x-y)-p_{\varepsilon}(x+y),\quad p_{\varepsilon}(x)=(2\pi\varepsilon)^{-\frac{1}{2}}\exp\{-x^{2}/2\varepsilon\}.\label{eq: absorbing gaussian kernel}
\end{equation}
We denote the action of $G_{\varepsilon}$ on a measure $\mu$
by $\mathcal{T}_{\varepsilon}\mu$, that is

\[
(\mathcal{T}_{\varepsilon}\mu)(x):=\int_{0}^{\infty}G_{\varepsilon}(x,y)d\mu(y).
\]
For simplicity of presentation, we introduce the notation $\partial_{x}^{-1}$
for the anti-derivative, which amounts to
\begin{equation}
(\partial_{x}^{-1}\mathcal{T}_{\varepsilon}\Delta_{t})(x):=-\int_{x}^{\infty}(\mathcal{T}_{\varepsilon}\Delta_{t})(y)dy.\label{eq: Antiderivative of smoothed soln}
\end{equation}
Recall the embedding $\mathbf{M}_{\pm}\hookrightarrow H^{-1}$, where $\mathbf{M}_{\pm}$ is the space of finite signed measures on $\mathbb{R}$ (with total variation norm). As in \cite[Prop.~6.5]{HL2016}, we then have
\begin{equation}
\left\Vert \Delta_{t}\right\Vert _{H^{-1}}\leq\liminf_{\varepsilon\rightarrow0}\left\Vert \partial_{x}^{-1}\mathcal{T}_{\varepsilon}\Delta_{t}\right\Vert _{2},
\label{eq: H-1 norm via antiderivatve}
\end{equation}
where $\Vert \cdot \Vert_2$ denotes the $L^2$-norm on $\mathbb{R}_+$. Therefore, we can estimate the $H^{-1}$ norm of the difference $\Delta_t$
via the anti-derivatives of the smoothed solutions.

Let us now briefly outline the key ideas behind our approach. The
first observation is that $y\mapsto G_{\varepsilon}(x,y)$
is an admissible test function in $\mathscr{C}_{{\scriptscriptstyle 0}}$,
so we can plug it into the SPDE and thus obtain expressions for the smoothed solutions, $\mathcal{T}_{\varepsilon}\nu_{t}$
and $\mathcal{T}_{\varepsilon}\tilde{\nu}_{t}$. Integrating these to introduce the anti-derivatives, and looking at their
difference, this then allows us to obtain an equation for $\partial_{x}^{-1}\mathcal{T}_{\varepsilon}\Delta_{t}$
in terms of $\mathcal{T}_{\varepsilon}\Delta_{t}$ along with the occurrence
of $\mathcal{T}_{\varepsilon}\nu_{t}$ and some critical `boundary effects'
as well as a collection of simpler error terms.

In order to control the $L^{2}$-norm of $\partial_{x}^{-1}\mathcal{T}_{\varepsilon}\Delta_{t}$,
we thus need a uniform estimate on $\mathcal{T}_{\varepsilon}\nu_{t}$, and we
need to contain the boundary effects (and error terms) as $\varepsilon\rightarrow0$. These two tasks are the subject of Section \ref{L2reg_and_bdr_effects}, which then allows us to derive the aforementioned $L^2$  estimate for  $\partial_{x}^{-1}\mathcal{T}_{\varepsilon}\Delta_{t}$ in Section \ref{subsec:The H^-1 Estimate}. However, before any of this, we begin by showing how to derive the uniqueness of the SPDE given that the desired energy estimate holds.

\subsubsection{Uniqueness of the SPDE \textemdash{} Proof of Theorem \ref{thm:UNIQUENESS}\label{subsec:Uniqueness:-Proof-of}}

As above, we let $\nu$ be a limit point of the particle system and
suppose $\tilde{\nu}$ is another solution to the SPDE (\ref{eq:LIMIT SPDE}),
satisfying Assumption \ref{Assumption 2.3}.

Recall that the local Lipschitzness in the loss variable only holds
in a piecewise fashion on the intervals $[\theta_{i-1},\theta_{i})$
for $i=1,\ldots,k$. This hurdle is easily overcome by the following
piecewise stopping argument: Suppose we can prove uniqueness on $[0,t_0]$ if $L_t,\tilde{L}_t\in[0,\theta_1)$ for $t<t_0$ (see Prop.~\ref{smooth_energy_est} and the ensuing arguments) and introduce the stopping times
\[
\tau_{1}:=\inf\{t>0:L_{t}\geq\theta_{1}\}\land T\quad\text{and}\quad\tilde{\tau}_{1}:=\inf\{t>0:\tilde{L}_{t}\geq\theta_{1}\}\land T.
\]
Then $L_{t},\tilde{L}_{t}\in[0,\theta_{1})$ for $t<\varsigma_{1}:=\tau_{1}\land\tilde{\tau}_{1}$,
so we get uniqueness up to $\varsigma_{1}$. Note that, by the uniqueness, $L=\tilde{L}$ on $[0,\varsigma_{1}]$ and hence $\varsigma_{1}=\tau_{1}=\tilde{\tau}_{1}$. Thus, we can repeat
the uniqueness arguments on $[\varsigma_{1},\varsigma_{2})$, where
$L_{t},\tilde{L}_{t}\in[\theta_{1},\theta_{2})$, by defining $\varsigma_{2}:=\tau_{2}\land\tilde{\tau}_{2}$
for
\[
\tau_{2}:=\inf\{t>\varsigma_{1}:L_{t}\geq\theta_{2}\}\land T\quad\text{and}\quad\tilde{\tau}_{2}:=\inf\{t>\varsigma_{1}:\tilde{L}_{t}\geq\theta_{2}\}\land T.
\]
Continuing in this way for $\varsigma_{3},\ldots,\varsigma_{k-1}$,
we get uniqueness on all of $[0,T]$, since $L_{t}$ and $\tilde{L}_{t}$
are strictly increasing (recall part (ii) of Assumption \ref{Assumption 2.3}).

Below we prove uniqueness when the local Lipschitzness in $L$ holds everywhere, noting that the arguments imply uniqueness if $L$ and $\tilde{L}$ are confined to a particular piece $[\theta_{i-1},\theta_i)$. Thus, in view of the above stopping argument, the next result will suffice to complete the proof of Theorem \ref{thm:UNIQUENESS}.
\begin{prop}[Smoothed $H^{-1}$ estimate]\label{smooth_energy_est}
	Suppose the local Lipschitzness in Assumption \ref{Assumption 1 - finite system}(iii) holds everywhere, as opposed to piecewise. Then, as $\varepsilon\downarrow0$, we have
	\begin{align*}
	\mathbb{E}  \left\Vert \partial_{x}^{-1} \mathcal{T}_{\varepsilon}\Delta_{t\wedge t_n} \right\Vert _{2}^2 +&\, c_0\mathbb{E} \int_0^{t\wedge t_{n}}\!\left\Vert \mathcal{T}_\varepsilon\Delta_{s}\right\Vert _{2}^{2}ds  \leq  c_n \mathbb{E} \int_0^{t\wedge t_{n}} \!\!d_0(\nu_s,\tilde{\nu}_s)\!\left\Vert \partial_x^{-1}\mathcal{T}_\varepsilon\Delta_s \right\Vert_{2}ds\\
	& + c_{n}\mathbb{E}\int_{0}^{t\wedge t_{n}} \!|L_{s}-\tilde{L}_{s}|^2+\!d_1(\nu_{s},\tilde{\nu}_{s})^{2}ds+ o(1)  ,
	\end{align*}
	for a fixed $c_0>0$, and with $c_n$ only depending on $n$, where 
	$(t_{n})$ is a sequence of stopping times such that $t_{n}\uparrow T$ as $n\uparrow \infty$.
\end{prop}
\begin{proof}
	The proof is the subject of Sections \ref{L2reg_and_bdr_effects} and \ref{subsec:The H^-1 Estimate}.
	\end{proof}
Note that at this point Gr\"onwall already gives that $\mathbb{E}\!\int_0^{t\land t_n}\!\Vert\frac{d\Delta_s}{dx}\Vert_2ds$ is finite and, in particular, $\Delta_{s\land t_n}$ has a density in $L^2$, which we make use of below. Observe also that $|L-\tilde{L}|$  is dominated by $d_1(\nu,\tilde{\nu})$, however, it is included in the estimate as it shows where the above `piecewise stopping argument' would come into play. The next lemma relates the left- and right-hand sides of the smoothed $H^{-1}$ estimate, thus opening the door for a stronger Gr\"onwall argument that will allow us to finish the proof of Theorem \ref{thm:UNIQUENESS}.

\begin{lem}\label{lem: Lemma8.8 from hambly-ledger}
	There exists $c>0$ such that, for all $s\leq T$, $\delta\in(0,1)$, and $\lambda>1$,
	\[
	\begin{cases}
	d_0(\nu_s,\tilde{\nu}_s) \leq c\lambda(1+\delta^{-1}) {\bigl\Vert \partial_{x}^{-1}\mathcal{T}_{\varepsilon}\Delta_s\bigr\Vert}_{2} + c\delta^{\frac{1}{2}}\bigl\Vert {\textstyle \frac{d\Delta_s}{dx}} \bigr\Vert_2 + f_s(\lambda)+g_s(\varepsilon)\\[5pt]	
	d_1(\nu_s,\tilde{\nu}_s) \leq c\sqrt{\lambda+\delta^{-1}}{\bigl\Vert \partial_{x}^{-1}\mathcal{T}_{\varepsilon}\Delta_{s}\bigr\Vert}_{2} + c\delta^{\frac{1}{2}}\bigl\Vert {\textstyle \frac{d\Delta_s}{dx}} \bigr\Vert_2 + f_s(\lambda)+g_s(\varepsilon),
	\end{cases}
	\]
	as $\varepsilon\downarrow0$. Here $(f_t(\lambda))_{t\leq T}$ is a process such that, for every $a>0$, $\mathbb{E}\!\int_0^T\!f_s(\lambda)^2ds \leq c_a e^{-a\lambda}$ for some $c_a>0$, and $(g_s(\varepsilon))_{s\leq T}$ is process such that $\mathbb{E}\!\int_0^T \!g_s(\varepsilon)^2ds=o(1)$ as $\varepsilon\downarrow0$.
\end{lem}

Inserting the bound for $d_0$ in Prop.~\ref{smooth_energy_est}, we get the leading term $\lambda(1+\delta^{-1})\Vert \partial_{x}^{-1}\mathcal{T}_{\varepsilon}\Delta_s\Vert_2^2$, where it is crucial that $\lambda$ is not squared. For the rest of the terms coming from $d_0$ we can simply apply Young's inequality. After also inserting the bound for $d_1$, it follows that
	\begin{align}\label{unique_est}
\mathbb{E} & \left\Vert \partial_{x}^{-1}\mathcal{T}_\varepsilon \Delta_{t\wedge t_n} \right\Vert _{2}^2  \leq  \,
\delta c'_{n}\mathbb{E}\!\int_{0}^{t\wedge t_{n}} \Bigl\Vert \frac{d\Delta_s}{dx} \Bigr\Vert_2^2 ds - c_0\mathbb{E} \!\int_0^{t\wedge t_{n}}\!\left\Vert \mathcal{T}_\varepsilon\Delta_{s}\right\Vert _{2}^{2}ds \\
& \,+ c'_n (1+ \lambda + \lambda \delta^{-1}) \int_0^{t} \mathbb{E}\left\Vert \partial_x^{-1}\mathcal{T}_\varepsilon\Delta_{s\land t_n} \right\Vert_{2}^2ds + c'_a e^{-a\lambda} + o(1),\nonumber
\end{align}
as $\varepsilon\downarrow0$. By construction of $\mathcal{T}_\varepsilon$, we have $\Vert \mathcal{T}_{\varepsilon}\Delta_s \Vert_2 \leq \Vert \frac{d\Delta_s}{dx} \Vert_2$ and $\Vert \mathcal{T}_{\varepsilon}\Delta_s \Vert_2 \rightarrow \Vert \frac{d\Delta_s}{dx} \Vert_2$, so dominated convergence gives
\[
\mathbb{E}\int_0^{t\land t_n}\bigl\Vert \mathcal{T}_\varepsilon\Delta_s \bigr\Vert_2 ds \rightarrow  \mathbb{E}\int_0^{t\land t_n}\Bigl\Vert \frac{d\Delta_s}{dx} \Bigr\Vert_2 ds \quad \text{as}\quad\varepsilon\downarrow0.
\]
Therefore, by taking $\delta:=c_0/2c'_n$, the sum of the first two terms on the right-hand side of (\ref{unique_est}) is eventually non-positive for small  $\varepsilon>0$. In turn, we can apply the integrating factor $\exp\{c'_n(1+\lambda+\lambda\delta^{-1})t\}$ to the anti-derivative term in (\ref{unique_est}) and deduce that
\[
\mathbb{E} \left\Vert \partial_{x}^{-1}\mathcal{T}_{\varepsilon}\Delta_{t\wedge t_n} \right\Vert _{2}^2  \leq 
c'_ac'_n(1+\lambda+\lambda\delta^{-1})e^{c'_n(1+\lambda+\lambda\delta^{-1})T}e^{-a\lambda}+o(1)
\]
as $\varepsilon\downarrow0$. Recalling (\ref{eq: H-1 norm via antiderivatve}) and appealing to Fatou's lemma, it follows that
\[
\mathbb{E} \left\Vert \Delta_{t\wedge t_n} \right\Vert _{H^{-1}}^2  \leq 
c'_{a,n,T}(1+\lambda+\lambda\delta^{-1}) \exp\{ \lambda(1+\delta^{-1})T - a\lambda  \}.
\]
Consequently, we can simply take $a:=2(1+\delta^{-1})T$ and send $\lambda\rightarrow\infty$ to arrive at
\[
\mathbb{E} \left\Vert \Delta_{t\wedge t_n} \right\Vert _{H^{-1}}^2=0 \quad\forall t\in[0,T].
\]
Since $n$ was arbitrary and $t_n\uparrow T$ as $n\uparrow\infty$, we conclude that
$ \nu_t=\tilde{\nu}_t$ for all $t\in[0,T]$. This completes the proof of Theorem \ref{thm:UNIQUENESS}.

\begin{proof}[Proof of Lemma \ref{lem: Lemma8.8 from hambly-ledger}]
Fix $\psi \in\mathcal{C}_{d_0}$, where $\mathcal{C}_{d_0}:=\{ f : \Vert f\Vert_{\text{Lip}} \leq 1, |f(x)|\leq 1+|x|\}$. Fixing also $\lambda>1$ and $\delta>0$, we can take a cut-off function $\chi\in\mathcal{C}_{c}^{\infty}(\mathbb{R})$ equal to 1 on $[\delta,\lambda-1]$ and supported in $(\delta/2,\lambda)$ such that
\[
|\chi|\leq1,\quad |\partial\chi|\leq C_{\chi}/\delta \;\; \text{on}\;\; [\delta/2,\delta], \quad \text{and} \quad |\partial\chi|\leq C_{\chi} \;\;\text{on}\; \;[\lambda-1,\lambda].
\]
Observe that $\chi\psi\in\mathcal{W}^{1,\infty}_0(\delta/2,\lambda)$ with
\begin{align*}
\label{eq:cut_off}
\left\Vert \partial (\chi\psi)\right\Vert _{2}^{2} &= \int_{0}^{\delta}\left|\partial\chi\right|^{2}\left|\psi\right|^{2}dx + \int_{\lambda-1}^{\lambda}\left|\partial\chi\right|^{2}\left|\psi\right|^{2}dx + \int_{0}^{\lambda}\left|\chi\right|^{2}\left|\partial\psi\right|^{2}dx \\
&\leq C_\chi^2(\delta^{-1}+1) \Vert \psi \Vert_{L^{\infty}(0,\lambda)}^2 + \lambda.
\end{align*}
Next, we can observe that
\begin{align*}
|\langle \Delta_s,\psi\rangle| &\leq |\langle \Delta_s,\chi\psi\rangle| + |\langle \Delta_s,(\chi-1)\psi\textbf{1}_{[0,\delta]}\rangle| + \langle \nu_s + \tilde{\nu}_s,|\psi|\textbf{1}_{[\lambda-1,\infty)}\rangle\\
&\leq |\langle \Delta_s,\chi\psi\rangle| + (1+\delta)\delta^{1/2}\bigl\Vert {\textstyle \frac{d\Delta_s}{dx}} \bigr\Vert_2 + \langle \nu_s + \tilde{\nu}_s,(1+\left|\cdot\right|)\textbf{1}_{[\lambda-1,\infty)}\rangle
\end{align*}
Let $(\cdot,\cdot)$ be the inner product on $L^2$. Integration by parts and Cauchy--Schwarz gives
\begin{align*}
|\langle \Delta_s, \chi\psi \rangle| & \leq \bigl|(\mathcal{T}_\varepsilon\Delta_s , \chi\psi)\bigr| + \bigl|(\mathcal{T}_\varepsilon\Delta_s , \chi\psi) - \langle \Delta_s , \chi\psi \rangle \bigr| \\
& \leq {\bigl\Vert \partial_{x}^{-1}\mathcal{T}_{\varepsilon}\Delta_{s}\bigr\Vert}_{2} {\bigl\Vert \partial_{x}(\chi \psi) \bigr\Vert}_2  + \bigl|(\mathcal{T}_\varepsilon\Delta_s , \chi\psi) - \langle \Delta_s , \chi\psi \rangle \bigr|.
\end{align*}
By the Arzel\`a--Ascoli theorem applied to $\{\chi f:f\in \mathcal{C}_{d_0}\}$, we can find a finitely family $ \{\varphi_i\in \text{Lip}(\mathbb{R}):i=1,\ldots,k(\delta,\lambda)\}$ supported in $[\delta/2,\lambda]$ so that, for any $f \in \mathcal{C}_{d_0}$,
\[
\sup_{x\in\mathbb{R}}|\varphi_i(x)-\chi f(x)|\leq \delta^{\frac{1}{2}} \lambda^{-\frac{1}{2}}, \quad \text{for some} \;\; i=1,...,k(\delta,\lambda).
\]
Consequently, there is a $\varphi_i$ such that
\[
\bigl|(\mathcal{T}_\varepsilon\Delta_s , \chi\psi) - \langle \Delta_s , \chi\psi \rangle \bigr| 
 \leq  \bigl|(\mathcal{T}_\varepsilon\Delta_s , \varphi_i) - \langle \Delta_s , \varphi_i \rangle \bigr| +2  \sqrt{\delta} \bigl\Vert {\textstyle \frac{d\Delta_s}{dx}} \bigr\Vert_2,
\]
where we have used Cauchy--Schwarz with $\Vert \mathcal{T}_\varepsilon\Delta_s \Vert_2 \leq \Vert \frac{d\Delta_s}{dx} \Vert_2$ and $\Vert \varphi_i - \chi\psi \Vert_2 \leq \sqrt{\delta}$. Note that  $(\mathcal{T}_\varepsilon\Delta_s , \varphi_i) \rightarrow \langle \Delta_s , \varphi_i \rangle$ as $\varepsilon \rightarrow 0$ (see e.g.~\cite[Prop.~6.4]{HL2016}), so defining
\[
g_s^\varepsilon=g_s^\varepsilon(\delta,\lambda):=\sup \{ |(\mathcal{T}_\varepsilon\Delta_s , \varphi_i) - \langle \Delta_s , \varphi_i \rangle |:i=1,\ldots,k(\delta,\lambda)\},
\]
we have $\mathbb{E} \int_0^T\! g_s(\varepsilon)^2  ds  \rightarrow 0$ as $\varepsilon\rightarrow0$ by bounded convergence. Consequently,
\[
|\langle \Delta_s, \chi\psi \rangle| \leq {\bigl\Vert \partial_{x}^{-1}\mathcal{T}_{\varepsilon}\Delta_{s}\bigr\Vert}_{2} {\bigl\Vert \partial_{x}(\chi \psi) \bigr\Vert}_2 + 2  \sqrt{\delta} \bigl\Vert {\textstyle \frac{d\Delta_s}{dx}} \bigr\Vert_2 + g_s(\varepsilon),
\]
where $g$ is as required by the lemma. Finally, combining the above, we can take suprema over $\psi$ in the function classes that define $d_0$ and $d_1$ to find that
\[
d_0(\nu_s,\tilde{\nu}_s) \leq c\lambda\sqrt{1+\delta^{-1}} {\bigl\Vert \partial_{x}^{-1}\mathcal{T}_{\varepsilon}\Delta_{s}\bigr\Vert}_{2} + c\delta^{\frac{1}{2}}\bigl\Vert {\textstyle \frac{d\Delta_s}{dx}} \bigr\Vert_2 + f_s(\lambda)+g_s(\varepsilon),
\]
and
\[
d_1(\nu_s,\tilde{\nu}_s) \leq c\sqrt{\lambda+\delta^{-1}}{\bigl\Vert \partial_{x}^{-1}\mathcal{T}_{\varepsilon}\Delta_{s}\bigr\Vert}_{2} + c\delta^{\frac{1}{2}}\bigl\Vert {\textstyle \frac{d\Delta_s}{dx}} \bigr\Vert_2 + f_s(\lambda)+g_s(\varepsilon),
\]
 where $f_s(\lambda):= \langle \nu_s + \tilde{\nu}_s,(1+\left|\cdot\right|)\textbf{1}_{[\lambda-1,\infty)}\rangle$. Hence the proof is complete by reference to the exponential decay properties of $f_s(\lambda)$ as guaranteed by Lemma \ref{lem:Tails for centre of mass}.
\end{proof}

\subsection{$L^{2}$ regularity and containment of boundary effects}\label{L2reg_and_bdr_effects}

Our first task in this section is to establish a weighted $L^{2}$ estimate for $\mathcal{T}_{\varepsilon}\nu_{t}$
that is uniform in $\varepsilon>0$ and $t\in[0,T]$. This can be achieved by exploiting the extra control that we have
over $\nu$ because it arises as a limit point of the particle system.

Specifically, we can use that the empirical measures $(\nu^{{\scriptscriptstyle N}})$
are dominated by their whole-space counterparts so that, in the limit, $\nu$ is dominated
by a solution to the whole-space version of the SPDE (\ref{eq:LIMIT SPDE}) with $M$ and $L$ still defined in terms of $\nu$.
This is crucial as we can perform the $L^{2}$ estimates for the whole-space
SPDE without any boundary effects. On the other hand, the estimates for the half-line only succeed because we work
in the weaker space $H^{-1}$, where we can control the  boundary effects
solely by relying on the boundary decay from Proposition \ref{prop:limit_regularity}
and Assumption \ref{Assumption 2.3} (see Lemma 5.5 and Section 5.3.2).
\begin{prop}
[$L^{2}$~energy estimate]\label{prop: Energy bound on particle limit}Let
$\nu^{*}$ be a limit point of $(\nu^{{\scriptscriptstyle N}})$.
Then 
\begin{equation}
\mathbb{E}\biggl[\sup_{t\in[0,T]}\sup_{\varepsilon>0}\int_{0}^{\infty}(1+x^{2})(\mathcal{T}_{\varepsilon}\nu_{t}^{*})^{2}dx\biggr]<\infty.\label{eq:L2_Energy_est}
\end{equation}
\end{prop}

\begin{proof}
The rough ideas are the same as \cite[Prop.~7.1]{HL2016},
so we will refer back to this in order to avoid
duplication. First, let $\bar{\nu}^{{\scriptscriptstyle N}}:=\sum_{i=1}^{\scriptscriptstyle N}a_{i}^{\scriptscriptstyle N}\delta_{X_{t}^{i}}$
and note $\nu^{{\scriptscriptstyle N}}(A) \leq \bar{\nu}^{{\scriptscriptstyle N}}(A)$ for all $A\in\mathcal{B}(\mathbb{R})$.
Now, the same work as for $(\nu^{{\scriptscriptstyle N}})$ ensures $(\nu^{{\scriptscriptstyle N}},\bar{\nu}^{{\scriptscriptstyle N}})$ is tight with limit points $\nu^* \leq \bar{\nu}^*$, where $\bar{\nu}^*$ satisfies the whole-space analogue of the SPDE (\ref{eq:LIMIT SPDE}) in the sense that the space of test functions is all of $\mathscr{S}$, but with $M$ and $L$ still defined in terms of $\nu^*$.
Arguing as in Lemmas 7.2 and 7.3 of \cite{HL2016}, to prove
(\ref{eq:L2_Energy_est}) it suffices to bound
\[
J_{1}:=\liminf_{\varepsilon\rightarrow0}\mathbb{E}\biggl[\sup_{t\leq T}\int_{\mathbb{R}}(\bar{\mathcal{T}}_{\varepsilon}\bar{\nu}_{t}^{*})^{2}dx\biggr]\quad\text{and}\quad J_{2}:=\liminf_{\varepsilon\rightarrow0}\mathbb{E}\biggl[\sup_{t\leq T}\int_{\mathbb{R}}x^{2}(\bar{\mathcal{T}}_{\varepsilon}\bar{\nu}_{t}^{*})^{2}dx\biggr],
\]
where the action of $\bar{\mathcal{T}}_{\varepsilon}$ is given by
\[
(\bar{\mathcal{T}}_{\varepsilon}\bar{\nu}_{t}^{*})(x):=\int_{\mathbb{R}}p_{\varepsilon}(x-y)d\bar{\nu}_{t}^{*}(y).
\]
Note that $y\mapsto p_{\varepsilon}(x-y)$ is certainly not in $\mathscr{C}_{{\scriptscriptstyle 0}}$,
but it has rapid decay at $\pm\infty$, so it is an admissible test
function for the whole-space SPDE satisfied by $\bar{\nu}^{*}$.
We split the remaining parts of the proof into three steps.

\textbf{Step 1. }We begin by showing that $J_{1}\apprle\int_{0}^{\infty}V_{0}(x)^{2}dx$.
To simplify things, let 
\[
\mathfrak{b}_{t}(x):=b(t,x,\nu_{t},L_t)-\alpha(t,x,\nu_{t},L_t)\mathfrak{L}'_{t}.
\]
Then $\mathfrak{b}_{t}$ will play the r\^ole of the drift $\mu_{t}$
in the proof of Proposition 7.1 from \cite{HL2016}. Proceeding as
in \cite{HL2016}, we can test the SPDE with $p_{\varepsilon}(\cdot-y)$
and introduce appropriate error terms in order to get a tractable
expression for $\bar{\mathcal{T}}_{\varepsilon}\bar{\nu}_{t}^{*}$. Next, we
can then apply Itô's formula and thus convert this into an expression
for the square $(\bar{\mathcal{T}}_{\varepsilon}\bar{\nu}_{t}^{*})^{2}$. In this way, we arrive at
\begin{align}
d(\bar{\mathcal{T}}_{\varepsilon}\bar{\nu}_{t}^{*})^{2}= & -2(\bar{\mathcal{T}}_{\varepsilon}\bar{\nu}_{t}^{*})\bigl(\mathfrak{b}_{t}\partial_{x}(\bar{\mathcal{T}}_{\varepsilon}\bar{\nu}_{t}^{*})-\partial_{x}\mathfrak{b}_{t}\mathcal{\bar{H}}_{t,\varepsilon}^{\mathfrak{b}}+\mathcal{\bar{E}}_{t,\varepsilon}^{\mathfrak{b}}\bigr)dt\label{eq: Ito square}\\
 & +(\bar{\mathcal{T}}_{\varepsilon}\bar{\nu}_{t}^{*})\partial_{x}\bigl(\sigma_{t}^{2}\partial_{x}(\bar{\mathcal{T}}_{\varepsilon}\bar{\nu}_{t}^{*})-\partial_{x}\sigma_{t}^{2}\mathcal{\bar{H}}_{t,\varepsilon}^{\sigma^{2}}+\mathcal{\bar{E}}_{t,\varepsilon}^{\sigma^{2}}\bigr)dt\nonumber \\
 & -2\rho_{t}(\bar{\mathcal{T}}_{\varepsilon}\bar{\nu}_{t}^{*})\bigl(\sigma_{t}\partial_{x}(\bar{\mathcal{T}}_{\varepsilon}\bar{\nu}_{t}^{*})-\partial_{x}\sigma_{t}\mathcal{\bar{H}}_{t,\varepsilon}^{\sigma}+\mathcal{\bar{E}}_{t,\varepsilon}^{\sigma}\bigr)dW_{t}\nonumber \\
 & +\rho_{t}^{2}\bigl(\sigma_{t}\partial_{x}(\bar{\mathcal{T}}_{\varepsilon}\bar{\nu}_{t}^{*})-\partial_{x}\sigma_{t}\mathcal{\bar{H}}_{t,\varepsilon}^{\sigma}+\mathcal{\bar{E}}_{t,\varepsilon}^{\sigma}\bigr)^{2}dt,\nonumber 
\end{align}
where the error terms $\mathcal{\bar{E}}$ and $\mathcal{\bar{H}}$
are defined in Lemma \ref{lem:Error Terms Whole Space} of the Appendix.
Compared to \cite{HL2016}, we must be careful with the linear growth of
$\mathfrak{b}_{t}$, however, this is taken care
of via integration by parts: Since $\bar{\mathcal{T}}_{\varepsilon}\bar{\nu}_{t}^{*}$
vanishes at $\pm\infty$ by the tails of $\bar{\nu}_{t}^{*}$ (using the the whole-space analogue of Lemma \ref{lem:BOREL-CANTELLI-ARGUMENTS}), we thus get
\begin{align}
-\int_{\mathbb{R}}\int_{0}^{t}2\mathfrak{b}_{s}(\bar{\mathcal{T}}_{\varepsilon}\bar{\nu}_{s}^{*})\partial_{x}(\bar{\mathcal{T}}_{\varepsilon}\bar{\nu}_{s}^{*})dsdx= & -\int_{0}^{t}\int_{\mathbb{R}}\mathfrak{b}_{s}\partial_{x}(\bar{\mathcal{T}}_{\varepsilon}\bar{\nu}_{s}^{*})^{2}dxds\nonumber \\
= & \int_{0}^{t}\int_{\mathbb{R}}\partial_{x}\mathfrak{b}_{s}(\bar{\mathcal{T}}_{\varepsilon}\bar{\nu}_{s}^{*})^{2}dxds\leq C\int_{0}^{t}\bigl\Vert\bar{\mathcal{T}}_{\varepsilon}\bar{\nu}_{s}^{*}\bigr\Vert_{2}^{2}ds.\label{eq: modified drift bound}
\end{align}

Given this, the plan is to integrate over $x\in\mathbb{R}$ in the
equation (\ref{eq: Ito square}) in order to obtain an estimate for
the $L^{2}$-norm of $\bar{\mathcal{T}}_{\varepsilon}\bar{\nu}_{t}^{*}$. After
integrating over $x$, we can then appeal to the previous estimate
(\ref{eq: modified drift bound}) and, similarly, we can perform another
integration by parts in the second line on the right-hand side of
(\ref{eq: Ito square}). Using this, and the fact that $|\mathcal{\bar{H}}_{t,\varepsilon}^{g}|\apprle \bar{\mathcal{T}}_{2\varepsilon}\bar{\nu}_{t}^{*}$
by Lemma \ref{lem:Error Terms Whole Space}, we can thus apply Young's
inequality (with free parameter $\theta>0$) to see that
\begin{align}
\bigl\Vert\bar{\mathcal{T}}_{\varepsilon}\bar{\nu}_{t}^{*}\bigr\Vert_{2}^{2}\leq & \;\bigl\Vert\bar{\mathcal{T}}_{\varepsilon}\nu_{0}\bigr\Vert_{2}^{2}+C_{\theta}\int_{0}^{t}\bigl\Vert\bar{\mathcal{T}}_{\varepsilon}\bar{\nu}_{s}^{*}\bigr\Vert_{2}^{2}ds+C_{\theta}\int_{0}^{t}\bigl\Vert\bar{\mathcal{T}}_{2\varepsilon}\bar{\nu}_{s}^{*}\bigr\Vert_{2}^{2}ds\label{eq: Whole Space estimate}\\
 & +C_{\theta}\int_{0}^{t}\bigl(\bigl\Vert\mathcal{\bar{E}}_{s,\varepsilon}^{\mathfrak{b}}\bigr\Vert_{2}^{2}+\bigl\Vert\mathcal{\bar{E}}_{s,\varepsilon}^{\sigma^{2}}\bigr\Vert_{2}^{2}+\bigl\Vert\mathcal{\bar{E}}_{s,\varepsilon}^{\sigma}\bigr\Vert_{2}^{2}\bigr)ds\nonumber \\
 & +\int_{0}^{t}\int_{\mathbb{R}}\bigl(\rho_{s}^{2}\sigma_{s}^{2}+\theta\rho_{s}^{2}\sigma_{s}^{2}-\sigma_{s}^{2}+\theta\bigr)\bigl|\partial_{x}(\bar{\mathcal{T}}_{\varepsilon}\bar{\nu}_{s}^{*})\bigr|^{2}dxds\nonumber \\
 & +\int_{0}^{t}\int_{\mathbb{R}}\rho_{s}\partial_{x}\sigma_{s}(\bar{\mathcal{T}}_{\varepsilon}\bar{\nu}_{s}^{*})^{2}+2\rho_{s}(\bar{\mathcal{T}}_{\varepsilon}\bar{\nu}_{s}^{*})(\partial_{x}\sigma_{s}^{2}\mathcal{\bar{H}}_{s,\varepsilon}^{\sigma}-\mathcal{\bar{E}}_{s,\varepsilon}^{\sigma}\bigr)dxdW_{s}.\nonumber 
\end{align}
Here we have used the stochastic fubini theorem to switch the order
of integration in the stochastic integral (allowed due
to the exponential tails as in
\cite[Lem.~8.3]{HL2016}) and we have also integrated by parts
in the $dx$-integral inside the stochastic integral.

\textbf{Step 2. }Since $\rho^{2}$ is bounded away from $1$ (and $\sigma^2$ is bounded away from $0$),
we can choose $\theta$ sufficiently small so that the 
third line of (\ref{eq: Whole Space estimate}) is negative and hence we can discard it. Raising both sides
of (\ref{eq: Whole Space estimate}) to a power $k\geq1$, we thus
have
\begin{align}
\sup_{r\leq t}\bigl\Vert\bar{\mathcal{T}}_{\varepsilon}\bar{\nu}_{r}^{*}\bigr\Vert_{2}^{2k}\leq & \;C\bigl\Vert\bar{\mathcal{T}}_{\varepsilon}\nu_{0}\bigr\Vert_{2}^{2k}+C\int_{0}^{t}\Bigl(\bigl\Vert\bar{\mathcal{T}}_{\varepsilon}\bar{\nu}_{s}^{*}\bigr\Vert_{2}^{2k}+\bigl\Vert\bar{\mathcal{T}}_{2\varepsilon}\bar{\nu}_{s}^{*}\bigr\Vert_{2}^{2k}\Bigr)ds\label{eq: 4th power energy estimate}\\
 & +C\int_{0}^{t}\bigl(\bigl\Vert\mathcal{\bar{E}}_{s,\varepsilon}^{\mathfrak{b}}\bigr\Vert_{2}^{2k}+\bigl\Vert\mathcal{\bar{E}}_{s,\varepsilon}^{\sigma^{2}}\bigr\Vert_{2}^{2k}+\bigl\Vert\mathcal{\bar{E}}_{s,\varepsilon}^{\sigma}\bigr\Vert_{2}^{2k}\bigr)ds\nonumber \\
 & +C\Bigl\{\sup_{r\leq t}\int_{0}^{r}\!\int_{\mathbb{R}}\rho_{s}\partial_{x}\sigma_{s}(\bar{\mathcal{T}}_{\varepsilon}\bar{\nu}_{s}^{*})^{2}+2\rho_{s}(\bar{\mathcal{T}}_{\varepsilon}\bar{\nu}_{s}^{*})(\partial_{x}\sigma_{s}^{2}\mathcal{\bar{H}}_{s,\varepsilon}^{\sigma}-\mathcal{\bar{E}}_{s,\varepsilon}^{\sigma}\bigr)dxdW_{s}\Bigr\}^{k},\nonumber 
\end{align}
with $C=C(k)$. By Burkholder\textendash Davis\textendash Gundy, and Hölder's and Young's inequalities, the expected supremum of the stochastic integral is bounded by a constant times
\begin{multline*}
 \quad\mathbb{E}\Bigl\{\int_{0}^{t}\Bigl(\int_{\mathbb{R}}(\bar{\mathcal{T}}_{\varepsilon}\bar{\nu}_{s}^{*})\bigl(\bar{\mathcal{T}}_{\varepsilon}\bar{\nu}_{s}^{*}+|\partial_{x}\sigma_{s}^{2}\mathcal{\bar{H}}_{s,\varepsilon}^{\sigma}|+|\mathcal{\bar{E}}_{s,\varepsilon}^{\sigma}|\bigr)dx\Bigr)^{2}ds\Bigr\}^{k/2}\\ \leq
C'\mathbb{E}\sup_{r\leq t}\bigl\Vert\bar{\mathcal{T}}_{\varepsilon}\bar{\nu}_{r}^{*}\bigr\Vert_{2}^{2k}+C''\mathbb{E}\int_{0}^{t}\bigl(\bigl\Vert\bar{\mathcal{T}}_{\varepsilon}\bar{\nu}_{t}^{*}\bigr\Vert_{2}^{2k}+\bigl\Vert\bar{\mathcal{T}}_{2\varepsilon}\bar{\nu}_{t}^{*}\bigr\Vert_{2}^{2k}+\bigl\Vert\mathcal{\bar{E}}_{s,\varepsilon}^{\sigma}\bigr\Vert_{2}^{2k}\bigr)ds,\quad
\end{multline*}
Consequently, taking expectations in (\ref{eq: 4th power energy estimate}) gives
\begin{align*}
\inf_{\varepsilon\leq\varepsilon'}\mathbb{E}\sup_{r\leq t}\bigl\Vert\bar{\mathcal{T}}_{\varepsilon}\bar{\nu}_{r}^{*}\bigr\Vert_{2}^{2k}\leq & \;C_{0}\inf_{\varepsilon\leq\varepsilon'}\bigl\Vert\bar{\mathcal{T}}_{\varepsilon}\nu_{0}\bigr\Vert_{2}^{2k}+tC_{0}\inf_{\varepsilon\leq\varepsilon'}\mathbb{E}\sup_{r\leq t}\bigl\Vert\bar{\mathcal{T}}_{\varepsilon}\bar{\nu}_{r}^{*}\bigr\Vert_{2}^{2k}\\
 & +C_{0}\inf_{\varepsilon\leq\varepsilon'}\mathbb{E}\int_{0}^{t}\bigl(\bigl\Vert\mathcal{\bar{E}}_{s,\varepsilon}^{\mathfrak{b}}\bigr\Vert_{2}^{2k}+\bigl\Vert\mathcal{\bar{E}}_{s,\varepsilon}^{\sigma^{2}}\bigr\Vert_{2}^{2k}+\bigl\Vert\mathcal{\bar{E}}_{s,\varepsilon}^{\sigma}\bigr\Vert_{2}^{2k}\bigr)ds.
\end{align*}
If we now restrict to $t \leq T_{0}:=1/2C_{0}$ and
send $\varepsilon'\rightarrow0$, then we get
\[
\liminf_{\varepsilon\rightarrow0}\mathbb{E}\sup_{t\leq T_{0}}\bigl\Vert\bar{\mathcal{T}}_{\varepsilon}\bar{\nu}_{t}^{*}\bigr\Vert_{2}^{2k}\leq2C_{0}\bigl\Vert V_{0}\bigr\Vert_{2}^{2k},
\]
where we have used that $\Vert \bar{\mathcal{T}}_{\varepsilon}\nu_{0}\Vert _{2}\rightarrow\Vert V_{0}\Vert_{2}$
as $\varepsilon\rightarrow0$ and that the error terms vanish by Lemma
\ref{lem:Error Terms Whole Space}. For $k=1$, this proves the bound
on $J_{1}$ in small time, that is, for $t\in[0,T_0]$ with $T_0=1/2C_0$. The extension of this bound to all of $[0,T]$ follows by propagating the argument onto the finitely many intervals $[T_k,T_{k+1} \land T]$ with $T_{k+1}=T_k+1/2C_0$ for $k=1,\ldots,\left\lceil 2C_{0}T\right\rceil$ (as in the proof of \cite[Prop.~7.1]{HL2016}).

\textbf{Step 3. }We now show how to extend the previous work to prove
that $J_{2}$ is finite. In order to succeed at this, we will need
to control the 4th moments of $\sup_{t\leq T}\Vert\bar{\mathcal{T}}_{\varepsilon}\bar{\nu}_{t}^{*}\Vert_{2}$,
which is the reason for introducing the power $k$ in Step 2 above. The idea is simply to multiply by $x^{2}$ in (\ref{eq: Ito square})
before integrating over $x\in\mathbb{R}$ and then proceed as in Steps 1 and 2 above. Beginning with the first term,
an integration by part yields
\begin{align*}
-\int_{\mathbb{R}}\int_{0}^{t}2x^{2}\mathfrak{b}_{s}(\bar{\mathcal{T}}_{\varepsilon}\bar{\nu}_{s}^{*})\partial_{x}(\bar{\mathcal{T}}_{\varepsilon}\bar{\nu}_{s}^{*})dsdx= & \,\int_{0}^{t}\int_{\mathbb{R}}\bigl(x^{2}\partial_{x}\mathfrak{b}_{s}+2x\mathfrak{b}_{s}\bigr)(\bar{\mathcal{T}}_{\varepsilon}\bar{\nu}_{s}^{*})^{2}dxds\\
\leq & \;C\int_{0}^{t}\bigl\Vert x(\bar{\mathcal{T}}_{\varepsilon}\bar{\nu}_{s}^{*})\bigr\Vert_{2}^{2}ds+C\int_{0}^{t}(1+|M_{s}|)\bigl\Vert\bar{\mathcal{T}}_{\varepsilon}\bar{\nu}_{s}^{*}\bigr\Vert_{2}^{2}ds.
\end{align*}
Crucially, the second term on the right-hand
side can be controlled by
\begin{equation}
\mathbb{E}\int_{0}^{t}(1+|M_{s}|)\bigl\Vert\bar{\mathcal{T}}_{\varepsilon}\bar{\nu}_{s}^{*}\bigr\Vert_{2}^{2}ds\leq\int_{0}^{t}\mathbb{E}\left[(1+|M_{s}|)^{2}\right]ds+\int_{0}^{t}\mathbb{E}\left[\bigl\Vert\bar{\mathcal{T}}_{\varepsilon}\bar{\nu}_{s}^{*}\bigr\Vert_{2}^{4}\right]ds,\label{eq: x^2 M*Tv bound}
\end{equation}
which is finite because of the result in Step 2 with $k=2$. Arguments
completely analogous to those that led to (\ref{eq: Whole Space estimate})\textendash (\ref{eq: 4th power energy estimate})
then yield
\begin{align}
\bigl\Vert x\bar{\mathcal{T}}_{\varepsilon}\bar{\nu}_{t}^{*}\bigr\Vert_{2}^{2}\leq & \;\bigl\Vert x\bar{\mathcal{T}}_{\varepsilon}\nu_{0}\bigr\Vert_{2}^{2}+C_{\theta}\int_{0}^{t}\bigl\Vert x\bar{\mathcal{T}}_{\varepsilon}\bar{\nu}_{s}^{*}\bigr\Vert_{2}^{2}ds+C_{\theta}\int_{0}^{t}\bigl\Vert x\bar{\mathcal{T}}_{2\varepsilon}\bar{\nu}_{s}^{*}\bigr\Vert_{2}^{2}ds+C^{\prime}\label{eq: Whole Space estimate-1}\\
 & +C_{\theta}\int_{0}^{t}\bigl(\bigl\Vert x\mathcal{\bar{E}}_{s,\varepsilon}^{\mathfrak{b}}\bigr\Vert_{2}^{2}+\bigl\Vert x\mathcal{\bar{E}}_{s,\varepsilon}^{\sigma^{2}}\bigr\Vert_{2}^{2}+\bigl\Vert x\mathcal{\bar{E}}_{s,\varepsilon}^{\sigma}\bigr\Vert_{2}^{2}\bigr)ds\nonumber \\
 & -2\int_{0}^{t}\int_{\mathbb{R}}x^{2}\rho_{s}\bar{\mathcal{T}}_{\varepsilon}\bar{\nu}_{s}^{*}(\sigma_{s}\partial_{x}\bar{\mathcal{T}}_{\varepsilon}\bar{\nu}_{t}^{*}+\partial_{x}\sigma_{s}\mathcal{\bar{H}}_{s,\varepsilon}^{\sigma}+\mathcal{\bar{\mathcal{E}}}_{s,\varepsilon}^{\sigma})dxdW_{s},\nonumber 
\end{align}
where the extra constant $C^{\prime}$ comes from (\ref{eq: x^2 M*Tv bound}).
Given this, we can argue as in Step 2 (with $k=1$) to bound $J_{2}$
by a multiple of $C^{\prime}+\Vert xV_{0}\Vert_{2}^{2}$.
\end{proof}
\begin{cor}
[Density process]\label{cor:DENSITY}Any limit point $\nu^{*}$ of
$(\nu^{{\scriptscriptstyle N}})$ has an $L^{2}$-valued density process
$(\displaystyle{V_{t}^{*}})_{t\geq0}$ with $\left\Vert x\displaystyle{V_{t}^{*}}\right\Vert _{2}<\infty$.
\end{cor}

\begin{proof}
Given Proposition \ref{prop: Energy bound on particle limit}, this follows by a standard weak compactness argument for the  $L^{2}$-bounded sequence
$(\mathcal{T}_{\varepsilon}\nu^{*}_{t})_{\epsilon>0}$.
\end{proof}
The next lemma represses the `boundary effects' in the
estimates for $\partial_{x}^{-1}\mathcal{T}_{\varepsilon}\Delta_{t}$, by ensuring that the relevant terms vanish as $\varepsilon\rightarrow0$. This
depends critically on the behaviour of the mass of the solutions near the boundary, so the essential ingredient
is the boundary decay from Assumption \ref{Assumption 2.3} as satisfied by the limit points because of Proposition \ref{prop:limit_regularity}.
\begin{lem}
[Boundary estimate]\label{lem:Boundary Estimate}Let $\mu$ satisfy
Assumption \ref{Assumption 2.3} and let $g_{t}(y)$ be a (stochastic)
function with $\left|g_{t}(y)\right|\apprle1+\left|y\right|+M_{t}$,
where $M_{t}:=\left\langle \mu_{t},\psi\right\rangle $ for some $\psi\in\text{\emph{Lip}}(\mathbb{R})$.
Then
\[
\mathbb{E}\int_{0}^{T}\int_{0}^{\infty}\bigl|\bigl\langle\mu_{t},g_{t}(\cdot)p_{\varepsilon}(x+\cdot)\bigr\rangle\bigr|^{2}dxdt\rightarrow0\quad\text{as}\quad\varepsilon\rightarrow0.
\]
\end{lem}

\begin{proof}
Notice first that, by Jensen's inequality,
\[
\bigl|\bigl\langle\mu_{t},g_{t}(\cdot)p_{\varepsilon}(x+\cdot)\bigr\rangle\bigr|^{2}\leq C\varepsilon^{-1}e^{-x^{2}/\varepsilon}\int_{0}^{\infty}(1+y^{2}+M_{t}^{2})e^{-y^{2}/\varepsilon}d\mu_{t}(y).
\]
and hence
\[
\int_{0}^{\infty}\bigl|\bigl\langle\mu_{t},g_{t}(\cdot)p_{\varepsilon}(x+\cdot)\bigr\rangle\bigr|^{2}dx\leq C\int_{0}^{\infty}(1+y^{2}+M_{t}^{2})\varepsilon^{-\frac{1}{2}}e^{-y^{2}/\varepsilon}d\mu_{t}(y)
\]
We divide the proof into three cases, where either $1$,
$y^{2}$ or $M_{t}^{2}$ appears in the integrand. The first case follows as in \cite[Lem.~7.6]{HL2016}. For the second
case, fix $0<\varepsilon<1$, and let $\eta\in(0,1)$ be a free parameter.
Splitting the integral on $y\leq\varepsilon^{\eta}$ and its complement gives
\begin{equation}
\varepsilon^{-\frac{1}{2}}\int_{0}^{\infty}y^{2}e^{-y^{2}/\varepsilon}d\mu_{t}(y)\leq\varepsilon^{-\frac{1}{2}}\left(\mu_{t}(0,\varepsilon^{\eta})+\left\langle \mu_{t},y^{2}\right\rangle \exp\left\{ -\varepsilon^{2\eta-1}\right\} \right).\label{eq: Boundary lemma estimate}
\end{equation}
By (iii)-(iv) of Assumption \ref{Assumption 2.3}, appealing
also to Lemma \ref{lem:Tails for centre of mass}, it follows that
\begin{equation}
\mathbb{E}\int_{0}^{T}\int_{0}^{\infty}y^{2}\varepsilon^{-\frac{1}{2}}e^{-y^{2}/\varepsilon}d\mu_{t}(y)dt\leq C\varepsilon^{-\frac{1}{2}}\left(\varepsilon^{\eta(1+\beta)}+\exp\left\{ -\varepsilon^{2\eta-1}\right\} \right)\label{eq: Boundary lemma estimate 2}
\end{equation}
for a constant $\beta>0$, where the right-hand side of (\ref{eq: Boundary lemma estimate 2})
converges to zero as long as
\[
(2+2\beta)^{-1}<\eta<2^{-1}.
\]
For the final case, we can rely on Hölder's inequality to see that
\[
\mathbb{E}\int_{0}^{T}\!M_{t}^{2}\int_{0}^{\infty}\varepsilon^{-\frac{1}{2}}e^{-y^{2}/\varepsilon}d\mu_{t}dt\leq\mathbb{E}\biggl[\int_{0}^{T}\!M_{t}^{2q}dt\biggr]^{\frac{1}{q}}\mathbb{E}\biggl[\int_{0}^{T}\Bigl|\int_{0}^{\infty}\varepsilon^{-\frac{1}{2}}e^{-y^{2}/\varepsilon}d\mu_{t}\Bigr|^{p}dt\biggr]^{\frac{1}{p}}.
\]
Since $M_{t}^{2q}\apprle\mu_{t}(0,\infty)+\langle\mu_{t},y^{2q}\rangle$
by the Lipschitzness of $\psi$, we get  $\mathbb{E}\!\int_{0}^{T}\!\!M_{t}^{2q}dt<\infty$ for all $q>1$ by Assumption
\ref{Assumption 2.3}(iii) and Lemma \ref{lem:Tails for centre of mass}. Moreover, arguing as
in (\ref{eq: Boundary lemma estimate})\textendash (\ref{eq: Boundary lemma estimate 2}),
\[
\mathbb{E}\int_{0}^{T}\Bigl|\int_{0}^{\infty}\varepsilon^{-\frac{1}{2}}e^{-y^{2}/\varepsilon}d\mu_{t}(y)\Bigr|^{p}dt\leq C\varepsilon^{-\frac{p}{2}}\left(\varepsilon^{\eta(1+\beta)}+\exp\left\{ -p\varepsilon^{2\eta-1}\right\} \right),
\]
so the claim follows by taking $p:=1+\beta/2$ and choosing $\eta$
in the range
\[
(1+\beta/2)(2+2\beta)^{-1}<\eta<2^{-1}.
\]
This finishes the proof.
\end{proof}

\subsection{The smoothed $H^{-1}$ estimate --- proof of Proposition \ref{smooth_energy_est}}\label{subsec:The H^-1 Estimate}
In this section we finalize the proof of the smoothed $H^{-1}$ estimate from Proposition \ref{smooth_energy_est}.
\begin{rem}
[`$o_{\text{sq}}(1)$'--notation] As in \cite{HL2016}, we denote by
$o_{\text{sq}}(1)$ any family of $L^{2}$ valued processes $\{ (\zeta^\varepsilon_t)_{t\leq T}\}_{\varepsilon>0}$
that satisfy
$\mathbb{E}{\textstyle\int_{0}^{T}}\Vert \zeta^\varepsilon_t\Vert _{2}^{2}dt\rightarrow0$ as $\varepsilon\downarrow0$.
For example, Lemma~\ref{lem:Boundary Estimate} and Lemma~\ref{lem:Error terms half-space}
ensure that the `boundary effects' and error terms will be
$o_{\text{sq}}(1)$.
\end{rem}

As outlined at the start of Section \ref{subsec:Energy-Estimates},
we will derive an equation for $\partial_{x}^{-1}\mathcal{T}_{\varepsilon}\Delta_{t}$
by testing the SPDE for $\nu$ and $\tilde{\nu}$ with the kernel
$y\mapsto G_{\varepsilon}(\cdot,y)$ and then integrating the difference
of the resulting expressions over $(x,\infty)$ for $x>0$. As in the proof of Proposition \ref{prop: Energy bound on particle limit}, we set $\mathfrak{b}_t:=b_t-\alpha_t{\mathfrak{L}}'_t$ and $\tilde{\mathfrak{b}}_t:=\tilde{b}_t-\tilde{\alpha}_t\tilde{\mathfrak{L}}'_t$. Then we can argue as in \cite[Sec.~7]{HL2016}, using Lemma \ref{lem:Boundary Estimate} and
Lemma \ref{lem:Error terms half-space} in place of their counterparts
in \cite{HL2016}, to conclude that
\begin{align*}
d\partial_{x}^{-1}\mathcal{T}_{\varepsilon}\Delta_{t}= & -\bigl(\tilde{\mathfrak{b}}_{t}\mathcal{T}_{\varepsilon}\Delta_{t}+\delta_{t}^{\mathfrak{b}}\mathcal{T}_{\varepsilon}\nu_{t}\bigr)dt+{\textstyle \frac{1}{2}}\partial_{x}\bigl(\sigma_{t}^{2}\mathcal{T}_{\varepsilon}\Delta_{t}+\mathcal{E}_{t,\varepsilon}^{\sigma^{2}}-\tilde{\mathcal{E}}_{t,\varepsilon}^{\sigma^{2}}\bigr)dt  \\
 & -\sigma_{t}(\tilde{\rho}\mathcal{T}_{\varepsilon}\Delta_{t}+\delta_{t}^{\rho}\mathcal{T}_{\varepsilon}\nu_{t})dW^0_{t}+o_{\text{sq}}(1)dt+o_{\text{sq}}(1)dW^0_{t},
\end{align*}
where ${\mathcal{E}}^{g}$ is an error term defined
in Lemma \ref{lem:Error terms half-space}, with $\tilde{\mathcal{E}}^{g}$ defined analogously only for $\tilde{\nu}$, and
\[
\delta_{t}^{g}:=g(t,x,\nu_{t},L_{t})-g(t,x,\tilde{\nu}_{t},\tilde{L}_{t}).
\]
Applying Itô's formula, it follows that
\begin{align}
d(\partial_{x}^{-1}\mathcal{T}_{\varepsilon}\Delta_{t})^{2}= & -2(\partial_{x}^{-1}\mathcal{T}_{\varepsilon}\Delta_{t})\bigl(\tilde{\mathfrak{b}}_{t}\mathcal{T}_{\varepsilon}\Delta_{t}+\delta_{t}^{\mathfrak{b}}\mathcal{T}_{\varepsilon}\nu_{t}\bigr)dt
\label{eq: square of diff}\\
 & +(\partial_{x}^{-1}\mathcal{T}_{\varepsilon}\Delta_{t})\partial_{x}\bigl(\sigma_{t}^{2}\mathcal{T}_{\varepsilon}\Delta_{t}+\mathcal{E}_{t,\varepsilon}^{\sigma^{2}}-\tilde{\mathcal{E}}_{t,\varepsilon}^{\sigma^{2}}\bigr)dt\nonumber \\
  & +\sigma_{t}^{2}\bigl(\tilde{\rho}\mathcal{T}_{\varepsilon}\Delta_{t}+\delta_{t}^{\rho}\mathcal{T}_{\varepsilon}\nu_{t}\bigr){}^{2}dt\nonumber \\
 & -2\alpha(\partial_{x}^{-1}\mathcal{T}_{\varepsilon}\Delta_{t})\sigma_{t}\bigl(\tilde{\rho}\mathcal{T}_{\varepsilon}\Delta_{t}+\delta_{t}^{\rho}\mathcal{T}_{\varepsilon}\nu_{t}\bigr)dW_{t}\nonumber \\
 & +(\partial_{x}^{-1}\mathcal{T}_{\varepsilon}\Delta_{t})o_{\text{sq}}(1)dt+(\partial_{x}^{-1}\mathcal{T}_{\varepsilon}\Delta_{t})o_{\text{sq}}(1)dW_{t}+o_{\text{sq}}(1)^{2}dt.\nonumber 
\end{align}
 We can now integrate the above in $x\in \mathbb{R}_+$ to arrive at the $L^2$ norm of $\partial_{x}^{-1}\mathcal{T}_{\varepsilon}\Delta_{t}$. Then it remains to estimate the resulting integrals on the right-hand side --- a procedure we split into five short steps (deviating substantially from \cite{HL2016}). For convenience, we define
\[
|M|_{t,\star}:=\textstyle{\sup_{s\leq t}}\{1+|M_t|+|\tilde{M}_t|\} \;\; \text{and} \;\; \Vert \mathcal{T}\!\nu \Vert_{t,\star}:=\textstyle{\sup_{s\leq t}\sup_{\varepsilon>0}}\Vert (1+x)\mathcal{T}_\varepsilon \nu_s \Vert_{2}.
\]
\textbf{Step 1. }We start by considering the first line on the right-hand side of (\ref{eq: square of diff}). Note that we can write $2(\partial_{x}^{-1}\mathcal{T}_{\varepsilon}\Delta_{t})(\mathcal{T}_{\varepsilon}\Delta_{t})=\partial_{x}(\partial_{x}^{-1}\mathcal{T}_{\varepsilon}\Delta_{t})^{2}$ in the first term. Hence, we can integrate by parts in $x$ and use Young's inequality
with free parameter $\theta$ to get
\begin{multline*}
-\int_{0}^{\infty}2\tilde{\mathfrak{b}}_{t}(\partial_{x}^{-1}\mathcal{T}_{\varepsilon}\Delta_{t})(\mathcal{T}_{\varepsilon}\Delta_{t})dx=-\int_{0}^{\infty}(\partial_{x}\tilde{\mathfrak{b}}_{t})(\partial_{x}^{-1}\mathcal{T}_{\varepsilon}\Delta_{t})^{2}dx\\
\qquad\qquad\qquad\qquad+2\tilde{\mathfrak{b}}_{t}(0)\int_{0}^{\infty}(\partial_{x}^{-1}\mathcal{T}_{\varepsilon}\Delta_{t})(\mathcal{T}_{\varepsilon}\Delta_{t})dx\\
\leq C_{\mathfrak{b}}\Vert\partial_{x}^{-1}\mathcal{T}_{\varepsilon}\Delta_{t}\Vert_{2}^{2}+C_{\theta}C_{\mathfrak{b}}|M|_{t,\star}^2\Vert\partial_{x}^{-1}\mathcal{T}_{\varepsilon}\Delta_{t}\Vert_{2}^{2}+\theta\Vert \mathcal{T}_{\varepsilon}\Delta_{t}\Vert_{2}^{2}.
\end{multline*}
For the second integral in the first line of (\ref{eq: square of diff}), we recall that
\[
|\delta^{\mathfrak{b}}_t| \apprle (|x|+|M|_{t,\star})\bigl( d_0(\nu_t,\tilde{\nu}_t) +|L_t-\tilde{L}_t| + |\mathfrak{L}_{t}^{\prime}-\tilde{\mathfrak{L}}_{t}^{\prime}|  \bigr).
\]
Thus, using Cauchy-Schwarz on the $d_0$ term and Young's inequality on the others,
\begin{align*}
\Bigl|\int_0^{\infty}\!\delta^{\mathfrak{b}}_t & (\partial_{x}^{-1}\mathcal{T}_{\varepsilon}\Delta_{t}) \mathcal{T}_{\varepsilon}\nu_{t}dx \Bigr| \apprle  |M|_{t,\star}  \Vert \mathcal{T}\!\nu \Vert_{t,\star} \Vert\partial_{x}^{-1}\mathcal{T}_{\varepsilon}\Delta_{t}\Vert_{2}d_0(\nu_t,\tilde{\nu}_t)\\
& + |M|_{t,\star}^2 \Vert\partial_{x}^{-1}\mathcal{T}_{\varepsilon}\Delta_{t}\Vert_{2}^{2} 
+ |M|_{t,\star}^2 \Vert \mathcal{T}\!\nu \Vert_{t,\star}^2 (|L_t-\tilde{L}_t|^2+|\mathfrak{L}_{t}^{\prime}-\tilde{\mathfrak{L}}_{t}^{\prime}|^2 ),
\end{align*}
where we can note that
\begin{align*}
\int_{0}^{t} |M|_{s,\star}^2 \Vert \mathcal{T}\!\nu \Vert_{s,\star}^2 |\mathfrak{L}_{s}^{\prime}-\tilde{\mathfrak{L}}'_{s}|^{2}ds & \leq |M|_{t,\star}^2 \Vert \mathcal{T}\!\nu \Vert_{t,\star}^2 \! \int_{0}^{t}\Bigl|\int_{0}^{s}\!\mathfrak{K}^{\prime}(s-r)(L_{r}-\tilde{L}_{r})dr\Bigr|^{2}ds\\
& \leq |M|_{t,\star}^2 \Vert \mathcal{T}\!\nu \Vert_{t,\star}^2 \left\Vert \mathfrak{K}^{\prime}\right\Vert_{1}^{2} \int_{0}^{t}\bigl|L_{r}-\tilde{L}_{r}\bigr|^{2}dr.
\end{align*}

\textbf{Step 2. }Now consider the second line on the right-hand-side of (\ref{eq: square of diff}). By performing an integration by parts and using Young's inequality with free parameter $\theta$, we get
\begin{multline*}
\int_{0}^{\infty}\!(\partial_{x}^{-1}\mathcal{T}_{\varepsilon}\Delta_{t})\partial_{x}\bigl(\sigma_{t}^{2}\mathcal{T}_{\varepsilon}\Delta_{t}+\mathcal{E}_{t,\varepsilon}^{\sigma^{2}}-\tilde{\mathcal{E}}_{t,\varepsilon}^{\sigma^{2}}\bigr)dx = -\bigl\Vert \sigma_{t}\mathcal{T}_{\varepsilon}\Delta_{t}\bigr\Vert_{2}^{2}-\int_{0}^{\infty}\!(\mathcal{T}_{\varepsilon}\Delta_{t})(\mathcal{E}_{t,\varepsilon}^{\sigma^{2}}-\tilde{\mathcal{E}}_{t,\varepsilon}^{\sigma^{2}})dx\\
\leq
-\bigl\Vert\sigma_{t}\mathcal{T}_{\varepsilon}\Delta_{t}\bigr\Vert_{2}^{2}+C_{\theta}\bigl\Vert\mathcal{E}_{t,\varepsilon}^{\sigma^{2}}-\tilde{\mathcal{E}}_{t,\varepsilon}^{\sigma^{2}}\bigr\Vert_{2}^{2}+\theta\bigl\Vert \mathcal{T}_{\varepsilon}\Delta_{t}\bigr\Vert_{2}^{2},
\end{multline*}
where we have used that $\mathcal{T}_{\varepsilon}\Delta_{t}$, $\mathcal{E}_{t,\varepsilon}^{g}$, and $\tilde{\mathcal{E}}_{t,\varepsilon}^{g}$
are zero at zero and vanish at infinity.

\textbf{Step 3.} In the third line on the right-hand side of (\ref{eq: square of diff}), we expand the square and apply Young's inequality with free parameter $\theta$ to the lower order term. This yields
\begin{align*}
\int_0^\infty \!\sigma_{t}^{2}\bigl(\tilde{\rho}\mathcal{T}_{\varepsilon}\Delta_{t}+\delta_{t}^{\rho}\mathcal{T}_{\varepsilon}\nu_{t}\bigr){}^{2} dx
\leq &
 \, \bigl\Vert \sigma_{t}\tilde{\rho}_t\mathcal{T}_{\varepsilon}\Delta_{t} \bigr\Vert_{2}^2
+ C_\sigma|\delta^{\rho}_t|^2\bigl\Vert \mathcal{T}_{\varepsilon}\nu_{t} \bigr\Vert_{2}^2  \\
 & \;+ C_{\theta}|\delta^{\rho}|^2\bigl\Vert \mathcal{T}_{\varepsilon}\nu_{t} \bigr\Vert_{2}^2 + \theta  \bigl\Vert \mathcal{T}_{\varepsilon}\Delta_{t} \bigr\Vert_{2}^2.
\end{align*}
Also, we recall here that $|\delta^{\rho}_t|^2\apprle |M|_{t,\star}^2(|L_{t}-\tilde{L}_{t}|^{2}+d_1(\nu_t,\tilde{\nu}_t)^{2})$.

\textbf{Step 4.} When taking expectation in (\ref{eq: square of diff}),
the stochastic integrals vanish. Thus, by taking expectation and integrating over $x>0$, it follows from Steps 1--3 and Young's inequality with
free parameter $\theta>0$ that
\begin{multline}\label{eq: L2 norm of antiderivative}
\mathbb{E}\left\Vert \partial_{x}^{-1}\mathcal{T}_{\varepsilon}\Delta_{t}\right\Vert _{2}^{2}
\leq 
C_{\theta}\mathbb{E} \int_{0}^{t}\! \Vert \mathcal{T}\!\nu \Vert_{s,\star}^{2}|M|_{s,\star}^2 \bigl(|L_{s}-\tilde{L}_{s}|^{2}+d_1(\nu_s,\tilde{\nu}_s)^{2}\bigr)ds
\\
+C_{\theta}\mathbb{E}\int_{0}^{t}\!|M|_{s,\star}^2\left\Vert \partial_{x}^{-1}\mathcal{T}_{\varepsilon}\Delta_{s}\right\Vert_2^2 ds
+ C_\theta \mathbb{E} \int_0^t \!|M|_{s,\star}  \Vert \mathcal{T}\!\nu \Vert_{s,\star} \bigl\Vert\partial_{x}^{-1}\mathcal{T}_{\varepsilon}\Delta_{s}\bigr\Vert_{2}d_0(\nu_s,\tilde{\nu}_s) ds
\\
+ \mathbb{E}\int_{0}^{t} \int_0^\infty \! \{ \sigma_{s}^{2}\tilde{\rho}_{s}^{2}-\sigma_{s}^{2}+2\theta \} | \mathcal{T}_{\varepsilon}\Delta_{s}|^{2}dxds+o(1),
\end{multline}
as $\varepsilon\downarrow0$, for a constant $C_{\theta}$ which only depends  on the free parameter $\theta$.

\textbf{Step 5.} Since $\rho$ is bounded away from $1$ and $\sigma$ is bounded away from $0$, we can take the
free parameter $\theta$ sufficiently small so that (for all $x$ and $t$)
\begin{equation}
 \sigma(s,x)^{2}\rho(t,\tilde{\nu}_t)^{2}-\sigma(s,x)^{2}+2\theta \leq -c_0
\label{eq: small theta trick}
\end{equation}
for a fixed constant $c_0>0$.
Next, we can consider the stopping times
\[
t_{n}:=\inf\bigl\{ t>0: \Vert \mathcal{T}\!\nu\Vert_{t,\star}^{2}>n\;\;\text{or}\;\;|M|_{t,\star}^2>n\bigr\}\wedge T,
\]
for $n\geq1$, and notice that, by Proposition \ref{prop: Energy bound on particle limit},
we have $t_{n}\uparrow T$ as $n\rightarrow\infty$. Evaluating the
estimate (\ref{eq: L2 norm of antiderivative}) at
$t\land t_{n}$ and using (\ref{eq: small theta trick}), we get
\begin{align*}
&\mathbb{E}\left\Vert \partial_{x}^{-1}\mathcal{T}_{\varepsilon}\Delta_{t\wedge t_{n}}\right\Vert _{2}^{2} + c_0\mathbb{E} \int_0^{t\wedge t_{n}}\!\left\Vert \mathcal{T}_\varepsilon\Delta_{s}\right\Vert _{2}^{2}ds
\leq
Cn\!\int_{0}^{t}\!\mathbb{E}\left\Vert \partial_{x}^{-1}\mathcal{T}_{\varepsilon}\Delta_{s\land t_{n}}\right\Vert_{2}^{2}ds +o(1)\\
 & + Cn^{2}\mathbb{E}\int_{0}^{t\land t_{n}}\!|L_{s}-\tilde{L}_{s}|^{2}+d_1(\nu_s,\tilde{\nu}_s)^{2}ds
+Cn^2\mathbb{E} \int_0^{t\land t_{n}} \! d_0(\nu_s,\tilde{\nu}_s)\bigl\Vert\partial_{x}^{-1}\mathcal{T}_{\varepsilon}\Delta_{s}\bigr\Vert_{2}ds.
\end{align*}
Finally, by applying the integrating factor $\exp\{ -Cnt\}$ to the first term on the right-hand side, we obtain the estimate from Proposition \ref{smooth_energy_est} with $c_n:=n^2(e^{CT}-1)$.

\section{Density Estimates\label{sec:Proofs Density and Prob}}

The purpose of this section is to prove the density estimates stated in Proposition~\ref{prop: Aronson estimate}. Our approach will rely on techniques that are entirely probabilistic in nature and we follow a simple intuitive procedure.

The first step is to ensure sufficiently fast decay of the tails of the particles. This is achieved in Section 6.1, where we show that the sub-Gaussianity of the initial law propagates nicely for all positive times. In Section 6.2, we transform the particles into Brownian motions with drift that have the same hitting times of the origin, and then we use the sub-Gaussianity to introduce a change of measure related to the drift.

Finally, we derive the density estimates in Section 6.3, by comparing
any given transformed particle at time $t>0$ with an independent absorbed
Brownian motion (under the original measure) started from the transformed particle's position at earlier times
$s<t$ and run for the remaining time $t-s$.

\subsection{Sub-Gaussianity}

For notational convenience, we define the dominating processes
\begin{equation}\label{eq: dominating process}
\Lambda_{t}^{i,{\scriptscriptstyle N}}:=|X_{t}^{i}|+{\textstyle \sum_{j=1}^{{\scriptscriptstyle N}}}a_{j}^{{\scriptscriptstyle N}}|X_{t}^{j}|, \quad	\Gamma_{t}^{i,{\scriptscriptstyle N}}:=|X_{t}^{i}|^{2}+{\textstyle \sum_{j=1}^{{\scriptscriptstyle N}}}a_{j}^{{\scriptscriptstyle N}}|X_{t}^{j}|^{2}.
\end{equation}
Notice that, by Assumption \ref{Assumption 1 - finite system}, we have
\[
|X_{t}^{i}|,|M_{t}^{{\scriptscriptstyle N}}|,|b(t,X_{t}^{i},\nu_{t}^{{\scriptscriptstyle N}})|\apprle1+\Lambda_{t}^{i,{\scriptscriptstyle N}}.
\]
\begin{prop}
	\label{prop:SUBGAUSSIAN} Let Assumption \ref{Assumption 1 - finite system}
	be satisfied. For every $\epsilon>0$,
	there is a (smooth) decreasing function $t\mapsto\eta_t$ with $\eta_0=\epsilon/2$ and $\eta_t>0$ such that, for all $t>0$ and $N\geq1$,
	\begin{equation}
	\label{eq: Gamma_t SubGaussian}
	\mathbb{E}e^{\eta_{t}{\Gamma}_t^{i,{\scriptscriptstyle N}}}\leq e^{c\int_{0}^{t}\!\eta_{s}ds}\mathbb{E}\Bigl[e^{\epsilon|X_{0}^{i}|{}^{2}}\Bigr]^{\frac{1}{2}} \mathbb{E}\Bigl[e^{\epsilon\sum_{j=1}^{N}a_{j}^{N}|X_{0}^{j}|{}^{2}}\Bigr]^{\frac{1}{2}}.
	\end{equation}
	In particular, $X_t^i$, $M_t^{\scriptscriptstyle N}$, and $\Lambda_{t}^{i,{\scriptscriptstyle N}}$ \hspace{-5pt} are all sub-Gaussian uniformly in $N\geq1$.
\end{prop}

\begin{proof}
First of all, the only fact we will use about the drift is $|b_t^i|\apprle 1+\Lambda_t^{i,\scriptscriptstyle N}$, so noting that $|b_t^i-(\mathfrak{L}^{\scriptscriptstyle N})'_t\alpha^i_t|\apprle1+\Lambda_t^{i,\scriptscriptstyle N}$, we can assume without loss of generality that $\alpha_t^i\equiv0$.
	
	Fix a strictly positive (deterministic) function $\eta\in\mathcal{C}^{1}(\mathbb{R}_+)$
	with $\eta_{0}=\epsilon/2$ and define $\Psi\in\mathcal{C}^{1,2}(\mathbb{R}_+\times\mathbb{R}^{N})$
	by
	\[
	\Psi(t,x):=e^{\eta_{t}(x_{i}^{2}+\sum_{j}a_{j}^{{\scriptscriptstyle N}}x_{j}^{2})}.
	\]
	For simplicity of notation, we write $\Gamma_{t}$ in place of $\Gamma_{t}^{i,{\scriptscriptstyle N}}$. Then $\Psi(t,X_{t})=e^{\eta_{t}\Gamma_{t}}$, where we have defined $X_{t}:=(X_{t}^{1},\ldots,X_{t}^{\scriptscriptstyle N})$.
Introducing the stopping times
\[
\tau_{n}:=\inf\left\{ t\geq0:\left|X_{t}\right|>n\right\} \quad\text{for}\quad n\geq1,
\]
and applying Itô's formula, we get
\begin{equation}
e^{\eta_{t\wedge\tau_{n}}\Gamma_{t\wedge\tau_{n}}}=e^{\eta_{0}\Gamma_{0}}+\int_{0}^{t\land\tau_{n}}\!(\partial_{t}+\mathcal{L})e^{\eta_{s}\Gamma_{s}}ds+\sum_{j=1}^{N}\int_{0}^{t\land\tau_{n}}\!\sigma_{s}^{j}\partial_{x^{j}}e^{\eta_{s}\Gamma_{s}}dB_{s}^{j},\label{eq: gronwall start}
\end{equation}
where $dB_{s}^{j}=\rho_{t}dW_{t}^{0}+\sqrt{1-\rho_{t}^{2}}dW_{t}^{j}$
and
\[
\mathcal{L}\Psi=\sum_{j=1}^{N}b_{t}^{j}\partial_{x_{j}}\Psi+\frac{1}{2}\sum_{j=1}^{N}(\sigma_{t}^{j})^{2}\partial_{x_{j}}^{2}\Psi+\frac{1}{2}\sum_{k\neq j}\sigma_{t}^{j}\sigma_{t}^{k}\rho_{t}^{2}\partial_{x_{j}x_{k}}\Psi.
\]
Computing the derivatives, we see that
\begin{align*}
(\partial_{t}+\mathcal{L})e^{\eta_{t}\Gamma_{t}}= & e^{\eta_{t}\Gamma_{t}}\biggl(\dot{\eta}_{t}\Gamma_{t}+2\eta_{t}\sum_{j=1}^{N}b_{t}^{j}(\delta_{ij}+a_{j}^{{\scriptscriptstyle N}})X_{t}^{j}\:+\\
 & \qquad+\,\eta_{t}\sum_{j=1}^{N}(\sigma_{t}^{j})^{2}(\delta_{ij}+a_{j})+2\eta_{t}^{2}\sum_{j=1}^{N}(\sigma_{t}^{j})^{2}(\delta_{ij}+a_{j}^{{\scriptscriptstyle N}})^{2}(X_{t}^{j})^{2}\\
 & \qquad+\,2\eta_{t}^{2}\rho_{t}^{2}\sum_{k\neq j}\sigma_{t}^{j}\sigma_{t}^{k}(\delta_{ij}+a_{j}^{{\scriptscriptstyle N}})(\delta_{ik}+a_{k}^{{\scriptscriptstyle N}})X_{t}^{j}X_{t}^{k}\biggr).
\end{align*}
Using the bounds $|\sigma^{j}|\leq C_{\sigma}$ and $|b^{j}|\leq C_{b}(1+|X_{t}^{j}|+\sum_{\ell}a_{\ell}^{{\scriptscriptstyle N}}|X_{t}^{\ell}|)$
as well as the basic inequalities
\[
x\leq1+x^{2},\quad(x+y)^{2}\leq2(x^{2}+y^{2}),\quad\text{and}\quad\bigl({\textstyle \sum}_{j=1}^{N}a_{j}^{{\scriptscriptstyle N}}x_{j}\bigr)^{2}\leq{\textstyle \sum}_{j=1}^{N}a_{j}^{{\scriptscriptstyle N}}x_{j}^{2},
\]
one then easily verifies that
\[
(\partial_{t}+\mathcal{L})e^{\eta_{t}\Gamma_{t}}\leq\Bigl(2\eta_{t}(C_{b}+C_{\sigma}^{2})+(\dot{\eta}_{t}+10C_{b}\eta_{t}+8C_{\sigma}^{2}\eta_{t}^{2})\Gamma_{t}\Bigr)e^{\eta_{t}\Gamma_{t}}.
\]
Consequently, if only we can choose $\eta\in\mathcal{C}^{1}(\mathbb{R}_+)$
such that
\[
\dot{\eta}_{t}+10C_{b}\eta_{t}+8C_{\sigma}^{2}\eta_{t}^{2}=0,\quad\eta_{0}=\epsilon/2,
\]
then we have
\begin{equation}
(\partial_{t}+\mathcal{L})e^{\eta_{t}\Gamma_{t}}\leq2\eta_{t}(C_{\sigma}^{2}+C_{b})e^{\eta_{t}\Gamma_{t}}.\label{eq: d_t + L}
\end{equation}
This is achieved by taking
\begin{equation}
\eta_{t}:=\frac{5C_{b}\epsilon}{10C_{b}e^{10C_{b}t}+4C_{\sigma}^{2}\epsilon(e^{10C_{b}t}-1)}.\label{eq: definition of eta_t}
\end{equation}
With this choice for $\eta$, the inequality (\ref{eq: d_t + L}) is satisfied and
thus (\ref{eq: gronwall start}) gives that
\begin{align}
e^{\eta_{t\wedge\tau_{n}}\Gamma_{t\wedge\tau_{n}}} & \leq e^{\eta_{0}\Gamma_{0}}+2(C_{\sigma}^{2}+C_{b})\int_{0}^{t\land\tau_{n}}\!\eta_{s}e^{\eta_{s}\Gamma_{s}}ds+\sum_{j=1}^{N}\int_{0}^{t\land\tau_{n}}\!\sigma_{s}^{j}\partial_{x_{j}}e^{\eta_{s}\Gamma_{s}}dB_{s}^{j}.\label{eq: Key estimate for Gronwall}
\end{align}
Taking expectations and noting that the stochastic integral is a
true martingale,
\[
\mathbb{E}\hspace{-1pt}\left[e^{\eta_{t\land\tau_{n}}\Gamma_{t\wedge\tau_{n}}}\right]\leq\mathbb{E}\hspace{-1pt}\left[e^{\eta_{0}\Gamma_{0}}\right]+2(C_{\sigma}^{2}+C_{b})\int_{0}^{t}\eta_{s}\mathbb{E}\hspace{-1pt}\left[e^{\eta_{s\land\tau_{n}}\Gamma_{s\wedge\tau_{n}}}\right]ds.
\]
By Cauchy\textendash Schwarz (recall $\eta_{0}=\epsilon/2$),
we have
\[
\mathbb{E}\hspace{-1pt}\left[e^{\eta_{0}\Gamma_{0}}\right]\leq\mathbb{E}\Bigl[e^{\epsilon|X_{0}^{i}|{}^{2}}\Bigr]^{\frac{1}{2}}\mathbb{E}\Bigl[e^{\epsilon\sum_{j}a_{j}^{{\scriptscriptstyle N}}|X_{0}^{j}|{}^{2}}\Bigr]^{\frac{1}{2}}
\]
and hence it follows from Gronwall's inequality that
\begin{align}
\mathbb{E}\hspace{-1pt}\left[e^{\eta_{t\land\tau_{n}}\Gamma_{t\wedge\tau_{n}}}\right] & \leq e^{2(C_{\sigma}^{2}+C_{b})\int_{0}^{t}\eta_{s}ds}\mathbb{E}\Bigl[e^{\epsilon|X_{0}^{i}|{}^{2}}\Bigr]^{\frac{1}{2}}\mathbb{E}\Bigl[e^{\epsilon\sum_{j}a_{j}^{{\scriptscriptstyle N}}|X_{0}^{j}|{}^{2}}\Bigr]^{\frac{1}{2}},\label{eq: gronwall estimate}
\end{align}
Noting that $\tau_{n}\uparrow\infty$ as $n\rightarrow\infty$ , we have $e^{\eta_{t\land\tau_{n}}\Gamma_{t\wedge\tau_{n}}}\rightarrow e^{\eta_{t}\Gamma_{t}}$ as $n\rightarrow\infty$, by the continuity of $t\mapsto\eta_{t}$ and $t\mapsto X_{t}^{j}$,
for $j=1,\ldots,N$. Consequently, Fatou's lemma yields
\[
\mathbb{E}\hspace{-1pt}\left[e^{\eta_{t}\Gamma_{t}}\right]\leq\liminf_{n\rightarrow\infty}\mathbb{E}\hspace{-1pt}\left[e^{\eta_{t\land\tau_{n}}\Gamma_{t\wedge\tau_{n}}}\right]\leq e^{2(C_{\sigma}^{2}+C_{b})\int_{0}^{t}\eta_{s}ds}\mathbb{E}\Bigl[e^{\epsilon|X_{0}^{i}|{}^{2}}\Bigr]^{\frac{1}{2}}\mathbb{E}\Bigl[e^{\epsilon\sum_{j}a_{j}^{{\scriptscriptstyle N}}|X_{0}^{j}|{}^{2}}\Bigr]^{\frac{1}{2}}.
\]
Since $\eta_{t}$ does not depend on $N$, the estimate is uniform
in $N\geq1$.
\end{proof}
The above result also yields control over the running maxima of the
processes.
\begin{cor}
\label{cor: Sup is subgaussian}Given $\eta_{T}$ from Proposition
\ref{prop:SUBGAUSSIAN}, fix any $\epsilon<\eta_{T}/2$. Then we have

\[
\mathbb{E}\Bigl[e^{\epsilon\sup_{t \leq T}\Gamma_{t}^{i,N}}\Bigr]\leq C_{T,\epsilon},
\]
for some constant $C_{T,\epsilon}>0$ that is uniform in $N\geq1$.
\end{cor}

\begin{proof} Set $\xi_{t}:=p\eta_{t}/2$ for $p<1$ and derive the estimate (\ref{eq: Key estimate for Gronwall}) for $\xi$. As $p<1$, $|\partial_{x_j}e^{\xi_s\Gamma_s}|^2$ is integrable by Proposition \ref{prop:SUBGAUSSIAN}, so we can apply Burkholder--Davis--Gundy to control the running max of the stochastic integral. Thus, taking expectation and using monotone convergence, we get
	\[
	\mathbb{E}\Bigl[e^{\frac{p\eta_T}{2}\sup_{t\leq T}\Gamma_{t}}\Bigr]\leq\mathbb{E}\Bigl[\sup_{t\leq T}e^{\xi_{t}\Gamma_{t}}\Bigr]\leq C_{T,p}.
	\]
	As $p<1$ was arbitrary, the claim follows.
\end{proof}

\subsection{Change of measure}

In this section we first transform each particle into a Brownian motion with drift in a way that
preserves the hitting times of the origin. Next, we then use the sub-Gaussianity to introduce a change-of-measure that can remove this drift and finally we obtain an
important estimate for the associated Radon-Nikodym derivative.
\begin{lem}
\label{lem:=00005BRestatement-of-Lemma4.1 from hambly ledger}Define
the transformation $\Upsilon\in C^{1,2}([0,T]\times\mathbb{R})$ by
\[
(t,x)\mapsto\Upsilon_{t}(x):=\int_{0}^{x}\frac{1}{\sigma(t,y)}dy.
\]
Fixing an arbitrary index $i\in\left\{ 1,\ldots,N\right\} $, we let
$Z_{t}:=\Upsilon_{t}(X_{t}^{i})$. Then
\[
dZ_{t}=\hat{b}_{t}^{i}dt+dB_{t}^{i}\quad\text{with}\quad Z_{0}=\Upsilon_{0}(X_{0}^{i}),
\]
where $B^{i}$ is a Brownian motion and the (stochastic) drift $\hat{b}^{i}_t$
obeys the growth condition
\begin{equation}
|\hat{b}_{t}^{i}|\apprle1+|X_{t}^{i}|+|M_{t}^{{\scriptscriptstyle N}}|.\label{eq:bound on D_t growth}
\end{equation}
Furthermore, the transformed process $Z$ satisfies
\begin{equation}
\text{\emph{sgn}}(Z_{t})=\text{\emph{sgn}}(X_{t}^{i})\quad\text{\emph{and}}\quad|Z_{t}|\apprle|X_{t}^{i}|.\label{eq: bound on Z_t growth}
\end{equation}
\end{lem}

\begin{proof}
Note that
\[
\partial_{t}\Upsilon_{t}(x)=-\int_{0}^{x}\frac{\partial_{t}\sigma(t,y)}{\sigma(t,y)^{2}}dy,\quad\partial_{x}\Upsilon_{t}(x)=\frac{1}{\sigma(t,x)},\quad\partial_{xx}^{2}\Upsilon_{t}(x)=-\frac{\partial_{x}\sigma(t,x)}{\sigma(t,x)^{2}}.
\]
Defining $B_{t}^{i}:=\int_{0}^{t}\rho_{s}dW_{s}^{0}+\int_{0}^{t}\!\sqrt{1-\rho_{s}^2}dW_{s}^{i}$,
we have $dX_{t}=(b^i_{t}-\alpha^i_{t}(\mathfrak{L}^{{\scriptscriptstyle N}})_{t}^{\prime})dt+\sigma_{t}dB_{t}$
with $d\langle X\rangle _{t}=\sigma_{t}^{2}dt$. Hence
an application of Itô's formula yields $dZ_{t}=\hat{b}_{t}^{i}dt+dB_{t}^{i}$,
where
\[
\hat{b}_{t}^{i}:=\frac{b_t(X_{t}^{i})-\alpha_{t}(X_{t}^{i})(\mathfrak{L}^{{\scriptscriptstyle N}})_{t}^{\prime}}{\sigma(t,X_{t}^{i})}-\frac{1}{2}\partial_{x}\sigma(t,X_{t}^{i})-\int_{0}^{X_{t}^{i}}\frac{\partial_{t}\sigma(t,y)}{\sigma(t,y)^{2}}dy.
\]
Now, the bound on $Z_{t}$ and the statement about its sign in (\ref{eq: bound on Z_t growth})
follow directly from the definition of $\Upsilon$, since $\sigma$
is strictly positive and bounded away from zero. Similarly, the growth
condition in (\ref{eq:bound on D_t growth}) follows from the properties of the coefficients
(Assumption~\ref{Assumption 1 - finite system}) and the fact that
$|(\mathfrak{L}^{{\scriptscriptstyle N}})_{t}^{\prime}|\leq\int_{0}^{t}|\mathfrak{K}^{\prime}(t-s)|L_{s}^{{\scriptscriptstyle N}}ds\leq\left\Vert \mathfrak{K}^{\prime}\right\Vert _{L^{1}}$.
\end{proof}
\begin{lem}
\label{lem:Stochastic exp and equivalnet measu}Fix $i\in\left\{ 1,\ldots,N\right\} $
and define the stochastic exponential
\[
\mathcal{E}_{t}:=\exp\left\{ -\int_{0}^{t}\hat{b}_{s}^{i}dB_{s}^{i}-\frac{1}{2}\int_{0}^{t}(\hat{b}_{s}^{i})^{2}ds\right\} ,
\]
where $\hat{b}^{i}$ is the drift of $Z$ as defined in Lemma
\ref{lem:=00005BRestatement-of-Lemma4.1 from hambly ledger}. Then
$Z$ is a Brownian motion under the probability measure $\mathbb{Q}$
given by the Radon-Nikodym derivative
\[
\frac{d\mathbb{Q}}{d\mathbb{P}}\Bigr|_{\mathcal{F}_{t}}=\mathcal{E}_{t}
\]
with initial value $Z_{0}$ distributed according to $\varpi_{0}:=\mu_{0}\circ\Upsilon_{0}^{-1}$.
\end{lem}

\begin{proof}
For notational convenience we drop the superscript $i$ and define
\[
\mathcal{E}_{s,t}:=\exp\left\{ -\int_{s}^{t}\hat{b}_{u}dB_{u}-\frac{1}{2}\int_{s}^{t}\hat{b}_{u}^{2}du\right\} ,\quad\mathcal{E}_{0,t}=\mathcal{E}_{t}.
\]
By standard arguments, $\mathcal{E}_{t}$ is a positive continuous
local martingale and hence also a supermartingale with $\mathcal{E}_{0}=1$.
The claim is now that $\mathcal{E}_{t}$ is in fact a true martingale
on $[0,T]$, which amounts to showing that $\mathbb{E}\mathcal{E}_{t}=1$
for all $t\in[0,T]$.

Recall that $|\hat{b}_{s}|\leq C(1+\varLambda_{s}^{i,\scriptscriptstyle N})$. While we cannot appeal to Novikov's condition directly, the sub-Gaussianity
of $\varLambda_{s}^{i,{\scriptscriptstyle N}}$ will allows us to
apply it on every interval of a fine enough partition of $[0,T]$.
To see this, we fix an arbitrary $n\geq1$ and partition the
interval $[0,T]$ by $t_{0}<\cdots<t_{n}$ where $t_{k}:=kT/n$.
An application of Jensen's inequality then yields
\begin{align*}
\mathbb{E}\exp\Bigl\{\frac{1}{2}\int_{t_{k-1}}^{t_{k}}\hat{b}_{u}^{2}du\Bigr\} & \leq \frac{n}{T}\mathbb{E}\int_{t_{k-1}}^{t_{k}}\exp\Bigl\{\frac{T}{2n}\hat{b}_{u}^{2}\Bigr\} du=\frac{n}{T}\int_{t_{k-1}}^{t_{k}}\mathbb{E}\exp\Bigl\{\frac{T}{2n}\hat{b}_{u}^{2}\Bigr\} du\\
 & \leq\sup_{t\in[t_{k-1},t_k]}\mathbb{E}\exp\Bigl\{\frac{C^{2}T}{n}(1+\varLambda_{t}^{i,{\scriptscriptstyle N}})^{2}\Bigr\}.
\end{align*}
Choosing $n\geq1$ sufficiently large so that $2C^{2}T/n\leq\eta_{T}$,
we deduce from Proposition \ref{prop:SUBGAUSSIAN} that
\[
\mathbb{E}\exp\Bigl\{\frac{1}{2}\int_{t_{k-1}}^{t_{k}}\hat{b}_{u}^{2}du\Bigr\}<\infty,\quad\text{for each}\quad k=1,\ldots,n.
\]
In particular, Novikov's condition now implies that $(\mathcal{E}_{t})_{t\in[0,t_{1}]}$
is a true martingale and hence
\[
\mathbb{E}\mathcal{E}_{t}=1,\quad\text{for all}\quad t\in[0,t_{1}].
\]
Noting that $(\mathcal{E}_{t_{1},t})_{t\in[t_{1},t_{2}]}$ is again
a stochastic exponential with $\mathcal{E}_{t_{1},t_{1}}=1$, another
application of Novikov's condition shows that $(\mathcal{E}_{t_{1},t})_{t\in[t_{1},t_{2}]}$
is a martingale. Consequently, we have
\[
\mathbb{E}\left[\mathcal{E}_{t_{1},t}\mid\mathcal{F}_{t_{1}}\right]=\mathcal{E}_{t_{1},t_{1}}=1
\]
and hence
\[
\mathbb{E}\mathcal{E}_{t}=\mathbb{E}\left[\mathcal{E}_{t_{1}}\mathcal{E}_{t_{1},t}\right]=\mathbb{E}\left[\mathcal{E}_{t_{1}}\mathbb{E}\left[\mathcal{E}_{t_{1},t}\mid\mathcal{F}_{t_{1}}\right]\right]=\mathbb{E}\left[\mathcal{E}_{t_{1}}\right]=1
\]
for all $t\in[t_{1},t_{2}]$. Considering inductively the intervals
$[t_{k-1},t_{k}]$ for $k=3,\ldots,n$, it follows by the same reasoning
that $\mathbb{E}\mathcal{E}_{t}=1$ for all  $t\in[0,T]$. Thus;, $\mathcal{E}_{t}$ is a true martingale on $[0,T]$ and hence
Girsanov's theorem implies that the process
\[
dZ_{t}=\hat{b}_{t}dt+dB_{t},\quad Z_{0}=\Upsilon_{0}(X_{0}^{1}),
\]
is a Brownian motion under $\mathbb{Q}$ with initial distribution
$\varpi_{0}=\mu_{0}\circ\Upsilon_{0}^{-1}$.
\end{proof}
Other than $Z$ being a Brownian motion under $\mathbb{Q}$, we shall
also need a specific estimate on the Radon\textendash Nikodym derivative,
$\mathcal{E}_{t}$, from the previous lemma.
\begin{lem}
[Radon--Nikodym estimate]\label{lem: Radon-Nik estimate}Let $\mathcal{E}_{t}$
be the stochastic exponential defined in Lemma \ref{lem:Stochastic exp and equivalnet measu}.
Given any $\epsilon>0$ such that $\mathbb{E}\exp\{\epsilon|X_{0}^{j}|{}^{2}\}<\infty$,
for $j=1,\ldots,N$, there exists $p>1$ close enough to $1$ such
that
\[
\mathbb{E}\left[\mathcal{E}_{t}^{1-p}\mid X_{0}^{i}=x_{0}\right]\leq C\exp\{ \epsilon x_{0}^{2}\}, \quad\text{where}\quad C=C(p,\epsilon,T).
\]
Moreover, for any $q>1$, we have
\[
\mathbb{E}\left[\bigl|\hat{b}_{t}^{i}\bigr|^{q}\mid X_{0}^{i}=x_{0}\right]\leq C(1+x_{0}^{q}),\quad\text{where}\quad C=C(q,T).
\]
\end{lem}

\begin{proof}
We begin by defining
\[
Y_{t}:=-\int_{0}^{t}\hat{b}_{s}^{i}dB_{s}^{i}.
\]
Using Hölder's inequality with $p>1$ , we then have
\begin{align*}
\mathbb{E}\left[\mathcal{E}_{t}^{1-p}\mid X_{0}^{i}=x_{0}\right] & =\mathbb{E}\left[\left(e^{pY_{t}-\frac{1}{2}\left\langle pY\right\rangle _{t}}\right)^{\frac{p-1}{p}}\left(e^{\frac{(p+1)}{2}(p-1)p\left\langle Y\right\rangle _{t}}\right)^{\frac{1}{p}}\mid X_{0}^{i}=x_{0}\right]\\
 & \leq\mathbb{E}\left[e^{pY_{t}-\frac{1}{2}\left\langle pY\right\rangle _{t}}\mid X_{0}^{i}=x_{0}\right]^{\frac{p-1}{p}}\mathbb{E}\left[e^{\frac{(p+1)}{2}(p-1)p\left\langle Y\right\rangle _{t}}\mid X_{0}^{i}=x_{0}\right]^{\frac{1}{p}}
\end{align*}
Noting that the first term on the right-hand side is bounded by $1$,
we conclude that
\[
\mathbb{E}\left[\mathcal{E}_{t}^{1-p}\mid X_{0}^{i}=x_{0}\right]\leq\mathbb{E}\left[e^{C_{p}\int_{0}^{t}(\hat{b}_{s}^{i})^{2}ds}\mid X_{0}^{i}=x_{0}\right]^{\frac{1}{p}}\quad\text{with}\quad C_{p}:=\frac{(p+1)}{2}(p-1)p.
\]
The crucial observation here is that $C_{p}\downarrow0$ as $p\downarrow1$.
Recalling the bound $\bigl|\hat{b}_{s}^{i}\bigr|\apprle1+\varLambda_{s}^{i,\scriptscriptstyle N}$,
we can apply Jensen's inequality to see that
\begin{align}
\mathbb{E}\left[\mathcal{E}_{t}^{1-p}\mid X_{0}^{i}=x_{0}\right] & \apprle \biggl(\frac{1}{T}\int_{0}^{T}\mathbb{E}\left[e^{TCC_{p}(\varLambda_{s}^{i,N})^{2}}\mid X_{0}^{i}=x_{0}\right]ds\biggr)^{\frac{1}{p}}\label{eq: prob X < prob BM}
\end{align}
Fix a power $p_{0}$ close enough to $1$ such that $TCC_{p_{0}}\leq\eta_{T}/2$,
with $\eta$ as in Proposition \ref{prop:SUBGAUSSIAN}, so
that
\[
\mathbb{E}\left[\mathcal{E}_{t}^{1-p_{0}}\mid X_{0}^{i}=x_{0}\right]\apprle\sup_{s\in\left[0,T\right]}\mathbb{E}\left[e^{\eta_{s}(\Lambda_{s}^{i,N})^{2}/2}\mid X_{0}^{i}=x_{0}\right]^{\frac{1}{p_{0}}}.
\]
Now, recalling the form of the estimate (\ref{eq: gronwall estimate})
in the proof of Proposition \ref{prop:SUBGAUSSIAN}, the fact that
$X_{0}^{i}$ is independent of $X_{0}^{j}$ for $j\neq i$ implies that
\[
\sup_{s\in\left[0,T\right]}\mathbb{E}\left[e^{\eta_{T}(\Lambda_{s}^{i,N})^{2}/2}\mid X_{0}^{i}=x_{0}\right]\leq Ce^{\epsilon x_{0}^{2}/p_{0}}\mathbb{E}\left[e^{\epsilon\sum_{j\neq i}a_{j}^{{\scriptscriptstyle N}}(X_{0}^{j})^{2}}\right]^{\frac{1}{2p_{0}}}.
\]
Combining this with the estimate (\ref{eq: prob X < prob BM}),
the first claim of the proposition follows.

Finally, using the bound $|\hat{b}_{t}^{i}|\apprle1+\Lambda_{t}^{i,{\scriptscriptstyle N}}$,
the second claim follows by a standard Gronwall argument, so we leave out the proof.
\end{proof}

\subsection{The density estimates}

The purpose of this section is to derive the desired density estimates for $X_t^i$ by controlling the probability $\mathbb{P}(X_{t}^{i}\in S,\,t<\tau_{i})$, for any given $i\in \left\{ 1,\ldots,N\right\}$ and $S\in \mathcal{B}(0,\infty)$.

Recall the transformation $\Upsilon$ from Lemma \ref{lem:=00005BRestatement-of-Lemma4.1 from hambly ledger} and note that $x\mapsto\Upsilon_{t}(x)$ is bijective with $\Upsilon_{t}(x)\leq0$
if and only if $x\leq0$. Hence, conditioning on the initial
value $X_{0}^{i}=x_{0}$, 
\[
\mathbb{P}^{x_{0}}(X_{t}^i\in S,\,t<\tau_{i})=\mathbb{P}^{z}(Z_{t}\in S_{t},\,t<\tau),
\]
where $\tau=\inf\{t>0:Z_{t}\leq0\}$, $Z_t=\Upsilon(X_t^i)$, $S_{t}=\Upsilon_{t}(S)$, and $z=\Upsilon_{0}(x_{0})$.

From here, the idea is to approximate the transformed particle $Z_{t\land\tau}$
by running it up to time $s$, for $s<t$, and then running an
independent absorbed Brownian motion $W$ for the remaining time $t-s$.
More precisely, given $z\in(0,\infty)$, we are interested in the
map
\begin{equation}
s\mapsto\mathbb{E}^{z}\!\left[\mathbb{P}^{Z_{s\land\tau}}(W_{t-s}\in S_{t},\;t-s<\tau_{{\scriptscriptstyle W}})\right],
\label{eq: Approx Zt}
\end{equation}
where $\tau_{\scriptscriptstyle W}=\inf\{t>0:W_{t}\leq0\}$. For a fixed time $t\in[0,T]$, we therefore define, for every $s<t$, the function
\[
u(s,x):=\mathbb{P}^{x}\bigl(W_{t-s}\in S_{t},\;t-s<\tau_{{\scriptscriptstyle W}}\bigr)=\int_{S_{t}}G_{t-s}(y,x)dy,
\]
where $G_{t}(y,x)=p_{t}(x-y)-p_{t}(x+y)$ with $p_{t}(x)=(2\pi t)^{-\frac{1}{2}}\exp\{-x^{2}/2t\}$.
Note that $u$ is a classical solution of the terminal-boundary value problem
\begin{equation}
\begin{cases}
\partial_{s}u(s,x)+\frac{1}{2}\Delta u(s,x)=0 & \text{on}\quad[0,t)\times(0,\infty)\\
u(t,x)=\mathbf{1}_{S_{t}}(x) & \text{on}\quad\left\{ t\right\} \times(0,\infty)\\
u(s,0)=0 & \text{on}\quad[0,t)\times\left\{ 0\right\} 
\end{cases}\label{eq:Backward Heat}
\end{equation}
We can write (\ref{eq: Approx Zt}) more succinctly as
\begin{equation}
s\mapsto v(s,z):=\mathbb{E}^{z}[u(s,Z_{s\land\tau})],
\label{v_function}
\end{equation}
and note that
\[
v(0,z)=\mathbb{P}^{z}(W_{t\land{\tau}_{W}}\in S_{t})\quad\text{and}\quad v(t,z)=\mathbb{P}^{z}(Z_{t\land\tau}\in S_{t}).
\]
By bounded convergence, it is immediate that $s\mapsto v(s,z)$ is
continuous. Additionally, we show in Lemma \ref{lem: derivative} below
that it is in fact absolutely continuous on $[0,t_{0}]$ for any $t_{0}<t$.
Consequently, if only we can show that the (a.e.) derivative $\partial_{s}v$
extends to $L^{1}(0,t)$, then we will have absolute continuity on
all of $[0,t]$ with

\begin{equation}
\mathbb{P}^{z}(Z_{t\land\tau}\in S_{t})=\mathbb{P}^{z}(W_{t\land\tau_{{\scriptscriptstyle W}}}\in S_{t})+\int_{0}^{t}\partial_{s}v(s,z)ds.\label{eq: FTC}
\end{equation}
Therefore, the key is simply to establish the right control over $s\mapsto\partial_{s}v(s,z)$.
We embark on this in the next section, but first we prove the previous
claim about the absolute continuity.
\begin{lem}
\label{lem: derivative}Fix an arbitrary $z\in(0,\infty)$. Then the
mapping $s\mapsto v(s,z)$ from (\ref{v_function}) is in $AC[0,t_{0}]$ for every $t_{0}<t$
with (a.e.) derivative
\[
\partial_{s}v(s,z)=\mathbb{E}^{z}[\mathbf{1}_{s<\tau}\hat{b}_{s}\partial_{x}u(s,Z_{s})].
\]
\end{lem}

\begin{proof}
Fix $s\in[0,t)$ and $h>0$ such that $\left|h\right|<t-s$. Recalling
Lemma \ref{lem:=00005BRestatement-of-Lemma4.1 from hambly ledger},
we have $dZ_{t\land\tau}=\mathbf{1}_{t<\tau}\hat{b}_{t}dt+\mathbf{1}_{t<\tau}dB_{t}$,
so an application of Itô's formula yields
\begin{multline*}
v(s+h,z)-v(s,z)=\mathbb{E}^{z}\Bigl[u(s+h,Z_{(s+h)\land\tau})-u(s,Z_{s\land\tau})\Bigr]\\
=\mathbb{E}^{z}\biggl[\int_{s}^{s+h}(\partial_{s}+\frac{1}{2}\triangle)u_{r}dr+\int_{s}^{s+h}\mathbf{1}_{r<\tau}\hat{b}_{r}\partial_{x}u_{r}dr+\int_{s}^{s+h}\mathbf{1}_{r<\tau}\partial_{x}u_{r}dB_{r}\biggr].
\end{multline*}
The first term on the right-hand side vanishes by (\ref{eq:Backward Heat})
and, since $\partial_{x}u(r,x)$ is bounded on the interval $[s,s+h]\subseteq[0,t)$,
the stochastic integral is a true martingale. Noting also that $\mathbb{E}|\sup_{r\leq T}\hat{b}_{r}|<\infty$
(see e.g.~Corollary \ref{cor: Sup is subgaussian}), Fubini's theorem thus
implies that
\[
v(s+h,z)-v(s,z)=\int_{s}^{s+h}\mathbb{E}^{z}[\mathbf{1}_{r<\tau}\hat{b}_{r}\partial_{x}u(r,Z_{r})]dr.
\]
This proves the claim. 
\end{proof}

\subsubsection{Estimates on the half-line}

Given the expression for $\partial_{s}v(s,z)$ in Lemma \ref{lem: derivative},
an application of Hölder's inequality with $q>1$ yields
\begin{align*}
\bigl|\partial_{s}v(s,z)\bigr| & \leq\mathbb{E}^{z}\bigl[\mathbf{1}_{s<\tau}|\hat{b}_{s}|\bigl|\partial_{x}u(s,Z_{s})\bigr|\bigr]=\mathbb{E}^{z}\Bigl[\mathbf{1}_{s<\tau}|\hat{b}_{s}|\int_{S_{t}}|\partial_{x}G_{t-s}(y,Z_{s})|dy\Bigr].\\
 & \leq\mathbb{E}\hspace{-1pt}\left[|\hat{b}_{s}|^{\frac{q}{q-1}}\:\bigr|\,X_{0}^{i}=x_{0}\right]^{\frac{q-1}{q}}\int_{S_{t}}\mathbb{E}^{z}\!\bigl[\mathbf{1}_{s<\tau}|\partial_{x}G_{t-s}(y,Z_{s})|^{q}\bigr]^{\frac{1}{q}}dy\\
 & \leq C(1+x_{0})\int_{S_{t}}\mathbb{E}^{z}\!\bigl[\mathbf{1}_{s<\tau}|\partial_{x}G_{t-s}(y,Z_{s})|^{q}\bigr]^{\frac{1}{q}}dy.
\end{align*}
Here the last inequality follows from the second claim in Lemma \ref{lem: Radon-Nik estimate}
and we emphasize that $q$ can be taken arbitrarily close to $1$.
By introducing the Radon-Nikodym derivative $\mathcal{E}$ from Lemma
\ref{lem:Stochastic exp and equivalnet measu}, Hölder's inequality
with $p>1$ yields
\begin{multline*}
\int_{S_{t}}\mathbb{E}^{z}\!\left[\mathbf{1}_{s<\tau}\bigl|\partial_{x}G_{t-s}(y,Z_{s})\bigr|^{q}\right]^{\frac{1}{q}}\hspace{-1pt}dy=\int_{S_{t}}\mathbb{E}_{\mathbb{Q}}^{z}\hspace{-1pt}\left[\mathcal{E}_{s}^{-\text{1}}\mathbf{1}_{s<\tau}\bigl|\partial_{x}G_{t-s}(y,Z_{s})\bigr|^{q}\right]^{\frac{1}{q}}\hspace{-1pt}dy\\
\leq\:\mathbb{E}\hspace{-1pt}\left[\mathcal{E}_{s}^{1-p}\,\mid\,X_{0}^{i}=x_{0}\right]^{\frac{1}{pq}}\int_{S_{t}}\mathbb{E}_{\mathbb{Q}}^{z}\!\left[\mathbf{1}_{s<\tau}\bigl|\partial_{x}G_{t-s}(y,Z_{s})\bigr|^{a}\right]^{\frac{1}{a}}\hspace{-1pt}dy,
\end{multline*}
where $a=a(p,q):=q\frac{p}{p-1}>1$. For any $\delta>0$, we can take
$p$ close enough to $1$ so that the first estimate from Lemma \ref{lem: Radon-Nik estimate}
applies. Consequently, there exists $a>1$ large enough such that
\begin{align}
\bigl|\partial_{s}v(s,z)\bigr| & \apprle e^{\delta x_{0}^{2}}\int_{S_{t}}\mathbb{E}_{\mathbb{Q}}^{z}\hspace{-1pt}\left[\mathbf{1}_{s<\tau}\bigl|\partial_{x}G_{t-s}(y,Z_{s\wedge\tau})\bigr|^{a}\right]^{\frac{1}{a}}\hspace{-1pt}dy\nonumber \\
 & =e^{\delta x_{0}^{2}}\int_{S_{t}}\left(\int_{0}^{\infty}\bigl|\partial_{x}G_{t-s}(y,x)\bigr|^{a}G_{s}(x,z)dx\right)^{\hspace{-1pt}\frac{1}{a}}\hspace{-1pt}dy,\label{eq: Derivative of v(s,z)}
\end{align}
where the last line follows from the fact that $Z_{s\wedge\tau}$
is an absorbing Brownian motion under $\mathbb{Q}$, as shown in Lemma \ref{lem:Stochastic exp and equivalnet measu}. Before
proceeding, we collect some useful bounds for exponential functions.
\begin{lem}
\label{lem: exp lemma}Fix any $x,y\in\mathbb{R}$ and $t>s$. Then
it holds for all powers $a\geq1$ that
\begin{align}
e^{-\frac{a(y-x)^{2}}{2(t-s)}}e^{-\frac{(x-z)^{2}}{2s}} &\leq e^{-\frac{(y-z)^{2}}{2t}}e^{-\frac{t}{2s}\left(\frac{(x-y)}{\sqrt{t-s}}+\frac{\sqrt{t-s}(y-z)}{t-s+sa}\right)^{2}},\label{eq:y-x formula} \\
e^{-\frac{a(y+x)^{2}}{2(t-s)}}e^{-\frac{(x-z)^{2}}{2s}} &\leq e^{-\frac{(y+z)^{2}}{2t}}e^{-\frac{t}{2s}\left(\frac{(y+x)}{\sqrt{t-s}}+\frac{\sqrt{t-s}(y+z)}{t-s+sa}\right)^{2}}.\label{eq: x+y formula}
\end{align}

\end{lem}

Recalling (\ref{eq: FTC}), the desired density estimates will follow
if we can obtain suitable bounds on the right-hand side of (\ref{eq: Derivative of v(s,z)}).
Our first result is the following.
\begin{prop}
\label{prop:Half-space density}For any $\delta>0$ there exists $a>1$
such that, for every $S\in\mathcal{B}((0,\infty))$, it holds uniformly
in $N\geq1$ that
\begin{multline*}
\quad\mathbb{P}^{x_{0}}\bigl(X_{t\land\tau_{i}}^{i,{\scriptscriptstyle N}}\in S\bigr)\leq\int_{S}G_{t}\bigl(\Upsilon_{t}(x),\Upsilon_{0}(x_{0})\bigr)\partial_{x}\Upsilon_{t}(x)dx\\
+C_{a}e^{\delta x_{0}^{2}}\int_{S}\bigl(t^{-\frac{1}{a}}\Upsilon_{0}(x_{0})^{\frac{1}{a}}\Upsilon_{t}(x)^{\frac{1}{a}}\land1\bigr)e^{-\frac{(\Upsilon_{t}(x)-\Upsilon_{0}(x_{0}))^{2}}{4at}}\partial_{x}\Upsilon_{t}(x)dx.\quad
\end{multline*}
\end{prop}

\begin{proof}
Given $\delta>0$, we can choose $a>1$ such that (\ref{eq: Derivative of v(s,z)})
holds. Writing out the expressions for $\partial_{x}G_{t-s}$ and
$G_{s}$, and noting that 
\begin{equation}
e^{-(x-z)^{2}/2s}-e^{-(x+z)^{2}/2s}\leq({\textstyle \frac{2xz}{s}}\land1)e^{-(x-z)^{2}/2s},\label{eq: simple exp bound (absorbing)}
\end{equation}
 we get $|\partial_{s}v(s,z)|\apprle e^{\delta x_{0}^{2}}I(s)$,
where
\begin{multline*}
I(s):=s^{-\frac{1}{2a}}(t-s)^{-\frac{1}{2}}\int_{S_{t}}\left(\int_{0}^{\infty}\Bigl|{\textstyle \frac{(y-x)}{t-s}}e^{-\frac{(y-x)^{2}}{2(t-s)}}+{\textstyle \frac{(y+x)}{t-s}}e^{-\frac{(y+x)^{2}}{2(t-s)}}\Bigr|^{a}{\textstyle \frac{xz}{s}}e^{-\frac{(x-z)^{2}}{2s}}dx\!\right)^{\!\frac{1}{a}}\!dy.
\end{multline*}
Recalling (\ref{eq: FTC}), we thus have
\begin{equation}
\mathbb{P}^{x}\bigl(X_{t\land\tau_{i}}^{i,{\scriptscriptstyle N}}\in S\bigr)\leq\mathbb{P}^{z}(W_{t}\in S_{t},\:t<\tau_{{\scriptscriptstyle W}})+Ce^{\delta x^{2}}\!\int_{0}^{t}I(s)ds\label{eq:Representation}
\end{equation}
for some $C>0$, so the claim amounts to controlling the integrand
$I(s)$. We split the work involved in this endeavour into six steps.

\textbf{Step 1.} To estimate $I(s)$, we begin by observing that
\begin{align*}
\Bigl|{\textstyle \frac{(y-x)}{t-s}}e^{-\frac{(y-x)^{2}}{2(t-s)}}+{\textstyle \frac{(y+x)}{t-s}}e^{-\frac{(y+x)^{2}}{2(t-s)}}\Bigr| & \leq{\textstyle \frac{\left|y-x\right|}{t-s}}\Bigl(e^{-\frac{(y-x)^{2}}{2(t-s)}}-e^{-\frac{(y+x)^{2}}{2(t-s)}}\Bigr)+{\textstyle \frac{2y}{t-s}}e^{-\frac{(y+x)^{2}}{2(t-s)}}.
\end{align*}
For the first term on the right-hand side, we can use (\ref{eq: simple exp bound (absorbing)})
to see that
\begin{align}
{\textstyle \frac{\left|y-x\right|}{t-s}}\Bigl(e^{-\frac{(y-x)^{2}}{2(t-s)}}-e^{-\frac{(y+x)^{2}}{2(t-s)}}\Bigr) & \leq{\textstyle \frac{\left|y-x\right|}{t-s}}\bigl({\textstyle \frac{2xy}{t-s}}\land1\bigr)e^{-\frac{(y-x)^{2}}{2(t-s)}}=:f_{1}(s).\label{eq:f_1}
\end{align}
For the second term, it will prove useful to observe that, for $x,y>0$,
\begin{align}
{\textstyle \frac{y}{t-s}}e^{-\frac{(y+x)^{2}}{2(t-s)}} \leq \bigl(\textstyle{\frac{x+y}{x}}\bigr)^{\frac{1}{a}} {\textstyle \frac{y}{t-s}}e^{-\frac{(y+x)^{2}}{2(t-s)}}\leq y^{\frac{1}{a}}x^{-\frac{1}{a}}{\textstyle \frac{(x+y)}{t-s}}e^{-\frac{(y+x)^{2}}{2(t-s)}}=:f_{2}(s)\label{f_2},
\end{align}
where we have used that $y^{1-\frac{1}{a}}\leq(x+y)^{1-\frac{1}{a}}$ since $a>1$. 

Based on (\ref{eq:f_1}) and (\ref{f_2}), we have $I(s)\leq I_{1}(s)+2I_{2}(s)$,
where
\[
I_{k}(s):=s^{-\frac{3}{2a}}(t-s)^{-\frac{1}{2}}z^{\frac{1}{a}}\int_{S_{t}}\hspace{-1pt}\biggl(\int_{0}^{\infty}f_{i}(s)^{a}{\textstyle x}e^{-\frac{(x-z)^{2}}{2s}}dx\!\biggr)^{\!\frac{1}{a}}\!dy,\quad k=1,2.
\]

\textbf{Step 2.} We begin with the second term, $I_{2}$. Using (\ref{eq: x+y formula})
of Lemma \ref{lem: exp lemma}, we have
\begin{align*}
I_{2}(s)= & \:s^{-\frac{3}{2a}}{\textstyle \frac{1}{t-s}}\int_{S_{t}}z^{\frac{1}{a}}y^{\frac{1}{a}}\left(\int_{0}^{\infty}\bigl({\textstyle \frac{x+y}{\sqrt{t-s}}}\bigr)^{a}e^{-\frac{a(y+x)^{2}}{2(t-s)}}e^{-\frac{(x-z)^{2}}{2s}}dx\!\right)^{\!\frac{1}{a}}\!dy\\
\leq\: & s^{-\frac{3}{2a}}{\textstyle \frac{1}{t-s}}\int_{S_{t}}z^{\frac{1}{a}}y^{\frac{1}{a}}e^{\frac{-(y+z)^{2}}{2t}}\left(\int_{0}^{\infty}\bigl({\textstyle \frac{x+y}{\sqrt{t-s}}}\bigr)^{a}e^{-\frac{t}{2s}\left(\frac{(y+x)}{\sqrt{t-s}}+\frac{\sqrt{t-s}(y+z)}{t-s+sa}\right)^{2}}dx\!\right)^{\!\frac{1}{a}}\!dy
\end{align*}
To evaluate the inner integral, we perform a change of variables with
\[
w=t^{\frac{1}{2}}s^{-\frac{1}{2}}{\textstyle \Bigl(\frac{(x+y)}{\sqrt{t-s}}+\frac{\sqrt{t-s}(y+z)}{t-s+\sigma s}\Bigr)},\quad dx=(t-s)^{\frac{1}{2}}t^{-\frac{1}{2}}s^{\frac{1}{2}}dw,
\]
where we note that
\[
0\leq{\textstyle \frac{(x+y)}{\sqrt{t-s}}}=t^{-\frac{1}{2}}s^{\frac{1}{2}}w-{\textstyle \frac{\sqrt{t-s}(y+z)}{t-s+\sigma s}}\leq w
\]
and hence
\[
I_{2}(s)\leq s^{-\frac{1}{a}}(t-s)^{\frac{1}{2a}-1}t^{-\frac{1}{2a}}\left(\int_{0}^{\infty}w^{a}e^{-\frac{w^{2}}{2}}dw\!\right)^{\!\frac{1}{a}}\!\int_{S_{t}}z^{\frac{1}{a}}y^{\frac{1}{a}}e^{\frac{-(y+z)^{2}}{2at}}dy.
\]
Since $a>1$, we have $\int_{0}^{t}s^{-\frac{1}{a}}(t-s)^{\frac{1}{2a}-1}ds=c_{a}t^{-\frac{1}{2a}}$
for some $c_{a}>0$, so $I_{2}$ is in $L^{1}(0,t)$ with
\begin{equation}
\int_{0}^{t}I_{2}(s)ds\leq C_{a}t^{-\frac{1}{a}}\int_{S_{t}}z^{\frac{1}{a}}y^{\frac{1}{a}}e^{\frac{-(y+z)^{2}}{2at}}dy,\label{eq: I2}
\end{equation}
where $C_{a}>0$ is some numerical constant depending only on $a$.

\textbf{Step 3.} In order to estimate the first term,\textbf{ $I_{1}$},\textbf{
}we rely on the inequality
\begin{equation}
x\bigl({\textstyle \frac{2xy}{t-s}}\land1\bigr)^{a}\leq4y{\textstyle \frac{\left|x-y\right|^{2}}{t-s}}+2y\quad\text{for any}\;a>1.\label{eq: inequality}
\end{equation}
To see that this inequality is true, simply note that, when $x\leq y$,
we have
\[
x\bigl({\textstyle \frac{2xy}{t-s}}\land1\bigr)^{a}\leq x\leq y
\]
while, for $y\leq x$, we can write $x=(x-y)+y$ to obtain
\[
x\bigl({\textstyle \frac{2xy}{t-s}}\land1\bigr)^{a}\leq2(\left|x-y\right|^{2}+y^{2})\bigl({\textstyle \frac{2y}{t-s}}\land{\textstyle \frac{1}{y}}\bigr)\leq2\left|x-y\right|^{2}{\textstyle \frac{2y}{t-s}}+2y.
\]
Now, using the inequality (\ref{eq: inequality}) as well as (\ref{eq:y-x formula})
from Lemma \ref{lem: exp lemma}, it follows that
\begin{align*}
I_{1}(s):= & \:s^{-\frac{1}{2a}}(t-s)^{-\frac{1}{2}}\int_{S_{t}}\hspace{-1pt}\left(\int_{0}^{\infty}\Bigl|\bigl({\textstyle \frac{2xy}{t-s}}\land1\bigr){\textstyle \frac{\left|y-x\right|}{t-s}}e^{-\frac{(y-x)^{2}}{2(t-s)}}\Bigr|^{a}\textstyle{\frac{xz}{s}}e^{-\frac{(x-z)^{2}}{2s}}dx\!\right)^{\!\frac{1}{a}}\!dy\\
= & \:{\textstyle \frac{1}{t-s}}s^{-\frac{3}{2a}}\int_{S_{t}}\hspace{-1pt}\left(\int_{0}^{\infty}xz\bigl({\textstyle \frac{2xy}{t-s}}\land1\bigr)^{a}{\textstyle |\frac{y-x}{\sqrt{t-s}}|}^{a}e^{-\frac{a(y-x)^{2}}{2(t-s)}}e^{-\frac{(x-z)^{2}}{2s}}dx\!\right)^{\!\frac{1}{a}}\!dy\\
\apprle & \:{\textstyle \frac{1}{t-s}}s^{-\frac{3}{2a}}\int_{S_{t}}z^{\frac{1}{a}}y^{\frac{1}{a}}e^{-\frac{(y-z)^{2}}{2at}}\left(\int_{0}^{\infty}F(s,x,y,z)dx\!\right)^{\!\frac{1}{a}}\!dy,
\end{align*}
where
\[
F(s,x,y,z):=\left({\textstyle |\frac{y-x}{\sqrt{t-s}}|}^{a}+{\textstyle |\frac{y-x}{\sqrt{t-s}}|}^{a+2}\right)e^{-\frac{t}{2s}\left(\frac{(y-x)}{\sqrt{t-s}}+\frac{\sqrt{t-s}(y-z)}{t-s+sa}\right)^{2}}.
\]
For the inner integral, we make the change of variables
\[
w=t^{\frac{1}{2}}s^{-\frac{1}{2}}\left({\textstyle \frac{(x-y)}{\sqrt{t-s}}}+{\textstyle \frac{\sqrt{t-s}(y-z)}{t-s+as}}\right),\quad dx=(t-s)^{\frac{1}{2}}t^{-\frac{1}{2}}s^{\frac{1}{2}}dw,
\]
and set
\begin{align*}
f(w) & =f(s,w,y,z):=\bigl|t^{-\frac{1}{2}}s^{\frac{1}{2}}w-{\textstyle \frac{\sqrt{t-s}(y-z)}{t-s+as}}\bigr|.
\end{align*}
This yields
\begin{align*}
I_{1}(s)\apprle & \:s^{-\frac{1}{a}}(t-s)^{\frac{1}{2a}-1}t^{-\frac{1}{2a}}\int_{S_{t}}z^{\frac{1}{a}}y^{\frac{1}{a}}e^{-\frac{(y-z)^{2}}{2at}}\left(\int_{\mathbb{R}}\bigl(f(w)^{a}+f(w)^{a+2}\bigr)e^{-\frac{w^{2}}{2}}dw\!\right)^{\!\frac{1}{a}}\!dy.
\end{align*}
Noting that $f(w)\leq\left|w\right|+t^{-1}(t-s)^{\frac{1}{2}}\left|y-z\right|,$
we can split up $I_{1}$ accordingly. Since
\begin{gather*}
t^{-\frac{1}{2a}}\int_{0}^{t}s^{-\frac{1}{a}}(t-s)^{\frac{1}{2a}-1}ds=c_{a}t^{-\frac{1}{a}},\quad t^{-1-\frac{1}{2a}}\int_{0}^{t}s^{-\frac{1}{a}}(t-s)^{\frac{1}{2a}-\frac{1}{2}}ds=c_{a}^{{\scriptscriptstyle \prime}}t^{-\frac{1}{2}-\frac{1}{a}},\\
\text{and}\quad t^{-1-\frac{5}{2a}}\int_{0}^{t}s^{-\frac{1}{a}}(t-s)^{\frac{3}{2a}-\frac{1}{2}}ds=c_{a}^{{\scriptscriptstyle \prime\prime}}t^{-\frac{1}{2}-\frac{2}{a}},
\end{gather*}
we thus obtain that
\begin{align}
\int_{0}^{t}I_{1}(s)ds\apprle & \:C_{a}t^{-\frac{1}{a}}\int_{S_{t}}z^{\frac{1}{a}}y^{\frac{1}{a}}e^{-\frac{(y-z)^{2}}{2at}}dy+C_{a}^{\prime}t^{-\frac{1}{2}-\frac{1}{a}}\int_{S_{t}}z^{\frac{1}{a}}y^{\frac{1}{a}}e^{-\frac{(y-z)^{2}}{2at}}\left|y-z\right|dy\nonumber \\
 & \qquad\qquad+C_{a}^{\prime\prime}t^{-\frac{1}{2}-\frac{2}{a}}\int_{S_{t}}z^{\frac{1}{a}}y^{\frac{1}{a}}e^{-\frac{(y-z)^{2}}{2at}}\left|y-z\right|^{1+\frac{2}{a}}dy,\label{eq: I1}
\end{align}
where
\begin{gather*}
C_{a}=c_{a}\left(\int_{\mathbb{R}}\bigl(\left|w\right|^{a}+\left|w\right|^{a+2}\bigr)e^{-\frac{w^{2}}{2}}dw\!\right)^{\!\frac{1}{a}},\quad C_{a}^{\prime}=c_{a}^{{\scriptscriptstyle \prime}}\left(\int_{\mathbb{R}}e^{-\frac{w^{2}}{2}}dw\!\right)^{\!\frac{1}{a}}=c_{a}^{{\scriptscriptstyle \prime}}2^{\frac{1}{2a}}\pi^{\frac{1}{2a}},\\
\text{and}\quad C_{a}^{\prime\prime}=c_{a}^{{\scriptscriptstyle \prime\prime}}\left(\int_{\mathbb{R}}e^{-\frac{w^{2}}{2}}dw\!\right)^{\!\frac{1}{a}}=c_{a}^{{\scriptscriptstyle \prime\prime}}2^{\frac{1}{2a}}\pi^{\frac{1}{2a}}.
\end{gather*}

\textbf{Step 4.} Given the above, we can now combine the estimates
for $I_{1}$ and $I_{2}$. However, we first recall the elementary
inequalities
\begin{gather*}
\left|y-z\right|e^{-\frac{(y-z)^{2}}{2at}}\leq C_{a}t^{\frac{1}{2}}e^{-\frac{(y-z)^{2}}{4at}},\quad(y+z)e^{\frac{-(y+z)^{2}}{2\sigma t}}\leq C_{a}t^{\frac{1}{2}}e^{-\frac{(y+z)^{2}}{4at}},\\
\text{and}\quad t^{-\frac{1}{a}}\left|y-z\right|^{1+\frac{2}{\sigma}}e^{-\frac{(y-z)^{2}}{2at}}\leq C_{a}t^{\frac{1}{2}}e^{-\frac{(y-z)^{2}}{4at}},
\end{gather*}
Using these, and recalling also (\ref{eq:Representation}), it follows
from (\ref{eq: I2}) and (\ref{eq: I1}) that there exists a constant
$C_{a}>0$ such that
\begin{align}
\mathbb{P}^{x_{0}}\bigl(X_{t\land\tau_{i}}^{i,{\scriptscriptstyle N}}\in S\bigr) & \leq\int_{S_{t}}G_{t}\bigl(y,z\bigr)dy+C_{a}e^{\delta x_{0}^{2}}\int_{S_{t}}t^{-\frac{1}{a}}z{}^{\frac{1}{a}}y{}^{\frac{1}{a}}e^{-\frac{(y-z)^{2}}{4at}}dy.\label{eq:bound_1}
\end{align}

\textbf{Step 5.} It remains to observe that (\ref{eq:bound_1}) also
holds with $1$ in place of $t^{-\frac{1}{a}}z{}^{\frac{1}{a}}y{}^{\frac{1}{a}}$.
To see this, note that in Step 1 we also have $|\partial_{s}v(s,z)|\apprle e^{\delta x_{0}^{2}}(J_{1}(s)+J_{2}(s))$,
where
\[
\begin{cases}
J_{1}(s):=s^{-\frac{1}{2a}}(t-s)^{-1}\int_{S_{t}}\hspace{-1pt}\Bigl(\int_{0}^{\infty}{\textstyle \bigl|\frac{y-x}{\sqrt{t-s}}\bigr|^{a}}e^{-\frac{a(y-x)^{2}}{2(t-s)}}e^{-\frac{(x-z)^{2}}{2s}}dx\!\Bigr)^{\!\frac{1}{a}}\!dy\\[5pt]
J_{2}(s):=s^{-\frac{1}{2a}}(t-s)^{-1}\int_{S_{t}}\hspace{-1pt}\Bigl(\int_{0}^{\infty}{\textstyle \bigl|\frac{y+x}{\sqrt{t-s}}\bigr|^{a}}e^{-\frac{a(y+x)^{2}}{2(t-s)}}e^{-\frac{(x-z)^{2}}{2s}}dx\!\Bigr)^{\!\frac{1}{a}}\!dy.
\end{cases}
\]
Using Lemma \ref{lem: exp lemma} and performing the same changes
of variables as in Steps 2 and 3, the computations simplify significantly
and we obtain the desired bound.

\textbf{Step 6.} Recall that $S_{t}=\Upsilon_{t}(S)$ and $z=\Upsilon_{0}(x_{0})$.
In view of Step 4 and Step 5, the proof is therefore finished by performing
the change of variables $y=\Upsilon_{t}(x)$.
\end{proof}

\subsubsection{Estimates on the whole space}

If we ignore the absorption at the boundary, then the estimates are
much simpler, and we get the following bound.
\begin{prop}
\label{prop:Whole space density}For every $\delta>0$ there exists
$a>1$ such that, for every $S\in\mathcal{B}(\mathbb{R})$, it holds
uniformly in $N\geq1$ that
\[
\mathbb{P}^{x_{0}}\hspace{-1pt}\bigl(X_{t}^{i,{\scriptscriptstyle N}}\!\in S\bigr)\leq\int_{S}p_{t}\bigl(\bar{\Upsilon}_{\!t}(x,x_{0})\bigr)\partial_{x}\bar{\Upsilon}_{t}(x,x_{0})dx+C_{a}\!\int_{S}e^{-\frac{(\bar{\Upsilon}_{t}(x,x_{0}))^{2}}{4at}}\partial_{x}\bar{\Upsilon}_{\!t}(x,x_{0})dx,
\]
where $p_{t}(x)=(2\pi t)^{-\frac{1}{2}}\exp\{x^{2}/2t\}$ and $\bar{\Upsilon}_{t}(x,x_{0})=\int_{x_{0}}^{x}\sigma(t,y)^{-1}dy$.
\end{prop}

\begin{proof}
Consider Lemma \ref{lem:=00005BRestatement-of-Lemma4.1 from hambly ledger}
with $\bar{Z}_{t}:=\bar{\Upsilon}_{t}(X_{t}^{i},x_{0})$ in place
of $Z_{t}=\Upsilon_{t}(X_{t}^{i})=\bar{\Upsilon}_{t}(X_{t}^{i},0)$
and let $\bar{\mathbb{Q}}$ denote the corresponding measure from
Lemma \ref{lem:Stochastic exp and equivalnet measu} such that $\bar{Z}_{t}$
is a standard Brownian motion under $\bar{\mathbb{Q}}$, started at
$0$ when $X_{0}^{i}=x_{0}$. Setting $\bar{S}_{t}:=\bar{\Upsilon}_{t}(S,x_{0})$,
we have
\[
\mathbb{P}^{x_{0}}\hspace{-1pt}\bigl(X_{t}^{i,{\scriptscriptstyle N}}\!\in S\bigr)=\mathbb{P}^{0}\hspace{-1pt}\bigl(\bar{Z}_{t}\in\bar{S}_{t}\bigr).
\]
From here, the bound follows by analogy with Proposition \ref{prop:Half-space density}
for the new measure $\bar{\mathbb{Q}}$ and
\[
\bar{v}(s):=\mathbb{E}^{0}\hspace{-1pt}\left[\bar{u}(s,\bar{Z}_{s})\right],\quad\bar{u}(s,x):=\mathbb{P}^{x}\bigl(W_{t-s}\in\bar{S}_{t}\bigr)=\int_{\bar{S}_{t}}p_{t-s}(y-x)dy.
\]
Since there are no boundary effects, the estimates simplify and, in fact, the
work is the same as for $J_{1}$ in Step 5 of the proof of Proposition
\ref{prop:Half-space density} with $z=0$.
\end{proof}

\subsection{Proof of Proposition \ref{prop: Aronson estimate}\label{subsec:Proof-of-Prop3.2}}

In view of Propositions \ref{prop:Half-space density} and \ref{prop:Whole space density}, only a few observations remain before we can deduce the density estimates (\ref{eq:finite half-line density estimate}) and (\ref{eq:finite whole-space density estimate}), thus proving
Proposition \ref{prop: Aronson estimate}. Given
Assumption \ref{Assumption 1 - finite system}, it holds by construction of $\bar{\Upsilon}$ that there are constants $C_{1},C_{2}>0$
such that $C_{1}|x-x_{0}|\leq|\bar{\Upsilon}_{t}(x,x_{0})|\leq C_{2}|x-x_{0}|$
and $|\partial_{x}\Upsilon_{t}(x)|\leq C_{2}$ for all $x\in\mathbb{R}$
and $t\in[0,T]$. Consequently, the whole-space estimate (\ref{eq:finite whole-space density estimate})
is an immediate consequence of Proposition \ref{prop:Whole space density}.

For the estimate with boundary decay, we begin by recalling the standard
bound
\[
G_{t}(x,x_{0})\leq C\frac{1}{\sqrt{t}}\bigl(\frac{x}{\sqrt{t}}\land1\bigr)\bigl(\frac{x_{0}}{\sqrt{t}}\land1\bigr)\exp\Bigl\{\frac{-(x-x_{0})^{2}}{4t}\Bigr\}\quad\text{for}\quad x,x_{0}\geq0.
\]
Next, we notice that, by definition of $\Upsilon$, there exist $C>0$
such that $|\Upsilon_{t}(x)|\leq Cx$ and $|\partial_{x}\Upsilon_{t}(x)|\leq C$
for all $x\geq0$ and $t\in[0,T]$. Finally, we can observe that
\begin{align*}
\bigl(\Upsilon_{t}(x)-\Upsilon_{0}(x_{0})\bigr)^{2} & \geq\left(\int_{x\land x_{0}}^{x\lor x_{0}}\!\frac{dy}{\sigma(t,y)}\right)^{\!2}+2\int_{x\land x_{0}}^{x\lor x_{0}}\!\frac{dy}{\sigma(t,y)}\int_{0}^{x\land x_{0}}\!\frac{1}{\sigma(t,y)}-\frac{1}{\sigma(0,y)}dy\\
 & \geq C(x-x_{0})^{2}-C^{\prime}t|x-x_{0}|(x\land x_{0}),
\end{align*}
by using the bounds on $\sigma$ and $\partial_{t}\sigma$ from Assumption
\ref{Assumption 1 - finite system}. Combining these observations,
the density estimate (\ref{eq:finite half-line density estimate})
follows from Proposition \ref{prop:Half-space density} by taking
$c_{x,y}:=4aC^{\prime}|x-y|(x\wedge y)$ and $\kappa:=1/a$.

It remains to observe that we can take $c_{x,y}\equiv0$ when $\sigma(t,x)=\sigma_{1}(t)\sigma_{2}(x)$
and that, in this case, we do not need any smoothness of $t\mapsto\sigma(t,x)$,
as mentioned in Remark \ref{rem:weaker_assumptions}.

To see this, the point is that it suffices to scale away the spatial
component of the volatility. Specifically, we can consider the analogue
of Lemma \ref{lem: Lemma8.8 from hambly-ledger} with
\[
\tilde{\Upsilon}(x):=\int_{0}^{x}\hspace{-1pt}\frac{dy}{\sigma_{2}(y)}\quad\text{and}\quad\tilde{Z}_{t}^{i}:=\tilde{\Upsilon}(X_{t}^{i})=\tilde{\Upsilon}(X_{0}^{i})+\int_{0}^{t}\tilde{b}_{s}^{i}ds+\int_{0}^{t}\hspace{-1pt}\sigma_{1}(s)dB_{s}.
\]
Fix $S\in\mathcal{B}(0,\infty)$ and set $\tilde{S}:=\tilde{\Upsilon}(S)$.
Fixing also $t\in(0,T]$, we can replicate Section 6.3 with $\tilde{v}(s,z):=\mathbb{E}^{z}[\tilde{u}(s,\tilde{Z}_{s\land\tilde{\tau}})]$
and
\[
\tilde{u}(s,x):=\mathbb{P}^{x}\hspace{-1pt}\Bigl(W_{\int_{s}^{t}\sigma_{1}(r)^{2}dr}\in\tilde{S},\;{\textstyle \int_{s}^{t}}\sigma_{1}(r)^{2}dr<\tau_{{\scriptscriptstyle W}}\Bigr)=\int_{\tilde{S}}G_{s,t}^{\sigma_{1}}(y,x)dy,
\]
where $G_{s,t}^{\sigma_{1}}$ is the Green's function for $\partial_{s}f(s,x)+\frac{1}{2}\sigma_{1}(s)^{2}\Delta f(s,x)=0$
as a terminal-boundary value problem on $[0,t)\times(0,\infty)$ with $f(s,0)=0$.
That is,
\[
G_{s,t}^{\sigma_{1}}(y,x)={\textstyle \bigl(2\pi\int_{s}^{t}\sigma_{1}(r)^{2}dr\bigr)^{-\frac{1}{2}}}\biggl(\exp\biggl\{\frac{-(x-y)^{2}}{2\int_{s}^{t}\sigma_{1}(r)^{2}dr}\biggr\}-\exp\biggl\{\frac{-(x+y)^{2}}{2\int_{s}^{t}\sigma_{1}(r)^{2}dr}\biggr\}\biggr).
\]
Since there are constants $c_{1},c_{2}>0$ such that $c_{1}(t-s)\leq\int_{s}^{t}\sigma_{1}(r)^{2}dr\leq c_{2}(t-s)$,
the same estimates as in the proof of Proposition \ref{prop:Half-space density}
yield
\[
\mathbb{P}^{x_{0}}\hspace{-1pt}\bigl(X_{t\land\tau_{i}}^{i}\in S\bigr)=\int_{\tilde{S}}G_{s,t}^{\sigma_{1}}(y,\tilde{\Upsilon}(x_{0}))dy+\int_{0}^{t}\partial_{s}\tilde{v}(s,\tilde{\Upsilon}(x_{0}))ds
\]
and, using also the bounds $|\tilde{\Upsilon}(x)|\leq cx$ and $|\partial_{x}\tilde{\Upsilon}|\leq c$,
we get
\[
\mathbb{P}^{x_{0}}\hspace{-1pt}\bigl(X_{t\land\tau_{i}}^{i}\in S\bigr)\apprle\int_{S}G_{0,t}^{\sigma_{1}}\bigl(\tilde{\Upsilon}(x),\tilde{\Upsilon}(x_{0})\bigr)dx+e^{\delta x_{0}^{2}}\int_{S}\bigl(t^{-\frac{1}{a}}x^{\frac{1}{a}}x_{0}^{\frac{1}{a}}\land1\bigr)e^{-\frac{(\tilde{\Upsilon}(x)-\tilde{\Upsilon}(x_{0}))^{2}}{4at}}dx.
\]
As $\tilde{\Upsilon}$ does not depend on time, we have $|\tilde{\Upsilon}(x)-\tilde{\Upsilon}(x_{0})|=|\int_{x_{0}}^{x}\sigma_{2}(y)^{-1}dy|$,
so there exist $c_{1}^{\prime},c_{2}^{\prime}>0$ such that $c_{1}^{\prime}|x-x_{0}|\leq|\tilde{\Upsilon}(x)-\tilde{\Upsilon}(x_{0})|\leq c_{2}^{\prime}|x-x_{0}|$.
Also, $G_{0,t}^{\sigma_{1}}$ satisfies an analogous bound to that
of $G_{t}(x,y)$ above, so we conclude that the density estimate (\ref{eq:finite half-line density estimate})
holds with $c_{x,y}\equiv0$, as desired. This finishes the proof
of Proposition \ref{prop: Aronson estimate}.

\appendix

\section{Appendix\label{sec:Appendix}}

\subsection{Technical lemmas}
\begin{lem}
\label{lem:Error terms half-space}Suppose $\nu$ satisfies Assumption
\ref{Assumption 2.3} and let $g_{s}(x)=g(s,x,\nu_{s},L_{s})$, where
$g$ is any of $b$, $\mathfrak{b}$, $\sigma$, or $\sigma^{2}$, with  $L_{s}=1-\nu_{s}(0,\infty)$. Define the error term
\[
\mathcal{E}_{t,\varepsilon}^{g}(x):=\left\langle \nu_{t},g_{t}(\cdot)G_{\varepsilon}(x,\cdot)\right\rangle -g_{t}(x)(\mathcal{T}_{\varepsilon}\nu_{t})(x).
\]
Then we have
\[
\mathbb{E}\int_{0}^{T}\left\Vert \mathcal{E}_{t,\varepsilon}^{g}\right\Vert _{L^{2}(\mathbb{R})}^{2}dt\rightarrow0\quad\text{as}\quad\varepsilon\rightarrow0.
\]
\end{lem}

\begin{proof}
This follows by a straightforward modification of Lemma 8.1 in \cite{HL2016}.
The only thing to note is that we can no longer use the crude
bound $\left|g_{t}(x)-g_{t}(y)\right|\leq2\left\Vert g_{t}\right\Vert _{\infty}$,
as $g_{t}$ need not be bounded. However, the arguments from \cite{HL2016}
are easily extended to the present case if we instead rely on $\left|g_{t}(x)-g_{t}(y)\right|\leq\left\Vert \partial_{x}g_{t}\right\Vert _{\infty}\left|x-y\right|$.
\end{proof}
\begin{lem}
\label{lem:Error Terms Whole Space}Suppose $\bar{\nu}$ satisfies
the whole-space analogues of (iii)-(iv) in Assumption~\ref{Assumption 2.3}.
Let $g_{s}=g(s,\cdot,\nu_{s},L_{s})$, where $g$ is any
of $b$, $\mathfrak{b}$, $\sigma$, or $\sigma^{2}$, and define the error terms
\begin{align*}
\mathcal{\bar{E}}_{t,\varepsilon}^{g}(x) & :=\left\langle \bar{\nu}_{t},g_{t}(\cdot)p_{\varepsilon}(x,\cdot)\right\rangle -g_{t}(x)\partial_{x}(\bar{\mathcal{T}}_{\varepsilon}\bar{\nu}_{t})(x)+\partial_{x}g_{t}(x)\mathcal{\bar{H}}_{t,\varepsilon}^{g}(x),\\
\mathcal{\bar{H}}_{t,\varepsilon}^{g}(x) & :=\left\langle \bar{\nu}_{t},(x-\cdot)\partial_{x}p_{\varepsilon}(x-\cdot)\right\rangle .
\end{align*}
Then, for $k=1,2$, we have
\[
\mathbb{E}\int_{0}^{T}\bigl\Vert\mathcal{\bar{E}}_{t,\varepsilon}^{g}\bigr\Vert_{L^{2}(\mathbb{R})}^{2k}dt \rightarrow0,\quad \mathbb{E}\int_{0}^{T}\bigl\Vert x\mathcal{\bar{E}}_{t,\varepsilon}^{g}\bigr\Vert_{L^{2}(\mathbb{R})}^{2}dt\rightarrow0,\quad\text{as}\quad\varepsilon\rightarrow0.
\]
\end{lem}

\begin{proof}
This follows by a simple modification of Lemma 8.2 in \cite{HL2016}.
\end{proof}
\begin{lem}
\label{lem:Tails for centre of mass}
Let $\nu_{t}^{{\scriptscriptstyle N}}$
be as defined in (\ref{eq: Particle System}) and let $\tilde{\nu}_t$
be any measure valued process satisfying Assumption \ref{Assumption 2.3}.
Then, for every $a>0$, we have
\begin{equation}
\mathbb{E}\int_{0}^{T}\bigl\langle\tilde{\nu}_{t},x^{k}\mathbf{1}_{[\lambda,\infty)}(x)\bigr\rangle dt=o(\lambda^{k}e^{-a\lambda})\quad\text{as}\quad\lambda\rightarrow\infty,\label{eq:LemmaA3_limit}
\end{equation}
and, likewise, it holds uniformly in $N\geq1$ and $t\in[0,T]$ that
\begin{equation}
\mathbb{E}\left\langle \nu_{t}^{{\scriptscriptstyle N}},x^{k}\mathbf{1}_{[\lambda,\infty)}(x)\right\rangle =o(\lambda^{k}e^{-a\lambda})\quad\text{as}\quad\lambda\rightarrow\infty\label{eq:LemmaA3_emp}
\end{equation}
\end{lem}

\begin{proof}
Fix an arbitrary $a>0$ and let $\varepsilon>0$ be given. By Corollary
\ref{Cor:Regularity of the empirical measures}, there exists $\lambda_{0}\geq0$
such that
\begin{equation}
\mathbb{E}\left[\nu_{t}^{{\scriptscriptstyle N}}(\lambda,\infty)\right]\leq\varepsilon e^{-a\lambda}\quad\forall\lambda\geq\lambda_{0},\label{eq:lemmaA3}
\end{equation}
uniformly in $N\geq1$ and $t\in[0,T\text{]}$. Given any $\lambda\geq\lambda_{0}$
we let $\left\{ s_{0},s_{1},\ldots\right\} $ denote the partition
of $[\lambda,\infty)$ with $s_{i}-s_{i-1}=1/a$. By (\ref{eq:lemmaA3})
and monotone convergence, we get

\begin{align*}
\mathbb{E}\left\langle \nu_{t}^{{\scriptscriptstyle N}},x\mathbf{1}_{[\lambda,\infty)}(x)\right\rangle  & \leq\sum_{i=1}^{\infty}s_{i}\mathbb{E}\nu_{t}^{{\scriptscriptstyle N}}(s_{i-1},\infty)\leq\varepsilon\sum_{i=1}^{\infty}s_{i}e^{-as_{i-1}}=\varepsilon e\sum_{i=1}^{\infty}s_{i}e^{-as_{i}}
\end{align*}
Now, $x\mapsto xe^{-ax}$ is decreasing for $x\geq1/a$, so taking
$\lambda_{0}^{\prime}:=\max\left\{ 1/a,\lambda_{0}\right\} $ and
noting that $s_{0}=\lambda$, it holds for all $\lambda>\lambda_{0}^{\prime}$
that
\begin{align*}
\mathbb{E}\left\langle \nu_{t}^{{\scriptscriptstyle N}},x\mathbf{1}_{[\lambda,\infty)}(x)\right\rangle  & \leq\varepsilon e\int_{\lambda}^{\infty}xe^{-ax}dx=\varepsilon e\bigl(\frac{\lambda}{a}+\frac{1}{a^{2}}\bigr)e^{-a\lambda}.
\end{align*}
This proves (\ref{eq:LemmaA3_emp}) for $k=1$ and the work for
$k\geq2$ is analogous. The claim in (\ref{eq:LemmaA3_limit}) follows
similarly, by relying on the exponential tail property from Assumption
\ref{Assumption 2.3}.
\end{proof}
\begin{lem}
\label{lem:BOREL-CANTELLI-ARGUMENTS}Let $\nu$ be a limit point from
Theorem \ref{thm:EXISTENCE} and let $\tilde{\nu}$ be any measure valued
process satisfying Assumption \ref{Assumption 2.3}. Then it holds
with probability 1 that (as $\lambda\rightarrow\infty$),
\[
\lim_{\varepsilon\downarrow0}\frac{\nu_{t}(0,\varepsilon)}{\varepsilon}=0,\quad\nu_{t}(\lambda,\infty)=O(e^{-\lambda}),\quad\text{and}\quad\int_{0}^{T}\!\tilde{\nu}_{t}(\lambda,\infty)dt=O(e^{-\lambda}).
\]
\end{lem}

\begin{proof}
For the first claim, we recall from Proposition \ref{prop:limit_regularity}
that there exists $\delta\in(0,1]$ and $\beta>0$ such that $\mathbb{E}\nu_{t}(0,\varepsilon)=t^{-\frac{\delta}{2}}O(\varepsilon^{1+\beta})$
as $\varepsilon\rightarrow0$. Using Markov's inequality, we thus
deduce that, for any $\theta>0$ and $n$ sufficiently large,
\[
\mathbb{P}\bigl(n^{2/\beta}\nu_{t}(0,n^{-2/\beta})>\theta\bigr)\leq\theta^{-1}n^{2/\beta}\mathbb{E}\nu_{t}(0,n^{-2/\beta})\leq Ct^{-\frac{\delta}{2}}\theta^{-1}n^{-2}.
\]
Hence the Borel\textendash Cantelli lemma gives $\limsup_{n}n^{2/\beta}\nu_{t}(0,n^{-2/\beta})=0$
with probability~1. Now, given $\varepsilon>0$, we have $(n+1)^{-2/\beta}<\varepsilon\leq n^{-2/\beta}$
for some $n\geq1$, so we deduce that
\[
\limsup_{\varepsilon\downarrow0}\frac{\nu_{t}(0,\varepsilon)}{\varepsilon}\leq\limsup_{n\geq1}\frac{\nu_{t}(0,n^{-2/\beta})}{(n+1)^{-2/\beta}}\leq\limsup_{n\geq1}\frac{\nu_{t}(0,n^{-2/\beta})}{n^{-2/\beta}}\Bigl(\frac{n+1}{n}\Bigr)^{\frac{2}{\beta}}.
\]
Since the latter is zero with probability 1, this proves the first
claim. The two remaining results follow by analogous considerations
for the tail probabilities, using the exponential tail properties
from Proposition \ref{prop:limit_regularity} and Assumption \ref{Assumption 2.3},
respectively.
\end{proof}

\subsection{Proofs of Propositions \ref{prop: 1_loss increments} and \ref{prop: 2_loss increments}}

\begin{lem}
\label{lem:sup-Tails for Loss increments}Let $\Lambda_{t}^{i,{\scriptscriptstyle N}}=|X_{t}^{i}|+{\textstyle \sum_{j=1}^{{\scriptscriptstyle N}}}a_{j}^{{\scriptscriptstyle N}}|X_{t}^{j}|$,
 as in (\ref{eq: dominating process}). Then it holds uniformly
in $N\geq1$ that
\[
\mathbb{P}\Bigl(\,\sup_{t\leq T}\Lambda_{t}^{i,{\scriptscriptstyle N}}\geq\lambda\Bigr)=o(1)\quad\text{as}\quad\lambda\rightarrow\infty.
\]
\end{lem}

\begin{proof}
This is an immediate consequence of Corollary \ref{cor: Sup is subgaussian}.
\end{proof}
Based on this lemma, we can adapt the arguments from Section 4 of \cite{HL2016}
to prove Proposition \ref{prop: 1_loss increments} and Proposition
\ref{prop: 2_loss increments}.
\begin{proof}
[Proof of Proposition \ref{prop: 1_loss increments}]Arguing as in
the proof of Proposition 4.6 in \cite{HL2016}, for all $a>0$ and
$\theta:=\frac{1}{2}(1-r)$, we have
\[
\mathbb{P}(L_{t+h}^{{\scriptscriptstyle N}}-L_{t}^{{\scriptscriptstyle N}}<\delta,L_{t}^{{\scriptscriptstyle N}}<r)\leq\mathbb{P}(L_{t+h}^{{\scriptscriptstyle N}}-L_{t}^{{\scriptscriptstyle N}}<\delta,\;\nu_{t}^{{\scriptscriptstyle N}}(0,a)>\theta)+o(e^{-a}).
\]
Let $E:=\left\{ L_{t+h}^{{\scriptscriptstyle N}}-L_{t}^{{\scriptscriptstyle N}}<\delta,\;\nu_{t}^{{\scriptscriptstyle N}}(0,a)>\theta\right\}$
and define a random index set $\mathcal{I}$ by
\[
\mathcal{I}:=\bigl\{1\leq i\leq N:X_{t}^{i}<a,\;t<\tau_{i},\;a_{i}^{{\scriptscriptstyle N}}>\theta/2N\bigr\}.
\]
Note that the particles with index in $\left\{ 1\leq i\leq N:a_{i}^{{\scriptscriptstyle N}}\leq{\textstyle \frac{1}{2}}\theta N^{-1}\right\} $
can contribute at most ${\textstyle \frac{1}{2}}\theta$ towards $\nu_{t}^{{\scriptscriptstyle N}}(0,a)$.
Recalling also that $a_{i}^{{\scriptscriptstyle N}}\leq m/N$ for
some $m>0$, see (\ref{eq: a_i coefficients}), it follows that on the event $E$ we must have $\left|\mathcal{I}\right|>N\theta/2m$.
Hence
\begin{equation}
\mathbb{P}(E)\leq\sum_{\mathcal{I}_{0}:\left|\mathcal{I}_{0}\right|>\frac{\theta N}{2m}}\mathbb{P}(E\mid\mathcal{I}=\mathcal{I}_{0})\mathbb{P}(\mathcal{I}=\mathcal{I}_{0}).
\label{P(E)}
\end{equation}
Moreover, since $L_{t+h}^{{\scriptscriptstyle N}}-L_{t}^{{\scriptscriptstyle N}}<\delta$
on $E$ while $a_{i}^{{\scriptscriptstyle N}}>\theta/2N$ for $i\in\mathcal{I}$,
we deduce that
\begin{equation}
\mathbb{P}(E\mid\mathcal{I}=\mathcal{I}_{0})\leq\mathbb{P}\Bigl(\#\bigl\{ i\in\mathcal{I}_{0}:\inf_{u\leq h}X_{t+u}^{i}\leq0\bigr\}<2\delta N/\theta \mid \mathcal{I}=\mathcal{I}_{0}\Bigr).\label{eq: (4.6) in Hambly-Ledger}
\end{equation}
In order to estimate (\ref{eq: (4.6) in Hambly-Ledger}), we let $Z_{t}^{i}=\Upsilon_{t}(X_{t}^{i})$
as in Lemma \ref{lem:=00005BRestatement-of-Lemma4.1 from hambly ledger}
and recall that $Z_{t}^{i}$ then satisfies
\begin{equation}
dZ_{t}^{i}=\hat{b}_{t}^{i}dt+dB_{t}^{i},\quad|\hat{b}_{t}^{i}|\leq c_{1}(1+\Lambda_{t}^{i,{\scriptscriptstyle N}}),\quad |Z_{t}^{i}|\leq c(1+|X_{t}^{i}|).
\label{eq: Z_t and drift bound}
\end{equation}
Using that $Z_{t}^{i}\leq0$ if and only if $X_{t}^{i}\leq0$, the estimate
in (\ref{eq: (4.6) in Hambly-Ledger}) implies
\[
\mathbb{P}(E \mid \mathcal{I}_{0}=\mathcal{I})\leq\mathbb{P}\Bigl(\#\bigl\{ i\in\mathcal{I}:\inf_{u\leq h}Z_{t+u}^{i}\leq0\bigr\}<2\delta N/\theta \mid\mathcal{I}_{0}=\mathcal{I}\Bigr).
\]
Moreover, it is immediate from (\ref{eq: Z_t and drift bound}) that
\[
Z_{t+u}^{i}\leq Z_{t}^{i}+c_{1}h+\sup_{u\leq h}\Lambda_{t+u}^{i,{\scriptscriptstyle N}}c_{1}h+\tilde{B}_{u}^{i},\quad\tilde{B}_{u}^{i}:=B_{t+u}^{i}-B_{t}^{i},\quad\tilde{\Lambda}_{h}^{i,{\scriptscriptstyle N}}:=\sup_{u\leq h}\Lambda_{t+u}^{i,{\scriptscriptstyle N}}.
\]
In particular, if $\inf_{u\leq h}\tilde{B}_{u}^{i}\leq-Z_{t}^{i}-c_{1}h(1+\tilde{\Lambda}_{h}^{i,{\scriptscriptstyle N}})$,
then $\inf_{u\leq h}Z_{t+u}^{i}\leq0$, and hence
\[
\mathbb{P}(E \mid\mathcal{I}_{0}=\mathcal{I})\leq\mathbb{P}\Bigl(\#\bigl\{ i\in\mathcal{I}:\inf_{u\leq h}\tilde{B}_{u}^{i}\leq-Z_{t}^{i}-c_{1}h(1+\tilde{\Lambda}_{h}^{i,{\scriptscriptstyle N}})\bigr\}\!<\! {\textstyle \frac{2}{\theta}}\delta N \mid\mathcal{I}_{0}=\mathcal{I}\Bigr).
\]
Using the bound $|Z_{t}^{i}| \leq c |X_{t}^{i}|$
from (\ref{eq: Z_t and drift bound}), the definition of $\mathcal{I}$
implies that there exists $c_{2}>0$ such that
\[
\mathbb{P}(E \mid \mathcal{I}_{0}=\mathcal{I})\leq\mathbb{P}\Bigl(\#\bigl\{ i\in\mathcal{I}:\inf_{u\leq h}\tilde{B}_{u}^{i}\leq-c_{2}a-c_{2}(1+\tilde{\Lambda}_{h}^{i,{\scriptscriptstyle N}})\bigr\}\!<\! {\textstyle \frac{2}{\theta}}\delta N/ \mid \mathcal{I}_{0}=\mathcal{I}\Bigr).
\]
Recalling the fact that
\[
\tilde{B}_{u}^{i}=\int_{t}^{t+u}\rho_{s}dW_{s}^{0}+\int_{t}^{t+u}\sqrt{1-\rho_{s}^{2}}dW_{s}^{i}=:I_{u}+J_{u}^{i},
\]
we split the above probability on the event $\{\sup_{u\leq h}|I_{u}|<c_{2}a\}\cap\{\sup_{u\leq h}\Lambda_{t+u}^{i,{\scriptscriptstyle N}}<a\}$
and its complement. In this way, we get
\begin{align*}
\mathbb{P}(E & \mid\mathcal{I}_{0}=\mathcal{I})\leq\mathbb{P}\Bigl(\#\bigl\{ i\in\mathcal{I}:\inf_{u\leq h}J_{u}^{i}\leq-3c_{2}a-c_{2}\bigr\}< {\textstyle \frac{2}{\theta}} \delta N \mid\mathcal{I}_{0}=\mathcal{I}\Bigr)\\
 & \qquad\qquad\qquad\qquad+\mathbb{P}\Bigl(\,\sup_{u\leq h}\left|I_{u}\right|\geq c_{2}a\Bigr)+\mathbb{P}\Bigl(\,\sup_{u\leq h}\Lambda_{t+u}^{i,{\scriptscriptstyle N}}\geq a\Bigr).
\end{align*}
From Lemma \ref{lem:sup-Tails for Loss increments} we know that the
last term is $o(1)$ as $a\rightarrow\infty$ uniformly in $N\geq1$.
Similarly, $I_{u}$ is a martingale, and hence the second term is
also $o(1)$ as $a\rightarrow\infty$ by Doob's Maximal Inequality.
Concerning the first term, we can introduce a time-change to make each $J^i$ an independent Brownian motion. Recalling (\ref{P(E)}), the proof can then be completed via a law of large numbers argument as in Proposition
4.6 of \cite{HL2016}, by carefully choosing the free parameter $a$ as a function of $\delta$. 
\end{proof}
\begin{proof}
[Proof of Proposition \ref{prop: 2_loss increments}]Arguing as in
Proposition 4.7 of \cite{HL2016}, it suffices to show that
\[
\lim_{\delta\rightarrow0}\lim_{N\rightarrow\infty}\mathbb{P}\left(L_{t+\delta}^{{\scriptscriptstyle N}}-L_{t}^{{\scriptscriptstyle N}}\geq\eta,\;\nu_{t}^{{\scriptscriptstyle N}}(0,\varepsilon)<\eta/2\right)=0.
\]
Let $E:=\left\{ L_{t+\delta}^{{\scriptscriptstyle N}}-L_{t}^{{\scriptscriptstyle N}}\geq\eta,\;\nu_{t}^{{\scriptscriptstyle N}}(0,\varepsilon)<\eta/2\right\} $
and define the random index set
\[
\mathcal{I}:=\left\{ 1\leq i\leq N:X_{t}^{i}\geq\varepsilon\;\;\text{or}\;\;t\geq\tau_{i}\right\} .
\]
By (\ref{eq: a_i coefficients}), there exists $m$ such that $a_{i}^{{\scriptscriptstyle N}}\leq m/N$,
so noting that $\left\{ \nu_{t}^{{\scriptscriptstyle N}}(0,\varepsilon)<\eta/2\right\} $
is contained in $\left\{ \left|\mathcal{I}\right|\geq\frac{N}{m}(1-\frac{\eta}{2})\right\} $,
we have
\[
\mathbb{P}\left(E\right)\leq\sum_{\mathcal{I}_{0}:\left|\mathcal{I}_{0}\right|\geq\frac{N}{m}(1-\frac{\eta}{2})}\mathbb{P}\left(E\mid\mathcal{I}=\mathcal{I}_{0}\right)\mathbb{P}(\mathcal{I}=\mathcal{I}_{0}).
\]
The conditional probabilities can be estimated by
\begin{align}
\mathbb{P}\left(E\mid\mathcal{I}=\mathcal{I}_{0}\right) & \leq\mathbb{P}\Bigl(\#\bigl\{ i\in\mathcal{I}_{0}:\inf_{s\in[t,t+\delta]}X_{s}^{i}\leq0,\;\;X_{t}^{i}\geq\varepsilon\bigr\}\geq\frac{N\eta}{2m}\mid\mathcal{I}=\mathcal{I}_{0}\Bigr)\nonumber \\
 & \leq\mathbb{P}\Bigl(\#\bigl\{ i\in\mathcal{I}_{0}:\inf_{s\in[t,t+\delta]}(X_{s}^{i}-X_{t}^{i})\leq-\varepsilon\bigr\}\geq\frac{N\eta}{2m}\mid\mathcal{I}=\mathcal{I}_{0}\Bigr)\label{eq: Prop4.7 first estimate}
\end{align}
Using the scale transform $\Upsilon$ from Lemma \ref{lem:=00005BRestatement-of-Lemma4.1 from hambly ledger},
we introduce $U_{s}^{i}:=\Upsilon_{t+s}(X_{t+s}^{i}-X_{t}^{i})$ and note that, as in Lemma \ref{lem:=00005BRestatement-of-Lemma4.1 from hambly ledger},
\begin{align}
dU_{s}^{i}= & u_{s}^{i}ds+\rho_{t+s}dW_{t+s}^{0}+(1-\rho_{t+s}^{2})^{\frac{1}{2}}dW_{t+s}^{i}=:u_{s}^{i}ds+dI_{s}+dJ_{s}^{i}\label{eq: decomposition of U}
\end{align}
where the drift satisfies $|u_{s}^{i}|\leq c_{1}(1+\Lambda_{t+s}^{i,{\scriptscriptstyle N}})$.
By construction of $\Upsilon$ and the boundedness from below of $1/\sigma$,
say $1/\sigma\geq c_{2}>0$, it follows from (\ref{eq: Prop4.7 first estimate})
that
\[
\mathbb{P}\left(E\mid\mathcal{I}=\mathcal{I}_{0}\right)\leq\mathbb{P}\Bigl(\#\bigl\{ i\in\mathcal{I}_{0}:\inf_{s\leq\delta}U_{s}^{i}\leq-c_{2}\varepsilon\bigr\}\geq\frac{N\eta}{2m}\mid\mathcal{I}=\mathcal{I}_{0}\Bigr).
\]

Given $\delta>0$, fix a constant $a=a(\delta)>0$ to be specified
later. Using the decomposition (\ref{eq: decomposition of U}) of
$U_{s}^{i}$ and the growth estimate for its drift $u_{s}^{i}$, we
see that, on the event
\begin{equation}
\Bigl\{\,\sup_{s\leq\delta}\Lambda_{t+s}^{i,{\scriptscriptstyle N}}<a/\delta\Bigr\}\cap\Bigl\{\,\sup_{s\leq\delta\text{}}|I_{s}|<a\Bigr\},\label{eq: event for splitting up probability}
\end{equation}
if $i\in\mathcal{I}_{0}$ is such that $\inf_{s\leq\delta}U_{s}^{i}\leq-c_{2}\varepsilon$,
then
\begin{align*}
\inf_{s\leq\delta}J_{s}^{i} & \leq-c_{2}\varepsilon+\delta c_{1}(1+\sup_{s\leq\delta}\Lambda_{t+s}^{i,{\scriptscriptstyle N}})+\sup_{s\leq\delta}\left|I_{s}\right|\\
 & \leq-c_{2}\varepsilon+c_{1}\delta+(1+c_{1})a\:\leq-c_{3}(\varepsilon-\delta-a)
\end{align*}
Consequently, splitting up the desired probability on the event (\ref{eq: event for splitting up probability})
and its complement,
\begin{align*}
\mathbb{P}\left(E\mid\mathcal{I}=\mathcal{I}_{0}\right) & \leq\mathbb{P}\Bigl(\#\bigl\{ i\in\mathcal{I}_{0}:\inf_{s\leq\delta}J_{s}^{i}\leq-c_{3}(\varepsilon-\delta-a)\bigr\}\geq\frac{N\eta}{2m}\mid\mathcal{I}=\mathcal{I}_{0}\Bigr)\\
 & \qquad\qquad\qquad+\mathbb{P}\Bigl(\,\sup_{s\leq\delta\text{}}\Lambda_{t+s}^{i,{\scriptscriptstyle N}}\geq a/\delta\Bigr)+\mathbb{P}\Bigl(\,\sup_{s\leq\delta}\left|I_{s}\right|\geq a\Bigr).
\end{align*}
By Doob's maximal inequality, the last term is bounded by $\delta a^{-2}$, so choosing $a=a(\delta)$ s.t. $\delta a^{-2}\rightarrow0$ as
$\delta\rightarrow0$, we get
\[
\mathbb{P}\Bigl(\,\sup_{s\leq\delta}\left|I_{s}\right|\geq a\Bigr)=o(1)\quad\text{as}\quad\delta\rightarrow0.
\]
Moreover, ensuring also that $\delta/a\rightarrow0$ as $\delta\rightarrow0$,
it follows from Lemma \ref{lem:sup-Tails for Loss increments} that
\[
\mathbb{P}\Bigl(\,\sup_{s\leq\delta}\Lambda_{t+s}^{i,{\scriptscriptstyle N}}\geq a/\delta\Bigr)=o(1)\quad\text{as}\quad\delta\rightarrow0.
\]
The above requirements are satisfied by $a=a(\delta):=\delta^{1/2}\log\log(1/\delta)$. With this choice for $a$, the proof can now be finished by the same
arguments as in Proposition 4.7 of \cite{HL2016}.
\end{proof}

\end{document}